\documentclass[11pt,leqno]{article}
\usepackage{amsmath,amssymb,mathrsfs,xy,amsthm,bm,url}
\input{xy}
\xyoption{all}

\newcommand{\mr}[1]{\mathrm{#1}}
\newcommand{\ms}[1]{\mathscr{#1}}
\newcommand{\mc}[1]{\mathcal{#1}}
\newcommand{\mb}[1]{\mathbb{#1}}
\newcommand{\mf}[1]{\mathfrak{#1}}

\newcommand{\Ddag}[1]{\mathscr{D}^\dag_{\mathscr{#1}}}
\newcommand{\DdagQ}[1]{\mathscr{D}^\dag_{{#1},\mathbb{Q}}}
\newcommand{\Dcomp}[2]{\widehat{\mathscr{D}}^{(#1)}_{#2}}
\newcommand{\DcompQ}[2]{\widehat{\mathscr{D}}^{(#1)}_{#2,\mathbb{Q}}}
\newcommand{\Dmod}[2]{{\mathscr{D}}^{(#1)}_{#2}}

\newcommand{\Dtild}[2]{\widetilde{\mathscr{D}}^{(#1)}_{#2}}
\newcommand{\DtildQ}[2]{\widetilde{\mathscr{D}}^{(#1)}_{#2,\mathbb{Q}}}
\newcommand{\DtdagQ}[1]{{\widetilde{\mathscr{D}}}^{\dag}_{#1,\mathbb{Q}}}
\newcommand{\Dtmod}[2]{{\widetilde{\mathscr{D}}}^{(#1)}_{#2}}

\newcommand{\indlim}{\mathop{\underrightarrow{\mathrm{lim}}}}

\newcommand{\Bcomp}[2]{\widehat{\mc{B}}^{(#1)}_{\mathscr{#2}}}
\newcommand{\Bmod}[2]{{\mc{B}}^{(#1)}_{#2}}
\newcommand{\BcompQ}[2]{\widehat{\mc{B}}^{(#1)}_{\mathscr{#2},\mathbb{Q}}}

\newcommand{\angles}[2]{\langle{#2}\rangle_{(#1)}}

\newcommand{\Dmqc}{D^-_{\mr{qc}}}
\newcommand{\DbqcQ}{D^b_{\mb{Q},\mr{qc}}}

\makeatletter

 \def\@seccntformat#1{\csname the#1\endcsname.\hspace{2ex}}

 \newcommand{\nsubsection}%
  {\@startsection{subsection}%
  {2}%
  {\z@}%
  {-3.5ex plus -1ex minus -.2ex}%
  {-0ex}%
  {\reset@font\normalsize\bfseries}}%

 \newcommand{\nnsubsubsection}%
  {\@startsection{subsubsection}%
  {3}%
  {\z@}%
  {-3.5ex plus -1ex minus -.2ex}%
  {-2ex}%
  {\reset@font\normalsize\bfseries}}%

  \renewcommand{\subsubsection}%
  {\@startsection{subsubsection}%
  {3}%
  {\z@}%
  {-3.5ex plus -1ex minus -.2ex}%
  {0ex}
  {\reset@font\normalsize\bfseries}}%

  \renewcommand{\subsection}%
  {\@startsection{subsection}%
  {2}%
  {\z@}%
  {-3.5ex plus -1ex minus -.2ex}%
  {0ex}
  {\reset@font\normalsize\bfseries}}%

 \newcommand{\nnsubsection}%
  {\@startsection{subsection}%
  {2}%
  {\z@}%
  {-3ex}%
  {1ex}%
  {\reset@font\normalsize\bfseries}}%

 \newcommand{\usubsection}%
  {\@startsection{subsection}%
  {2}%
  {\z@}%
  {-3.5ex plus -1ex minus -.2ex}%
  {0.5ex}
  {\reset@font\normalsize\bfseries}}%

 \@addtoreset{subsection}{section}
 \@addtoreset{equation}{subsection}

 \newcommand{\nsection}{\@startsection{section}{1}{\z@}%
     {-5ex}
     {1ex}
     {\reset@font\center\large\sc}}

 \renewenvironment{thebibliography}[1]
 {\nsection*{\refname\@mkboth{\refname}{\refname}}%
   \list{\@biblabel{\@arabic\c@enumiv}}%
        {\settowidth
	\labelwidth{\@biblabel{#1}}%
         \leftmargin
	 \labelwidth
         \advance
	 \leftmargin
	 \labelsep
         \@openbib@code
         \usecounter{enumiv}%
         \let\p@enumiv\@empty
	 \parskip=0pt
	 \itemsep=1pt
	 \parsep=1pt
	 \itemindent=\z@
         \renewcommand\theenumiv{\@arabic\c@enumiv}}%
   \sloppy
   \clubpenalty4000
   \@clubpenalty\clubpenalty
   \widowpenalty4000%
   \footnotesize
   \sfcode`\.\@m}
  {\def\@noitemerr
    {\@latex@warning{Empty `thebibliography' environment}}%
   \endlist}
\makeatother

\newtheoremstyle{thm}
 {1em}
 {3pt}
 {\itshape}
 {}
 {\bf}
 {. ---}
 {0.5em}
 {}
\newtheoremstyle{dfn}
 {1em}
 {3pt}
 {}
 {}
 {\bf}
 {. {---}}
 {0.5em}
 {}

\swapnumbers
\theoremstyle{thm}
\newtheorem{thm}[subsection]{Theorem}
\newtheorem{lem}[subsection]{Lemma}
\newtheorem*{lem*}{Lemma}
\newtheorem{cor}[subsection]{Corollary}
\newtheorem*{cor*}{Corollary}
\newtheorem{prop}[subsection]{Proposition}
\newtheorem*{prop*}{Proposition}

\newtheorem*{thm*}{Theorem}
\newtheorem*{thmM}{Theorem \ref{smoothpoincare}}

\theoremstyle{dfn}

\newtheorem{dfn*}[subsubsection]{Definition}

\newtheorem{rem}[subsection]{Remark}
\newtheorem*{rem*}{Remark}

\newcommand{\shom}{\mathop{\mc{H}om}\nolimits}
\newcommand{\sext}{\mathop{\mc{E}xt}\nolimits}

\usepackage{color}
\newenvironment{meta}{
\noindent \color{red}
\sffamily[}{\upshape]}

\newcommand{\LD}[1]{\underrightarrow{LD}^b_{\mathbb{Q},\mathrm{qc}}
(\widehat{\mathscr{D}}^{(\bullet)}_{#1})}
\newcommand{\LDn}[1]{\underrightarrow{LD}^b_{\mathbb{Q},\mathrm{qc}}
(\widehat{\mathscr{D}}^{(\bullet)}_{#1}}

\newcommand{\wedgeb}{\mathop{\bigwedge}\nolimits}

\newsavebox{\circlebox}
\savebox{\circlebox}{\fontencoding{OMS}\selectfont\char13}
\newlength{\circleboxwdht}
\newcommand{\ccirc}[1]{
  \setlength{\circleboxwdht}{\wd\circlebox}
  \addtolength{\circleboxwdht}{\dp\circlebox}
  \raisebox{0.2\dp\circlebox}{
    \parbox[][\circleboxwdht][c]{\wd\circlebox}
    {\centering\scriptsize #1}}
  \llap{\usebox{\circlebox}}
}

\newcommand{\btimesdag}{\mathop{\boxtimes}\limits^{\mb{L}}}
\newcommand{\otimesdag}{\mathop{\otimes}\limits^{\mb{L}}}
\newcommand{\Gammas}{\underline{\Gamma}}
\newcommand{\otimesddag}{\otimesdag{}^\dag}
\newcommand{\otimesddagt}{\widetilde{\otimes}^{\mb{L}}{}^\dag}
\newcommand{\pa}{\mbox{parf-amp}}
\newcommand{\tdp}{\widetilde{D}^b_\mr{perf}}

\newcommand{\cindex}[2]{\genfrac{}{}{0pt}{}{#1}{#2}}
\newcommand{\bD}{\mathbf{D}}
\newcommand{\bR}{\mathbf{R}}
\newcommand{\bL}{\mathbf{L}}
\newcommand{\wtilde}{\widetilde}

\newtheorem{propH}[subsubsection]{Proposition}
\newtheorem{defi}[subsubsection]{Definition}

\newtheorem{lemH}[subsubsection]{Lemma}
\newtheorem{souslem}[paragraph]{Lemma}
\newtheorem{thmH}[subsubsection]{Theorem}

\newcommand{\inft}{{\scriptstyle\infty}}

\newcommand{\what}{\widehat}
\newcommand{\ot}{\otimes}

\newcommand{\rig}{\rightarrow}

\newcommand{\sta}{\stackrel}

\newcommand{\lrig}{\leftarrow}

\newcommand{\pg}{_{\bullet}}

\newcommand{\der}{\partial}
\newcommand{\la}{\langle}
\newcommand{\ra}{\rangle}

\newcommand{\Qr}{{\bf Q}}
\newcommand{\Ne}{{\bf N}}
\newcommand{\Ze}{{\bf Z}}

\newcommand{\varep}{\varepsilon}

\newcommand{\Ga}{\Gamma}

\newcommand{\BB}{{\mathcal B}}

\newcommand{\DD}{{\mathcal D}}
\newcommand{\EE}{{\mathcal E}}
\newcommand{\FF}{{\mathcal F}}

\newcommand{\HH}{{\mathcal H}}

\newcommand{\OO}{{\mathcal O}}

\renewcommand{\SS}{{\mathcal S}}
\newcommand{\TT}{{\mathcal T}}
\newcommand{\UU}{{\mathcal U}}
\newcommand{\VV}{{\mathcal V}}

\newcommand{\XX}{{\mathcal X}}
\newcommand{\YY}{{\mathcal Y}}
\newcommand{\ZZ}{{\mathcal Z}}

\newcommand{\uk}{\underline{k}}
\newcommand{\ul}{\underline{l}}

\newcommand{\uder}{\underline{\der}}

\newcommand{\num}{\nu_m}

\newcommand{\ank}{A_N(K)^{\dagger}}

\newcommand{\Bcm}{\what{\BB}^{(m)}}

\newcommand{\Dcm}{\what{\DD}^{(m)}}

\newcommand{\spf}{{\rm Spf}\,}

\begin{document}
\title{Explicit calculation of Frobenius isomorphisms and Poincar\'{e}
duality in the theory of arithmetic $\ms{D}$-modules}
\author{Tomoyuki Abe}
\date{}
\maketitle
\begin{abstract}
 The aim of this paper is to compute the Frobenius structures of some
 cohomological operators of arithmetic $\ms{D}$-modules. To do this, we
 calculate explicitly an isomorphism between canonical sheaves defined
 abstractly. Using this calculation, we establish the relative
 Poincar\'{e} duality in the style of SGA4. As another application, we
 compare the push-forward as arithmetic $\ms{D}$-modules and the rigid
 cohomologies taking Frobenius into account. These theorems will lead us
 to an analog of ``Weil II'' and a product formula for $p$-adic epsilon
 factors.
\end{abstract}

\section*{Introduction}
In this paper, we prove several results concerning Frobenius structures
in the theory of arithmetic $\ms{D}$-modules. There are mainly three
goals in this paper.
\begin{itemize}
 \item[(G1)] Compute and describe Frobenius structures of some
	     cohomological operators, appearing in \cite{Ber2},
	     concretely in terms of differential operators.

 \item[(G2)] Establish a relative Poincar\'{e} duality in the style of
	     SGA4 in the theory of arithmetic $\ms{D}$-modules.

 \item[(G3)] Compare the push-forwards in the theory of $\ms{D}$-modules
	     and the rigid cohomologies {\em with Frobenius
	     structure}.
\end{itemize}

First, (G1) is the starting point of other two goals. We describe some
isomorphisms appearing in \cite{Ber2} explicitly by taking local
coordinates. Apart from (G2) and (G3), this calculation is used in
\cite{AM} to compute the geometric Fourier transform defined by
C. Noot-Huyghe explicitly. With this description, we are able to
re-prove Gross-Koblitz formula \cite{GK} using arithmetic
$\ms{D}$-modules. This calculation will be discussed in other places. We
expect that these ideas can be generalized to a calculation of $p$-adic
$\varepsilon$-factors.

For (G2), a duality theory was established by A. Virrion in
\cite{Vir2} to some extent. However, we need two more ingredients to
call it the Poincar\'{e} duality:
1) comparison of the extraordinary pull-back and the normal
pull-back for a smooth morphism, and 2) taking Frobenius structures into
account. Using our result 1) on the comparison of two types of
pull-backs, we are also able to compare duality functors of the
theory of rigid cohomology and that of arithmetic $\ms{D}$-modules,
which completes a work in \cite{Caro}.
For 2), even without Frobenius structures, her duality is very
powerful tool, but in practical uses of arithmetic $\ms{D}$-module
theory, Frobenius structure is another important ingredient that
contain arithmetic information. For example, $L$-functions for holonomic
$\ms{D}^\dag$-modules cannot be defined without Frobenius structures,
and thus to show the functional equation for $L$-functions, it is
necessary to consider Frobenius structures in the duality.

(G3) is another application of (G1). If we do not consider Frobenius
structures, this is a well-known result of Berthelot
\cite[4.3.6.3]{BerInt}. This type of comparison theorem is necessary
when we want to exploit
results of the theory of $\ms{D}$-modules in the theory of rigid
cohomologies and {\it vice versa}. For example, in \cite[3.3]{CarL},
the author discussed the relations of $L$-functions defined using the
theory of rigid cohomologies and that of arithmetic
$\ms{D}$-modules. This result can be reinforced and re-stated much
clearer using our result (cf.\ Remark \ref{compLcarLE}).
\bigskip

Now, let us go into more details of the results. 
Let $R$ be a complete discrete valuation ring of mixed characteristic
$(0,p)$, and we denote by $k$ its residue field which is assumed to be
perfect, $K$ its field of fractions. Let $\ms{X}$ be a smooth formal
scheme over $\mr{Spf}(R)$, and $X_0$ be the reduction of $\ms{X}$ over
$k$. Let $s$ be a positive integer, and we put
$X_0':=X_0\otimes_{k,F^s_k}k$ where
$F^s_k\colon\mr{Spec}(k)\rightarrow\mr{Spec}(k)$ denotes the $s$-th
absolute Frobenius isomorphism.
In this introduction, we also assume that there exist liftings
$\sigma\colon\mr{Spf}(R)\xrightarrow{\sim}\mr{Spf}(R)$ of $F^s_k$ and
$F\colon\ms{X}\rightarrow\ms{X}':=\ms{X}\otimes_{R,\sigma}R$ of the
relative Frobenius morphism $F^s_{X_0/k}\colon X_0\rightarrow X'_0$ for
simplicity. A coherent $F$-$\DdagQ{\ms{X}}$-module is a couple of a
coherent $\DdagQ{\ms{X}}$-module $\ms{M}$ and an isomorphism
$F^*\ms{M}^{\sigma}\xrightarrow{\sim}\ms{M}$ where $\ms{M}^\sigma$
denotes the $\DdagQ{\ms{X}'}$-module induced by $\ms{M}$ by the base
change $\sigma$. Frobenius structures are known to be stable under
reasonable cohomological operations of arithmetic $\ms{D}$-modules such
as push-forwards, extraordinary pull-backs, tensor products, {\it etc}
(cf.\ \cite{Ber2}).

When we try to calculate Frobenius structures of some cohomological
operations ({\it e.g.\ }push-forward functor), an obstacle lies in the
isomorphism $\omega_{\ms{X}}\xrightarrow{\sim}F^\flat\omega_{\ms{X}'}$
of \cite[2.4.2]{Ber2}. The construction of this isomorphism is formal
using general facts of \cite{Har}. However, we need to trace many
isomorphisms of \cite{Har} to compute it explicitly, which is
monotonous but messy. The advantage of this computation is
that it makes us possible to calculate Frobenius structures in
``brutal'' but very direct ways, at least locally. As an example
of the explicit computation, we calculate the Frobenius structure of
push-forwards (cf.\ paragraph \ref{calcpushexp}). We can also prove a
proper base change type lemma (cf.\ Lemma \ref{bcthmsch}). With an aid
of a result of Caro, we get the proper base change theorem in paragraph
\ref{properbc}.
Another application will be to prove the following theorem.
\begin{thmM}
 Let $f\colon\ms{X}\rightarrow\ms{Y}$ be a smooth morphism of relative
 dimension $d$ between smooth formal schemes. For a
 coherent $F$-$\DdagQ{\ms{Y}}$-module $\ms{M}$,
 \begin{equation}
  \label{mainresisom}
   \tag{b}
   f^!(\mb{D}_{\ms{Y}}(\ms{M}))\cong\mb{D}_{\ms{X}}(f^!\ms{M})(d)[2d]
 \end{equation}
 where $(d)$ denotes the $d$-th Tate twist (cf.\ paragraph
 {\normalfont\ref{dfntatetwist}}).
\end{thmM}
The construction of the isomorphism without Frobenius structures
requires only standard methods of the theory of arithmetic
$\ms{D}$-modules, but to see the compatibility with Frobenius
structures, we need the explicit calculation of the isomorphism of
canonical sheaves. Using a result we get on the way we prove this
theorem, we compare the rigid cohomologies and the push-forwards in the
arithmetic $\ms{D}$-module theory, which is (G3). This theorem can be
seen as a part of Poincar\'{e} duality. See the last section for an
account of this interpretation. Moreover this theorem leads us
to complete a work of Caro in \cite{CarD} (cf.\ Corollary
\ref{compduals}) comparing the duality functors in the theory of rigid
cohomology and that of arithmetic $\ms{D}$-modules.

In this paper, we also include some small but useful results
concerning Frobenius pull-backs. Namely, we prove: 1.\ commutation of
the dual functor and the tensor product in some cases, 2.\ the
K\"{u}nneth formula, 3.\ compatibility of the relative duality
homomorphism with Frobenius. The result 1 uses (\ref{mainresisom}) in
the proof, but results 2 and 3 are independent of the explicit
computations. The results 1 and 2 are included in this paper with the
intention of use in \cite{AM}. The result 3 is aimed to establish
the Poincar\'{e} duality as we have already mentioned.

Finally, let us point out some notable applications of our
result. Currently, we have the following two important applications:
\begin{itemize}
 \item establishing the ``yoga of weight'' in $p$-adic cohomologies,
       especially an analog of ``Weil II'' in the theory of arithmetic
       $\ms{D}$-modules. This will be treated in a paper of the author
       jointly with D. Caro (in preparation).

 \item a product formula for $p$-adic epsilon factors. See
       \cite{AM} for more details.
\end{itemize}
In the proofs of those two results, (G2) and (G3) are used extensively.
\bigskip

Let us see the structure of this paper. 
In \S 1, we describe isomorphisms which are key isomorphisms to
construct the commutativity. In \S 2, we calculate the Frobenius
structure of push-forwards explicitly. We should mention that this
calculation is a key to calculate the Frobenius structure of Fourier
transforms explicitly, which is carried out in \cite{AM}. With these two
sections, (G1) is attained. As an application, we show a proper base
change type lemma also in this section.
In \S 3, we show that the dual functor and the extraordinary
pull-back functor commute up to some degree shift and Tate twist in the
smooth case.  Using a lemma we prove to
show this commutativity, we will compare the rigid cohomology and the
push-forward of arithmetic $\ms{D}$-modules, and we get (G3).
In \S 4, we will show some complementary results, which are used in
\cite{AM}. The idea of the proof of the K\"{u}nneth formula is due to
P. Berthelot. In this section, we also prove that the relative duality
isomorphism of Virrion is compatible with Frobenius. Together with \S
3, (G2) is completed. In \S 5, we interpret the results in terms of the
philosophy of ``six functors'' by Grothendieck, which clarifies the
meaning of the results in this paper.

\nnsubsection*{Acknowledgments}
The author would like to thank Professor P. Berthelot for letting him
know the proof of the K\"{u}nneth formula. Most of the work of this
paper was done when the author was visiting to IRMA of {\it
Universit\'{e} de Strasbourg} in 2010. He would like
to thank A. Marmora and the institute for the hospitality.
He also like to express his gratitude to Professor A.  Shiho for
stimulating discussions. This work was supported by Grant-in-Aid for
JSPS Fellows 20-1070, and partially by {\it l'agence nationale de la
recherche} ANR-09-JCJC-0048-01.

\nnsubsection*{Notation}
\subsection{}
In this paper we fix a complete discrete valuation ring $R$ with mixed
characteristic $(0,p)$. We denote the residue field by $k$, the field of
fractions by $K$. For a non-negative integer $i$, we put $R_i$ to be
$R/\pi^{i+1}R$ where $\pi$ is a uniformizer. We denote by $e$ the
absolute ramification index of $K$.

In general, we use Roman fonts ({\it e.g.\ }$X$) for schemes and script
fonts ({\it e.g.\ }$\ms{X}$) for formal schemes. For a formal scheme
$\ms{X}$ over $\mr{Spf}(R)$, we usually denote by $X_i$ the reduction
$\ms{X}\otimes_RR_i$ over $\mr{Spec}(R_i)$.

\subsection{}
For a scheme $X$ over $\mr{Spec}(k)$, we denote by $F_X\colon
X\rightarrow X$ the absolute Frobenius homomorphism: it sends a section
$f$ of $\mc{O}_X$ to $f^p$. We fix a positive integer $s$, and put
$q:=p^s$. We put $X^{(s)}:=X\otimes_{k,F^{s*}_k}k$, and call it the
relative $s$-th Frobenius of $X$.

\subsection{}
\label{perfcatdfn}
Let $\ms{D}$ be a sheaf of rings on a topological space $X$. When we
simply say $\ms{D}$-module, it means left $\ms{D}$-module. We denote by
$D^*_{\mr{coh}}(\ms{D})$ ($*\in\{+,-,b\}$) the full subcategory of
$D^*(\ms{D})$ such that the objects consist of complexes whose
cohomology sheaves are coherent. We denote by $D_{\mr{perf}}(\ms{D})$
the full subcategory whose objects consist of perfect complexes ({\it
i.e.\ }complexes locally quasi-isomorphic to bounded complexes of
locally projective $\ms{D}$-modules). We denote by
$D_{\mr{ftd}}(\ms{D})$ the full subcategory consisting of finite
Tor-dimensional complexes ({\it i.e.\ }complexes possessing bounded flat
resolutions). We put $D^b_{\mr{perf}}(\ms{D}):=D_{\mr{perf}}(\ms{D})\cap
D_{\mr{ftd}}(\ms{D})$. When $X$ is quasi-compact and $\ms{D}$ is
coherent, $D^b_{\mr{perf}}(\ms{D})$ coincides with
$D_{\mr{perf}}(\ms{D})\cap D^b_{\mr{coh}}(\ms{D})$. See SGA6 Exp.\ I for
details.
When we denote by $D(\ms{D})^{\mr{g}}$ (resp.\ $D(\ms{D})^{\mr{d}}$) we
consider complexes of left (resp.\ right) $\ms{D}$-modules ($\mr{g}$ and
$\mr{d}$ stand for French words ``gauche'' and  ``droit''). When we put
$\mb{Q}$ as an index, this means tensor with $\mb{Q}$.

\subsection{}
In this paper, we freely use the language of arithmetic
$\ms{D}$-modules. For details see \cite{Ber1}, \cite{Ber2},
\cite{BerInt}. In particular, we use the rings $\Dmod{m}{X}$,
$\Dcomp{m}{\ms{X}}$, $\Ddag{\ms{X}}$ for a smooth scheme
$X$ and a smooth formal scheme $\ms{X}$. We use the category
$\LD{\ms{X}}$ whose definition is written in \cite[4.2]{BerInt}. Let $Z$
be a divisor of the special fiber of $\ms{X}$. Then by the same
construction, we can consider the category $\LDn{\ms{X}}(Z))$. For this
category, see also \cite[1.1.3]{CarL}.

\subsection{}
\label{abuselangocisoc}
Let $\ms{X}$ be a smooth formal scheme, and $Z$ be a divisor of its
special fiber. Let $\ms{U}:=\ms{X}\setminus Z$, $X$ and $U$ be the
special fibers of $\ms{X}$ and $\ms{U}$ respectively. Let
$\ms{M}$ be a coherent ($F$-)$\DdagQ{\ms{X}}(^\dag Z)$-module such that
it is coherent as an
$\mc{O}_{\ms{X},\mb{Q}}(^\dag Z)$-module. Let
$\mc{C}$ be the full subcategory of the category of coherent
($F$-)$\DdagQ{\ms{X}}(^\dag Z)$-modules consisting of such
$\ms{M}$. Then we know that the specialization functor induces an
equivalence between $\mc{C}$ and the category
($F$-)$\mr{Isoc}^\dag(U,X/K)$ by \cite[4.4.12]{Ber1} and
\cite[4.6.3, 4.6.7]{Ber2}. We say that $\ms{M}$ is a convergent
($F$-)isocrystal on $\ms{U}$ overconvergent along $Z$ by abuse of
language.

\section{Explicit calculation of isomorphisms of canonical sheaves}
\label{sect1}

In this section, we will explicitly calculate the isomorphisms of
\cite[2.4.3, 2.4.4]{Ber2} (cf.\ Theorem \ref{maincalc}), from which some
commutation results of Berthelot \cite{Ber2} are derived. The
existence of these isomorphisms are direct consequences of fundamental
properties of the functors of Hartshorne \cite{Har}, and for the
explicit calculations, we need to go back to the proofs of these
fundamental properties, and trace these isomorphisms step by step. We
follow the notation of \cite[III]{Har}.

\subsection{}
First we review the notations and functors of Hartshorne \cite[III]{Har}
in short. Let $f\colon X\rightarrow Y$ be a morphism of schemes. When
$f$ is smooth, we denote by $\omega_{X/Y}$ the canonical sheaf
$\wedgeb^d\Omega_{X/Y}$ where $d$ denotes the relative dimension of $f$.
When $f$ is regular closed immersion, let $\mc{J}$ be the
sheaf of ideals of $\mc{O}_Y$ defining $X$. Then we put
$\omega_{X/Y}:=(\wedgeb^d\mc{J}/\mc{J}^2)^\vee$ where $d$ is the
codimension of $X$ in $Y$, and ${}^\vee$ is the dual as an
$\mc{O}_X$-module. In both cases, $\omega_{X/Y}$ is a locally free
$\mc{O}_X$-module of rank $1$.

Suppose $f$ is smooth. We define a functor
$f^{\sharp}\colon D(\mc{O}_Y)\rightarrow D(\mc{O}_X)$ as follows. See
\cite[III, \S2]{Har} for more details. For $C\in D(\mc{O}_Y)$, we put
$f^{\sharp}(C):=f^*(C)\otimes_{\mc{O}_X}\omega_{X/Y}[d]$ where $d$ is
the relative dimension of $f$. We see that this functor takes
$D_{\mr{qc}}^b(\mc{O}_Y)$ into $D_{\mr{qc}}^b(\mc{O}_X)$ (here
$D^*_{\mr{qc}}$ denotes the full subcategory of the derived category
consisting of objects whose cohomologies are quasi-coherent sheaves).

In turn, suppose $f$ is a finite morphism. We denote by $\overline{f}$
the morphism of ringed spaces
$(X,\mc{O}_X)\rightarrow(Y,f_*\mc{O}_X)$. Then we define a functor
$f^{\flat}\colon D^+(\mc{O}_Y)\rightarrow D^+(\mc{O}_X)$ as follows. See
\cite[III, \S6]{Har} for more details. For $C\in D^+(\mc{O}_Y)$, we put
$f^{\flat}(C):=\overline{f}^*R\shom_{\mc{O}_Y}(f_*\mc{O}_X,C)$. We know
that this functor takes
$D^+_{\mr{qc}}(\mc{O}_Y)$ into $D^+_{\mr{qc}}(\mc{O}_X)$, and if $f$ has
finite Tor-dimension (cf.\ \cite[II \S 4]{Har}, {\it e.g.}\ flat
morphism), then it takes bounded complexes into bounded complexes.

\subsection{}
Now, consider the following diagram of schemes
\begin{equation*}
 \xymatrix{
  Y\ar[rr]^{g}&&Z\\
 &X\ar[ul]^f\ar[ur]_h&
  }
\end{equation*}
where $f$ and $h$ are regular closed immersions of codimension $d>0$,
and $g$ is a finite flat morphism. Since $g$ is finite flat, we get that
\begin{equation*}
 g^{\flat}(\mc{O}_Z)\cong\overline{g}^*\shom_{\mc{O}_Z}(g_*
  \mc{O}_Y,\mc{O}_Z).
\end{equation*}
There exists the natural equivalence $f^{\flat}g^{\flat}\cong h^{\flat}$
by \cite[III Proposition 6.2]{Har}. By taking the $d$-th cohomology, we
get an isomorphism
\begin{equation*}
 \iota\colon\overline{f}^*\sext^d_{\mc{O}_Y}(f_*\mc{O}_X,\overline{g}^*
  \shom_{\mc{O}_Z}(g_*\mc{O}_Y,\mc{O}_Z))\xrightarrow{\sim}
  \overline{h}^*\sext^d_{\mc{O}_Z}(h_*\mc{O}_X,\mc{O}_Z).
\end{equation*}
Now, suppose $Z$ is an affine scheme. Then the other two schemes are
also affine schemes. We denote the global sections of $X$ (resp.\ $Y$,
$Z$) by $R_X$ (resp.\ $R_Y$, $R_Z$). By \cite[III, Proposition
6.1]{Har}, the source and target of $\iota$ are quasi-coherent
$\mc{O}_X$-modules. Thus, $\iota$ is associated to the following
isomorphism of $R_X$-modules
\begin{equation*}
  \mr{Ext}^d_{R_Y}(R_X,\mr{Hom}_{R_Z}(R_Y,R_Z))
   \xrightarrow{\sim}\mr{Ext}^d_{R_Z}(R_X,R_Z),
\end{equation*}
and we also denote this isomorphism by $\iota$. We calculate this
isomorphism in terms of the fundamental local isomorphism \cite[III,
Proposition 7.2]{Har}.

Suppose moreover that there exists a system of local parameters
defining $X$ in $Y$ (resp.\ in $Z$) denoted by $\{y_i\}_{i\leq i\leq d}$
(resp.\ $\{z_i\}_{1\leq i\leq d}$). Let $I:=\mr{Ker}(R_Y\rightarrow
R_X)$. The sheaf $\omega_{X/Y}$ is the quasi-coherent sheaf associated
to $\mr{Hom}_{R_X}(\wedgeb^dI/I^2,R_X)$. Since
$y_1\wedge\dots\wedge y_d$ defines a basis of $\wedgeb^dI/I^2$, we
denote by $(y_1\wedge\dots\wedge y_d)^\vee$ its dual basis. In the same
way, we define a basis $(z_1\wedge\dots\wedge z_d)^\vee$ of
$\omega_{X/Z}$.

\begin{lem}
 \label{kernelcalc}
 We preserve the notation. We define
 a homomorphism $\alpha$ in the following diagram so that it is
 commutative.
 \begin{equation*}
  \xymatrix@C=80pt{
  \mr{Ext}^d_{R_Y}(R_X,\mr{Hom}_{R_Z}(R_Y,
   R_Z))\ar[r]_<>(.5){\sim}\ar[d]_{\iota}^{\sim}&
   \omega_{X/Y}\otimes_{R_Y}\mr{Hom}_{R_Z}(R_Y,R_Z)
   \ar[d]^{\alpha}\\
   \mr{Ext}^d_{R_Z}(R_X,R_Z)\ar[r]^{\sim}_{\beta}&
   \omega_{X/Z}
   }
 \end{equation*}
 Here the horizontal isomorphisms are the isomorphisms of {\normalfont
 \cite[III, 7.2]{Har}}. Let $g^*(z_i)=\sum_{1\leq j\leq d} f_{ij}y_j$
 where $f_{ij}\in R_Y$. Here the expression may not be unique, but take
 one. We put $G:=(f_{ij})_{i\leq i,j\leq d}\in\mr{Mat}_{d\times
 d}(R_Y)$. Then
 \begin{equation*}
  \alpha((y_1\wedge\dots\wedge y_d)^{\vee}\otimes\varphi)
   =\overline{\varphi(\det(G))}\cdot(z_1\wedge\dots\wedge z_d)
   ^\vee
 \end{equation*}
 where the over-line denotes to take the image of the homomorphism
 $R_Z\rightarrow R_X$ inducing $h$.
\end{lem}
\begin{proof}
 On the way we prove the lemma, we will review the definition of the
 homomorphism $\beta$. Let $R_Z\zeta_i$ be a
 free $R_Z$-module of rank $1$ whose generator is $\zeta_i$. Let
 $K_{\bullet}:=\wedgeb^{\bullet}(\bigoplus_{i=1}^dR_Z\zeta_i)$ be
 the Koszul complex. By definition, the differential homomorphism
 $K_r\rightarrow K_{r-1}$ is defined by sending
 $\zeta_{i_1}\wedge\dots\wedge\zeta_{i_r}$ to
 $\sum(-1)^{j}\,z_{i_j}\,\zeta_{i_1}\wedge\dots\wedge
 \widehat{\zeta_{i_j}}\wedge\dots\wedge\zeta_{i_r}$ where
 $\widehat{\zeta_{i_j}}$ means omit $\zeta_{i_j}$. The canonical
 homomorphism $R_Z\rightarrow R_X$ defines a complex
 \begin{equation*}
  \wedgeb^{\bullet}\Bigl(\bigoplus_{i=1}^{d}R_Z\zeta_i\Bigr)
   \rightarrow R_X\rightarrow 0,
 \end{equation*}
 which is known to be a free resolution of $R_X$. Now, let us
 define a homomorphism
 \begin{equation*}
  \wedgeb^r\Bigl(\bigoplus_{i=1}^dR_Z\zeta_i\Bigr)\rightarrow
   \wedgeb^r\Bigl(\bigoplus_{i=1}^dR_Y\nu_i\Bigr)
 \end{equation*}
 by mapping $\zeta_{i_1}\wedge\dots\wedge\zeta_{i_r}$ to
 \begin{equation*}
  \Bigl(\sum_{1\leq j\leq d}f_{i_1,j}\nu_j\Bigr)\wedge\dots\wedge
   \Bigl(\sum_{1\leq j\leq d}f_{i_r,j}\nu_j\Bigr)=\sum_{(j_1,\dots,j_r)
   \in[1,d]^r}f_{i_1,j_1}\dots f_{i_r,j_r}
   \cdot\nu_{j_1}\wedge\dots\wedge\nu_{j_r}
 \end{equation*}
 where $[1,d]$ is the set $\{i\in\mb{Z}\mid 1\leq i\leq d\}$.
 Then it is a standard calculation to check that these homomorphisms
 define a homomorphism of Koszul complexes:
 \begin{equation*}
  \gamma\colon
  \wedgeb^{\bullet}\Bigl(\bigoplus_{i=1}^{d}R_Z\zeta_i\Bigr)
   \rightarrow\wedgeb^{\bullet}\Bigl(\bigoplus_{i=1}^{d}R_Y\nu_i
   \Bigr).
 \end{equation*}
 This induces the following commutative diagram.
 \begin{equation*}
  \xymatrix{
   \mr{Hom}_{R_Y}(\wedgeb^{d}(\bigoplus_{i=1}^{d}
   R_Y\nu_i),R_Y)\otimes_{R_Y}\mr{Hom}_{R_Z}
   (R_Y,R_Z)\ar[d]\ar[r]&\omega_{X/Y}\otimes_{R_Y}
   \mr{Hom}_{R_Z}(R_Y,R_Z)\ar[d]^\alpha\\
  \mr{Hom}_{R_Z}(\wedgeb^{d}(\bigoplus_{i=1}^{d}R_Z
   \zeta_i),R_Z)\ar[r]&\omega_{X/Z}
   }
 \end{equation*}
 Here, the left vertical arrow is induced by $\gamma$.
 The bottom horizontal arrow is the surjective homomorphism defined by
 sending $\phi$ to
 $\overline{\phi(\zeta_1\wedge\dots\wedge\zeta_d)}\cdot
 (z_1\wedge\dots\wedge z_d)^{\vee}$. This factors through
 $\mr{Ext}^d_{R_Z}(R_X,R_Z)$, and this is $\beta$ by definition. The top
 horizontal arrow is defined in the same manner. The homomorphism
 $\gamma$ sends $\zeta_1\wedge\dots\wedge\zeta_d$ to
 $\det(G)\cdot(\nu_1\wedge\dots\wedge\nu_d)$. Thus, we get the lemma.
\end{proof}

\subsection{}
\label{notationbefcalc}
Now we will calculate the isomorphism $\mu_X$ in \cite[Lemme
2.4.2]{Ber2}, which is one of the two ingredients to calculate the
Frobenius isomorphisms explicitly. The other ingredient is the
explicit calculation of Frobenius by Garnier, which we will review in
paragraph \ref{Garnierrev}.

Let us fix the situation and notations. We fix a positive integer $s>0$,
and put $q:=p^s$ as in Notation. Let $S$ be a scheme endowed with
quasi-coherent $m$-PD ideal $(\mf{a},\mf{b},\alpha)$ such that
$p\in\mf{a}$ and $p$ is nilpotent, and $X$ be a smooth scheme over
$S$ of relative dimension $d$. Let $S_0$ be the
subscheme of $S$ defined by $\mf{a}$, and $X_0:=X\times_{S}S_0$. Suppose
$S$ and $X$ are affine schemes, and $X$ possesses a system of local
coordinates $\{x_i\}_{1\leq i\leq d}$ over $S$ ({\it i.e.\ }the
homomorphism $X\rightarrow\mb{A}^d_S$ induced by $\{x_i\}$ is
\'{e}tale). Recall $X^{(s)}_0:=X_0\otimes_{S_0,F^s_{S_0}}S_0$ is the
$s$-th relative Frobenius of $X_0$ over $S_0$. Let $X'$ be a smooth
lifting of $X_0^{(s)}$ over $S$. There exists a system of local
coordinates $\{y_i\}_{1\leq i\leq d}$ of $X'$.
Since $S$ and $X$ are affine, we may lift the relative Frobenius
homomorphism denoted by $F\colon X\rightarrow X'$ over $S$ uniquely such
that $F^*(y_i)=x_i^{q}$.

We will use multi-index notation.
For an integer $i$, we put $\underline{i}:=(i,\dots,i)$ in
$\mb{Z}^d$. For $\underline{k}=(k_1,\dots,k_d)$ and
$\underline{k}'=(k'_1,\dots,k'_d)$ in $\mb{Z}^d$, we denote by
$\underline{k}<\underline{k}'$
(resp. $\underline{k}\leq\underline{k}'$) if $k_i<k'_i$  (resp.\ $k_i\leq
k'_i$) for any $1\leq i\leq d$. We
define $\underline{k}-\underline{k}':=(k_1-k'_1,\dots,k_d-k'_d)$.

We know that
\begin{equation*}
 F_*\mc{O}_{X}\cong\bigoplus_{\underline{0}\leq\underline{k}<
  \underline{q}}\mc{O}_{X'}\underline{x}^{\underline{k}}.
\end{equation*}
Consider the dual $\mc{O}_X^{\vee}:=\mr{Hom}_{\mc{O}_{X'}}(F_*\mc{O}_X,
\mc{O}_{X'})$. We denote the dual basis of
$\{\underline{x}^{\underline{k}}\}$ by
$\{H\underline{x}^{-\underline{k}}\}$. The notation may seem a little
strange, but this notation is used to be consistent with Garnier's
calculation (cf.\ paragraph \ref{Garnierrev}).
For a quasi-coherent $\mc{O}_{X'}$-module $\ms{M}$, we get
\begin{align}
 \label{flatnotation}
 F^*(\ms{M})&=\mc{O}_X\otimes_{\mc{O}_{X'}}\ms{M}\\\notag
 F^\flat(\ms{M})&=\mr{Hom}_{\mc{O}_{X'}}(F_*\mc{O}_X,\ms{M})
 \cong\ms{M}\otimes_{\mc{O}_{X'}}\mc{O}_X^{\vee}
\end{align}
as $\mc{O}_X$-modules. We identify them, and a section
$m\otimes\varphi$ of $\ms{M}\otimes_{\mc{O}_{X'}}\mc{O}_X^{\vee}$ is
considered to be a section of $F^{\flat}(\ms{M})$.

\begin{prop}
 \label{mainpropbalhar}
 We preserve the notation, in particular $S$, $X$, $X'$ are
 affine. Recall the isomorphism of Berthelot {\normalfont
 \cite[2.4.2]{Ber2}}
 \begin{equation*}
  \mu_X\colon\omega_{X}\xrightarrow{\sim}F^{\flat}\omega_{X'}.
 \end{equation*}
 Using {\normalfont (\ref{flatnotation})}, we can described this
 isomorphism by
 \begin{equation*}
  \mu_X(dx_1\wedge\dots\wedge dx_d)=
   (dy_1\wedge\dots\wedge dy_d)\otimes
   H\underline{x}^{-(\underline{q}-\underline{1})}.
 \end{equation*}
\end{prop}
\begin{proof}
 Before starting the proof, we remind that Conrad pointed out in
 \cite{Con} that with sign convention of \cite{Har}, many compatibilities
 stated in \cite{Har} do not hold, and we need to use the modified
 convention as in \cite[2.2]{Con}. In this proof, since the outcome does not
 change, we follow the conventions of \cite{Har}. For skeptical readers,
 we put signs $\ccirc{1}$ through $\ccirc{4}$ below arrows of the
 homomorphisms whose sign change if we use the conventions of
 \cite{Con}, and see how they differ at the very end of this proof.

 To avoid confusions, we put $Y:=X'$ in this
 proof. Consider the following diagram.
 \begin{equation*}
  \xymatrix{
   X\times_SX\ar[rr]^{1\times F}\ar@<0.5ex>[rd]^{p_2}&&
   X\times_SY\ar[dr]^{q_Y}\ar@<-0.5ex>[dl]_{q_X}&\\
  &X\ar@<0.5ex>[ul]^{\Delta}\ar[rr]_{F}\ar[rd]_f
   \ar@<-0.5ex>[ur]_s&&Y\ar[dl]^g\\&&S&
   }
 \end{equation*}
 Here $p_2$ is the second projection, $\Delta$ is the diagonal morphism,
 $q_X$ is the first projection, $s$ is the graph morphism of $F$, $q_Y$
 is the second projection, and $f$ and $g$ are structural morphisms.
 For an affine scheme $Z$, we denote the global sections of $Z$ by
 $R_Z$. Note that all the schemes appearing in the diagram are affine.
 Let us consider $X\times_SX$
 (resp.\ $X\times_SY$) as a scheme over $X$ by the projection $p_2$
 (resp.\ $q_X$) unless otherwise stated. We put
 \begin{alignat*}{2}
  x_{1,i}&:=x_i\otimes1&\qquad&\text{in ${R_{X\times_SX}}$ and
  ${R_{X\times_SY}}$}\\
  x_{\Delta,i}&:=x_i\otimes1-1\otimes x_i&\qquad&\text{in
  $R_{X\times_SX}$}\\
  x_{s,i}&:=F^*(y_i)\otimes1-1\otimes y_i&\qquad&\text{in
  $R_{X\times_SY}$}.
 \end{alignat*}
 The set
 $\{x_{\Delta,1},\dots,x_{\Delta,d},x_{1,1},\dots,x_{1,d}\}$ (resp.\
 $\{x_{s,1},\dots,x_{s,d},x_{1,1},\dots,x_{1,d}\}$) forms a
 system of local coordinates of $X\times_SX$ (resp.\ $X\times_SY$), and
 $\{x_{\Delta,1},\dots,x_{\Delta,d}\}$
 (resp. $\{x_{s,1},\dots,x_{s,d}\}$) defines a local system of
 parameters defining $\Delta(X)$ (resp.\ $s(X)$). We also note that
 $dx_{\Delta,\bullet}:=dx_{\Delta,1}\wedge\dots\wedge dx_{\Delta,d}$
 defines a basis of $\omega_{X\times_SX/X}$. For
 $\varphi\in\mr{Hom}_{R_Y}(R_X,R_Y)$, we denote by
 $\varphi'\in\mr{Hom}_{R_{X\times_SY}}(R_{X\times_SX},R
 _{X\times_SY})$ the homomorphism defined by $1\otimes\varphi$.
 
 Now, let $d'\colon {R_X}\rightarrow
 R_{X\times_SX}\cong{R_X}\otimes_{R_S}{R_X}$ be a
 homomorphism of $R_S$-algebras defined by $d'(b):=b\otimes1-1\otimes
 b$. We consider $R_{X\times_SX}$ as an $R_X$-algebra by the {\it first}
 component for a while. Then for $a,b\in R_X$, we get
 \begin{equation*}
  d'(ab)=a\,d'(b)+b\,d'(a)-d'(a)\cdot d'(b).
 \end{equation*}
 Thus,
 \begin{align*}
  &(1\times F)^*(x_{s,i})=d'(F^{*}(y_i))=d'(x_i^q)=\Bigl(-d'x_i^{q-1}
   +\sum_{q-1>j\geq 0}f_{i,j}\,d'x_i^{j}\Bigr)\cdot d'x_i\\
  &\qquad=:F_{i}\cdot x_{\Delta,i}
 \end{align*}
 with $f_{i,j}\in R_X$. By definition, we have
 \begin{equation*}
  \begin{cases}   
   (H{x}_i^{-(q-1)})'(d'{x}_j^{k})=0&
   \text{for any $i\neq j$ or $i=j$ and $k\neq q-1$}\\
   (H{x}_i^{-(q-1)})'
   (d'{x}_i^{q-1})=-1\otimes 1.&
  \end{cases} 
\end{equation*}
 Let 
 \begin{equation}
  \label{calcofG}
  G:=\mr{diag}(F_1,\dots,F_d)\in\mr{Mat}_{d\times d}(R_{X
  \times_SX})
 \end{equation}
 where $\mr{diag}$ denotes the diagonal matrix. Then, we obtain
 \begin{equation}
  \label{calcimH}
  (H{\underline{x}^{-(\underline{q}-\underline{1})}})'
   (\det(G))=(H{\underline{x}^{-(\underline{q}-\underline{1})}})'
   ((-1)^dd'x_1^{q-1}\dots d'x_d^{q-1})=(-1\otimes 1)^{2d}
   =1\otimes 1.
 \end{equation}

 We set back the convention, and consider $X\times_SX$ as a scheme over
 $X$ by $p_2$. The homomorphism $\mu_X$ is defined in the following way:
 \begin{equation*}
  \omega_X\cong f^{\sharp}\mc{O}_S[-d]\cong F^{\flat}g^{\sharp}
   \mc{O}_S[-d]\cong F^{\flat}\omega_{Y/S}
 \end{equation*}
 where the first and third isomorphisms are by definition \cite[III,
 \S 2]{Har} and the second isomorphism is induced by \cite[III,
 Proposition 8.4]{Har}. In the rest of this proof, we will drop the
 section number III when we cite \cite{Har}. Since the sheaves we are
 considering are quasi-coherent and schemes are affine, we do not make
 any difference between sheaves and its global sections.

 We will start to calculate from $F^\flat\omega_{Y/S}$. In the rest of
 this proof, we will use the identification
 $F^{\flat}\ms{M}\cong\mc{O}_X^{\vee}\otimes_{\mc{O}_{X'}}\ms{M}$ to
 describe the elements contrary to the standard convention
 (\ref{flatnotation}) of this paper. Thus the sheaf
 $F^\flat\omega_{Y/S}$ is identified with
 $\mr{Hom}_{R_Y}(R_X,R_Y)\otimes_{R_Y}\omega_{Y/S}$. Take an element
 \begin{equation*}
  \varphi\otimes dy_{\bullet}\in~\mr{Hom}_{R_Y}(R_X,R_Y)
   \otimes_{R_Y}\omega_{Y/S}.
 \end{equation*}
 First, we need to calculate the isomorphism
 \underline{$a\colon F^\flat\omega_{Y/S}\xrightarrow{\sim}s^{\flat}
 q_Y^{\sharp}\omega_{Y/S}$}, which is the third isomorphism in the proof
 of \cite[8.4]{Har}. This isomorphism is the isomorphism of
 \cite[8.2]{Har}. To calculate this, first, we get an isomorphism
 \begin{align*}
  &\mr{Hom}_{R_Y}(R_X,R_Y)\otimes_{R_Y}\omega_{Y/S}\cong
  F^\flat\omega_{Y/S}\\
  &\qquad\xrightarrow[\ccirc{1}]{\sim}\Delta^{\flat}p_2^{\sharp}
  F^{\flat}\omega_{Y/S}\cong\omega_{X/X\times_SX}\otimes_{R_{X\times X}}
  \omega_{X\times_SX/X}\otimes_{R_{X}}\mr{Hom}_{R_Y}(R_X,R_Y)
  \otimes_{R_Y}\omega_{Y/S}.
 \end{align*}
 This isomorphism is defined in \cite[8.1]{Har}. Let
 $x^\vee_{\Delta,\bullet}$ denotes the dual basis of
 $x_{\Delta,1}\wedge\dots\wedge x_{\Delta,d}$ in
 $\omega_{X/X\times_SX}$, $dy_{\bullet}$ denotes $dy_1\wedge\dots\wedge
 dy_d$, and $dx_{1,\bullet}$ denotes $dx_{1,1}\wedge\dots\wedge
 dx_{1,d}$. We define $x^\vee_{s,\bullet}$ and $dx_{s,\bullet}$ in the
 same way. Then the isomorphism sends
 $\varphi\otimes dy_{\bullet}$ to $x^{\vee}_{\Delta,\bullet}
 \otimes dx_{1,\bullet}\otimes\varphi\otimes dy_{\bullet}$. Secondly, we get an
 isomorphism
 \begin{align*}
  &\omega_{X/X\times_SX}\otimes_{R_{X\times X}}\omega_{X\times_SX/X}
  \otimes_{R_{X}}\mr{Hom}_{R_Y}(R_X,R_Y)\otimes_{R_Y}\omega_{Y/S}\cong
  \Delta^{\flat}p_2^{\sharp}F^{\flat}\omega_{Y/S}\xrightarrow{\sim}
  \Delta^{\flat}(1\times F)^{\flat}q_Y^{\sharp}\omega_{Y/S}\\
  &\qquad\cong\omega_{X/X\times_SX}\otimes_{R_{X\times X}}
  \mr{Hom}_{R_{X\times_SY}}(R_{X\times_SX},
  R_{X\times_SY})\otimes_{R_{X\times Y}}\omega_{X\times_SY/Y}
  \otimes_{R_Y}\omega_{Y/S}.
 \end{align*}
 This isomorphism is defined in \cite[6.3]{Har}, and sends
 $x^{\vee}_{\Delta,\bullet}\otimes dx_{1,\bullet}\otimes\varphi
 \otimes dy_{\bullet}$ to $x^{\vee}_{\Delta,\bullet}\otimes
 \varphi'\otimes dx_{1,\bullet}\otimes dy_{\bullet}$. Thirdly, we get an
 isomorphism
 \begin{align*}
  &\omega_{X/X\times_SX}\otimes\mr{Hom}_{R_{X\times_SY}}(R_{X\times_SX},
  R_{X\times_SY})\otimes\omega_{X\times_SY/Y}\otimes\omega_{Y/S}\cong
  \Delta^{\flat}(1\times F)^{\flat}q_Y^{\sharp}\omega_{Y/S}\\
  &\qquad\xrightarrow{\sim}s^{\flat}q_Y^{\sharp}\omega_{Y/S}\cong
  \omega_{X/X\times_SY}\otimes_{R_{X\times_SY}}\omega_{X\times_SY/Y}
  \otimes_{R_Y}\omega_{Y/S}.
 \end{align*}
 This isomorphism is defined in \cite[6.2]{Har}, and this is the
 homomorphism we calculated in Lemma \ref{kernelcalc}. Thus, using this
 lemma, we get that it sends $x^{\vee}_{\Delta,\bullet}\otimes\varphi'
 \otimes dx_{1,\bullet}\otimes dy_{\bullet}$ to $\varphi'(\det(G))
 \cdot x^{\vee}_{s,\bullet}\otimes dx_{1,\bullet}\otimes dy_{\bullet}$
 where $G$ is the matrix defined in (\ref{calcofG}). Combining these
 three isomorphisms we got, we obtain
 \begin{align*}
  a\colon\mr{Hom}_{R_Y}(R_X,R_Y)\otimes_{R_Y}\omega_{Y/S}
  &\xrightarrow{\sim}\omega_{X/X\times_SY}\otimes_{R_{X\times_SY}}
  \omega_{X\times_SY/Y}\otimes_{R_Y}\omega_{Y/S}\\
  \varphi\otimes dy_{\bullet}&\mapsto\overline{\varphi'(\det(G))}
  \cdot x^{\vee}_{s,\bullet}\otimes dx_{1,\bullet}\otimes dy_{\bullet}
 \end{align*}
 where the over-line denotes taking the image of the canonical
 homomorphism $R_{X\times_SY}\rightarrow R_X$ inducing the morphism
 $s$.

 Now, we come back to the definition of the isomorphism of
 \cite[8.4]{Har}. We need to calculate the isomorphism
 \underline{$b\colon s^{\flat}q_Y^{\sharp}\omega_{Y/S}\xrightarrow{\sim}
 s^{\flat}q_X^{\sharp}f^{\sharp}\mc{O}_S[-d]$}, which is the second
 isomorphism in the proof of \cite[8.4]{Har}. We have an isomorphism
 \begin{align*}
  &s^{\flat}q_Y^{\sharp}\omega_{Y/S}\cong\omega_{X/X\times_SY}
  \otimes_{R_{X\times_SY}}\omega_{X\times_SY/Y}\otimes_{R_Y}
  \omega_{Y/S}\\
  &\qquad\xrightarrow[\ccirc{2}]{\sim}s^{\flat}(f\circ q_X)^{\sharp}
  \mc{O}_S[-d]\cong\omega_{X/X\times_SY}\otimes_{R_{X\times_SY}}
  \omega_{X\times_SY/S}.
 \end{align*}
 This isomorphism is defined in \cite[2.2]{Har}, and sends
 $x^{\vee}_{s,\bullet}\otimes dx_{1,\bullet}\otimes
 dy_{\bullet}$ to $x^{\vee}_{s,\bullet}\otimes((-1)^ddx_{s,\bullet}
 \wedge dx_{1,\bullet})$. Then we get an isomorphism
 \begin{equation*}
  \omega_{X/X\times_SY}\otimes_{R_{X\times_SY}}\omega_{X\times_SY/S}
   \cong s^{\flat}(f\circ q_X)^{\sharp}\mc{O}_S[-d]
   \xrightarrow[\ccirc{3}]{\sim}s^\flat q_X^\sharp f^\sharp\mc{O}_S
   [-d]\cong\omega_{X/X\times_SY}\otimes\omega_{X\times_SY/X}\otimes
   \omega_{X/S}.
 \end{equation*}
 This is also an isomorphism of \cite[2.2]{Har}, and sends $x^{\vee}
 _{s,\bullet}\otimes(dx_{s,\bullet}\wedge dx_{1,\bullet})$ to
 $(-1)^{d^2}x^{\vee}_{s,\bullet}\otimes dx_{s,\bullet}\otimes
 dx_{\bullet}$. Since $d^2+d$ is even, we get
 \begin{align*}
  b\colon\omega_{X/X\times_SY}\otimes_{R_{X\times_SY}}\omega_{X\times_SY/Y}
  \otimes_{R_Y}\omega_{Y/S}&\xrightarrow{\sim}\omega_{X/X\times_SY}
  \otimes\omega_{X\times_SY/X}\otimes\omega_{X/S}\\
  x^{\vee}_{s,\bullet}\otimes dx_{1,\bullet}\otimes dy_{\bullet}&
  \mapsto x^{\vee}_{s,\bullet}\otimes dx_{s,\bullet}\otimes
  dx_{\bullet}.
 \end{align*}

 At last, we get an isomorphism
 \begin{equation*}
  c\colon\omega_{X/X\times_SY}\otimes\omega_{X\times_SY/X}\otimes
   \omega_{X/S}\cong s^\flat q_X^\sharp\omega_{X/S}
   \xrightarrow[\ccirc{4}]{\sim}\omega_{X/S}.
 \end{equation*}
 This isomorphism is defined in \cite[8.1]{Har}, and sends $x^{\vee}
 _{s,\bullet}\otimes dx_{s,\bullet}\otimes dx_{\bullet}$ to
 $dx_{\bullet}$.

 Now, by definition, $\mu_X^{-1}=c\circ b\circ a$. The above calculation
 shows that
 \begin{equation*}
  \mu_X^{-1}(\varphi\otimes dy_{\bullet})=
   \overline{\varphi'(\det(G))}\cdot dx_{\bullet}.
 \end{equation*}
 Taking $\varphi=H\underline{x}^{-(\underline{q}-\underline{1})}$
 and considering (\ref{calcimH}), we get the lemma.

 As we noted at the beginning of this proof, we need some modification
 for the calculation of the homomorphisms $\ccirc{1}$ through
 $\ccirc{4}$ if we use the convention of \cite[2.2]{Con}. Precisely for
 the homomorphisms $\ccirc{1}$ and $\ccirc{4}$, we need to multiply by
 $(-1)^{d(d-1)/2}$, and for the homomorphisms $\ccirc{2}$ and
 $\ccirc{3}$, we need to multiply by $(-1)^{d^2}$. Thus, $\mu_X^{-1}$ is
 multiplied by $(-1)^{d(d-1)+2d^2}=1$, and the result remains to be the
 same as we stated.
\end{proof}

\begin{rem}
 When $F^*(y_i)$ is not equal to $x_i^q$, we can also calculate in the
 same way. We can write $d'(F^*(y_j))=\sum_{i}f_{i,j}\cdot d'x_i$ where
 $f_{i,j}\in R_{X\times_SX}$ using the notation of the proof of the
 proposition. We put $g_{i,j}:=((1\otimes
 H\underline{x}^{q-1})(f_{i,j}))^{-}$ where the
 over-line denotes to take the image of the canonical homomorphism
 $R_{X\times_SY}\rightarrow R_X$. Using this, we get
 \begin{equation*}
  \iota(dx_1\wedge\dots\wedge dx_d)=(\det\widetilde{G})^{-1}(H
   \underline{x}^{q-1})~dy_1\wedge\dots\wedge dy_d
 \end{equation*}
 where
 \begin{equation*}
  \widetilde{G}=
  \begin{pmatrix}
   g_{1,1}&\dots&g_{1,d}\\
   \vdots&&\vdots\\
   g_{d.1}&\dots&g_{d,d}
  \end{pmatrix}.
 \end{equation*}
 Note that this matrix is invertible since $\widetilde{G}\equiv I\bmod
 p\cdot\mr{Mat}(R_X)$. We do not know how we compute the determinant of
 this matrix further.
\end{rem}

Now the following theorem follows from the construction and
Proposition \ref{mainpropbalhar}.

\begin{thm}
 \label{maincalc}
 We preserve the notation of paragraph
 {\normalfont\ref{notationbefcalc}}, and let us denote by
 $(dx_\bullet)^{\vee}$ the dual basis of
 $dx_\bullet:=dx_1\wedge\dots\wedge dx_d$ in $\omega_{{X}}^{-1}$, and
 the same for $dx'_\bullet$ and $(dx'_\bullet)^\vee$. Let
 $\ms{M}$ be a left $\Dmod{m}{{X}'}$-module, and $\ms{N}$ be a right
 $\Dmod{m}{{X}'}$-module. Recall two isomorphisms of Berthelot
 {\normalfont \cite[2.4.3, 2.4.4]{Ber2}:}
 \begin{align*}
  \mu_{\ms{M}}&\colon\omega_{{X}}\otimes_{\mc{O}_{X}}F^*\ms{M}
  \xrightarrow{\sim}F^{\flat}(\omega_{{X}'}\otimes_{\mc{O}_{X'}}
  \ms{M}),\\
  \nu_{\ms{N}}&\colon\omega^{-1}_{{X}}\otimes_{\mc{O}_{X}}
  F^{\flat}\ms{N}\xrightarrow{\sim}F^*(\omega_{{X}'}^{-1}\otimes
  _{\mc{O}_{X'}}\ms{N}).
 \end{align*}
 Let $m\in\ms{M}$, $m'\in\ms{N}$, and $f\in\mc{O}_X$. Then we get that
 \begin{align*}
  &\mu_{\ms{M}}(dx_{\bullet}\otimes(f\otimes m))=(dx'_\bullet
  \otimes m)\otimes(H\underline{x}^{-(\underline{q}-
  \underline{1})}\cdot f),\\
  &\nu_{\ms{N}}\bigl((dx_\bullet)^{\vee}\otimes (m'\otimes
  (H\underline{x}^{-(\underline{q}-\underline{1})}\cdot f))\bigr)=
  f\otimes(dx'_\bullet)^{\vee}\otimes m'
 \end{align*}
 by using the notation of {\normalfont (\ref{flatnotation})}.
 \end{thm}

\section{Explicit calculation of Frobenius isomorphisms}
In this section, we will give first applications of the theorem in the
previous section. The main result of this section is the calculation of
the Frobenius structure of push-forwards.

\subsection{}
\label{Frobsituations}
We will fix two situations for the basis $R$ often used in this paper.
\begin{enumerate}
 \item \label{noFrobstr}
       The ring $R$ is complete discrete valuation ring as in Notation.

 \item  \label{Frobstr}
	We moreover assume that the $s$-th absolute Frobenius
	isomorphism $F^s_{k}$ lifts to an automorphism $\mr{Spf}(R)
	\xrightarrow{\sim}\mr{Spf}(R)$ which is denoted by $\sigma$. In
	this case $k$ is automatically perfect.
\end{enumerate}
For a scheme $X$ over $k$, we recall $X^{(s)}:=X\otimes_{k,F^{s*}_k}k$.
Let $\ms{X}$ be a smooth formal scheme over $\mr{Spf}(R)$, and let
$X_0$ the special fiber. Suppose that $X^{(s)}_0$ can be
lifted to a smooth formal scheme $\ms{X}'$ over $\mr{Spf}(R)$.
In the situation \ref{noFrobstr}, we are able to consider Frobenius
pull-backs even if there are no lifting of the relative Frobenius
morphism $X_0\rightarrow X^{(s)}_0$ (cf.\ \cite[2.2.3]{Ber2}). Thus we
are able to discuss the commutativity of Frobenius pull-backs with
several cohomological operations such as push-forwards or duals {\it
etc}. In the situation \ref{Frobstr}, moreover, we are able to define
$F$-$\DdagQ{\ms{X}}$-modules (cf.\ \cite[4.5.1]{Ber2}).

\subsection{}
\label{Garnierrev}
We will review the second ingredient to calculate Frobenius
isomorphisms, which are results of Garnier \cite{Gar}. See \cite{Ab} for
another aspect of Garnier's result.

We will consider the situation \ref{Frobsituations}.\ref{noFrobstr}.
Let $\ms{X}$ be a smooth affine formal scheme over $R$ whose special
fiber is denote by $X_0$, and suppose
given a system of local coordinates $\{x_1,\dots,x_d\}$. We denote by
$\{\partial_1,\dots,\partial_d\}$ the corresponding differential
operators. For a positive
integer $s$, we let $F_k^{s*}\colon k\rightarrow k$ be the $s$-th
absolute Frobenius homomorphism. Let $\ms{X}'$ be a smooth affine formal
scheme over $R$ which is a lifting of $X_0\times_{\sigma_s}k$. The
relative Frobenius morphism $X_0\rightarrow X_0\otimes_{F^{s*}_k}k$ can
be lifted to a morphism $F_{\ms{X}}\colon\ms{X}\rightarrow\ms{X}'$ by
the universal property of smoothness since $\ms{X}$ is assumed to be
affine. We sometimes denote $F_{\ms{X}}$ by $F$. We also fix a system of
local coordinates $\{x'_1,\dots,x'_d\}$ of $\ms{X}'$ such that
$F^*(x')=x^{q}$ using the universal property once again. We denote the
corresponding differential operators by
$\{\partial'_1,\dots,\partial'_d\}$.

Garnier constructed in \cite{Gar} a special differential operator
$H\in\Dcomp{s}{\ms{X}}$ called the {\em Dwork operator} with the
following properties:
\begin{enumerate}
 \item Suppose a primitive $q$-th root of unity is contained in $R$. Then 
       \begin{equation*}
	H_i=\sum_{\zeta^q=1}\sum_{k\geq 0}(\zeta-1)^k
	 x_i^k\partial_i^{[k]},\qquad
	 H:=\prod_{1\leq i\leq d}H_i.
       \end{equation*}
       These are global sections of $\Dcomp{s}{\ms{X}}$. The operator
       $H_i$ is called the {\em Dwork operator corresponding to $x_i$}.
       (cf.\ \cite[Proposition 4.5.2]{Gar})

 \item The operator $H$ is a projector from $\mc{O}_{\ms{X}}$ to
       $\mc{O}_{\ms{X}'}$. Precisely, we have
       $H^2=H$ in $\Dcomp{s}{\ms{X}}$, and its action on
       $\mc{O}_{\ms{X}}$ is $\mc{O}_{\ms{X}'}$-linear. (cf.\
       \cite[Proposition 2.5.1]{Gar})

 \item For $\underline{0}\leq\underline{k}<\underline{q}$, we get that
       $H\underline{x}^{-\underline{k}}$ in an element of
       $\Dcomp{s}{\ms{X}}$. We have $\sum_{\underline{0}\leq
       \underline{k}<\underline{q}}\underline{x}^{\underline{k}}
       H\underline{x}^{-\underline{k}}=1$.\\
       (cf.\ \cite[Proposition 2.5.1, 2.5.3]{Gar})

 \item For $\underline{0}\leq\underline{k}<\underline{q}$, the operator
       $H\underline{x}^{-\underline{k}}$ defines an
       $\mc{O}_{\ms{X}'}$-linear
       homomorphism $\mc{O}_{\ms{X}}\rightarrow\mc{O}_{\ms{X}'}$, and
       defines an element of $\mc{O}_{\ms{X}}^\vee$. The set
       $\{H\underline{x}^{-\underline{k}}\}$ defines the dual basis of
       $\{\underline{x}^{\underline{k}}\}$. \label{dualbasis}\\
       (cf.\ \cite[Proposition 2.5.1, 2.5.3]{Gar})
\end{enumerate}
The property \ref{dualbasis} justifies the notation
$H\underline{x}^{-\underline{k}}$ as the dual basis used in
paragraph \ref{notationbefcalc}.
Now, we define
\begin{equation*}
 (\partial'_i)^{\circ}:=(qx^{q-1})^{-1}\partial_iH,
\end{equation*}
and $P:=\sum_{\underline{k}}f_{\underline{k}}~
\underline{\partial'}^{\underline{k}}$ in $\Dcomp{m}{\ms{X}'}$ with
$f_{\underline{k}}\in\mc{O}_{\ms{X}',\mb{Q}}$, we put
\begin{equation*}
 P^{\circ}:=\sum_{\underline{k}}F^*(f_{\underline{k}})\cdot(\underline
  {\partial'})^{\circ\underline{k}}.
\end{equation*}
Note that $P^{\circ}$ is denoted by $P'$ in \cite[4.6.1]{Gar}. This
defines a ring homomorphism
$\Dcomp{m}{\ms{X}'}\rightarrow\Dcomp{m+s}{\ms{X}};\,P\mapsto
P^{\circ}$ (cf.\ \cite[4.3]{Gar}).

The most important property of this operator is that with which we are
able to describe the isomorphism \cite[4.1.2 (ii)]{Ber2}. Precisely,
there exists the following canonical isomorphism
\begin{equation*}
 F^*F^{\flat}\Dcomp{m}{\ms{X}'}\cong \mc{O}_{\ms{X}}\otimes
  _{\mc{O}_{\ms{X}'}}\Dcomp{m}{\ms{X}'}\otimes_{\mc{O}_{\ms{X}'}}
  \mc{O}_{\ms{X}}^\vee\xrightarrow{\sim}\Dcomp{m+s}{\ms{X}}.
\end{equation*}
Then according to \cite[4.7.2]{Gar}, this isomorphism can be described
as
\begin{equation}
 \label{Garisom}
 f\otimes P\otimes H\underline{x}^{-\underline{k}}\mapsto
  f\cdot P^\circ\cdot H\underline{x}^{-\underline{k}}.
\end{equation}

\begin{rem*}
 Throughout \cite{Gar}, the residue field $k$ is assumed to be perfect
 (cf.\ [{\it loc.\ cit.\ }1.1]). However, this assumption is not used
 in the paper, and this assumption is redundant. In fact, Berthelot is
 stating theorem of Frobenius descent \cite[4.2.4]{Ber2} without posing
 any perfectness assumption.
\end{rem*}

\subsection{}
\label{settingpush}
Let $\ms{X}$ be as in the previous paragraph. Let $\ms{Y}$ be another
smooth affine formal scheme over $R$ possessing a system of local
coordinates $\{y_1,\dots,y_{d'}\}$. We also assume that we have a smooth
lifting $\ms{Y}'$ of the relative Frobenius with a system of local
coordinates $\{y'_1,\dots,y'_{d'}\}$ and morphism
$F_{\ms{Y}}\colon\ms{Y}\rightarrow\ms{Y}'$ such that
$F^*_{\ms{Y}}(y'_i)=y^q_i$ for any $i$. We fix one non-negative integer
$j$ and denote $\ms{X}\otimes R_j$, $\ms{Y}\otimes R_j$, $\ms{X}'\otimes
R_j$, $\ms{Y}'\otimes R_j$ by $X$, $Y$, $X'$, $Y'$ respectively. Suppose
given a morphism of special fibers $f_0\colon X_0\rightarrow
Y_0$. Consider the following diagram.
\begin{equation*}
 \xymatrix{
  X\ar[r]^{F_X}\ar[d]_{f}&X'\ar[d]^{f'}\\
 Y\ar[r]_{F_Y}&Y'}
\end{equation*}
Here $F_X$ and $F_Y$ are reductions of $F_{\ms{X}}$ and $F_{\ms{Y}}$, and
$f$ and $f'$ are liftings of $f_0$ and $f'_0$. In general, we are {\em
not} able to take $f$ and $f'$ so that the diagram is
commutative. However, we can take $F_X$ and $F_Y$ locally with respect
to $X$. To see this, it suffices to treat the case where $f$ is a closed
immersion and  smooth morphism individually, and in both cases, the
verification is straightforward.

Let $\ms{M}$ be a quasi-coherent $\mc{O}_{X'}$-module. Since $X$ and $Y$
are affine schemes, we will identify quasi-coherent sheaves and its
global sections. We list up conventions of identifications used
to describe sections of certain sheaves as follows.
\begin{alignat}{2}
 \label{identifications}
 F_X^\flat(\ms{M})&\cong\ms{M}\otimes_{\mc{O}_{X'}}\mc{O}_X^\vee&\qquad
 \Dmod{m}{Y\rightarrow X}&\cong\mc{O}_X\otimes_{\mc{O}_Y}\Dmod{m}{Y}\\
 \notag
 F^*_X(\ms{M})&\cong\mc{O}_X\otimes_{\mc{O}_{X'}}\ms{M}&\qquad
 \Dmod{m}{Y\leftarrow X}&\cong(\mc{O}_X\otimes_{\mc{O}_Y}(\Dmod{m}{Y}
 \otimes_{\mc{O}_Y}\omega_Y^{-1}))\otimes_{\mc{O}_X}\omega_X
\end{alignat}
For example we have an identification
\begin{equation*}
  {F^{\flat}_X}{F^*_Y}\Dmod{m}{Y'\leftarrow X'}\cong
 \mc{O}_Y\otimes_{\mc{O}_{Y'}}((\mc{O}_{X'}\otimes_{\mc{O}_{Y'}}
 \Dmod{m}{Y'}\otimes_{\mc{O}_{Y'}}\omega_{Y'}^{-1})\otimes_{\mc{O}_X'}
 \omega_{X'})\otimes_{\mc{O}_{X'}}\mc{O}_X^\vee,
\end{equation*}
and for $f\in\mc{O}_Y$, $g\in\mc{O}_{Y'}$, $P\in\Dmod{m}{Y'}$ and
$\underline{0}\leq\underline{k}<\underline{q}$, the section
$f\otimes((g\otimes P\otimes(dy'_{\bullet})^\vee)\otimes dx'
_{\bullet})\otimes H\underline{x}^{-\underline{k}}$ on the right side of
the equality equally means a section on the left side by the
identification.

\begin{prop}
 \label{calcberthisom}
 We preserve the notation of the previous paragraph.
 The inverse of the isomorphism of Berthelot {\normalfont
 \cite[3.4.2 (i)]{Ber2}}
 \begin{equation*}
  \chi\colon{F^{\flat}_X}{F^*_Y}\Dmod{m}{Y'\leftarrow X'}\xrightarrow{\sim}
   \Dmod{m+s}{Y\leftarrow X}
 \end{equation*}
 can be described in the following way using the identification of
 {\normalfont (\ref{identifications})}. Let $f\in\mc{O}_X$, and
 $P\in\Dmod{m}{Y'}$,
 then we get
 \begin{equation*}
  \chi(f\otimes(1\otimes P\otimes(dy'_\bullet)^\vee\otimes
   dx'_{\bullet})\otimes H\underline{x}^{-\underline{k}})=
   x^{\underline{q}-\underline{k}-\underline{1}}\otimes
   (P^\circ\cdot H\underline{y}^{-(\underline{q}-\underline{1})}\cdot f)
   \otimes(dy_\bullet)^\vee\otimes dx_{\bullet}.
 \end{equation*}
\end{prop}
\begin{proof}
 Recall that we are identifying quasi-coherent sheaves and its global
 sections. By using $\nu^{-1}$ in Theorem \ref{maincalc}, we get an
 isomorphism 
 \begin{equation*}
  \mc{O}_Y\otimes_{\mc{O}_{Y'}}(\Dmod{m}{Y'}\otimes\omega_{Y'}
   ^{-1})\rightarrow(\Dmod{m}{Y'}\otimes_{\mc{O}_{Y'}}\mc{O}_Y^{\vee})
   \otimes\omega^{-1}_{Y}.
 \end{equation*}
 The theorem is saying that this sends $f\otimes P\otimes(dy'_{\bullet})
 ^\vee$ to $(P\otimes(H\underline{y}^{-(\underline{q}
 -\underline{1})}\cdot f))\otimes(dy_{\bullet})^\vee$. By using
 $\mu^{-1}$, we get
 \begin{equation*}
  (\ms{D}\otimes_{\mc{O}_{X'}}\omega_{X'})\otimes_{\mc{O}_{X'}}
   \mc{O}_X^\vee\rightarrow(\mc{O}_X\otimes_{\mc{O}_{X'}}
   \ms{D})\otimes\omega_{X},
 \end{equation*}
 where $\ms{D}:=f'^*((\Dmod{m}{Y'}\otimes\mc{O}_Y^{\vee})\otimes
 \omega^{-1}_{Y})$. For a section $D$ of $\ms{D}$, this homomorphism
 sends $(D\otimes dx'_{\bullet})\otimes H\underline{x}^{-\underline{k}}$
 to $(x^{\underline{q}-\underline{k}-\underline{1}}\otimes D)
 \otimes dx_\bullet$. At last, there exists the following isomorphism
 \begin{equation*}
  \mc{O}_X\otimes_{\mc{O}_{X'}} f'^*((\Dmod{m}{Y'}\otimes\mc{O}_Y
   ^{\vee})\otimes\omega^{-1}_{Y})\xrightarrow[(*)]{\sim}
   f^*(\mc{O}_{Y}\otimes(\Dmod{m}{Y'}\otimes\mc{O}_Y^{\vee})
   \otimes\omega^{-1}_{Y})\xrightarrow{\sim}f^*\Dmod{m+s}{Y}
   \otimes\omega_{Y}^{-1}.
 \end{equation*}
 Here the first isomorphism follows from the commutativity of the
 diagram. According to Garnier's calculation, this sends $f\otimes
 1\otimes((Q\otimes H\underline{y}^{-\underline{k}})\otimes
 (dy_{\bullet})^\vee)$ to $f\otimes(Q^{\circ}H\underline{y}
 ^{-\underline{k}})\otimes(dy_{\bullet})^\vee$. Combining these, we get
 the proposition.
\end{proof}

\begin{rem*}
 We will describe shortly the way to calculate $\chi$ when the diagram
 in paragraph \ref{settingpush} is not commutative. In this case,
 suppose the integer $m$ satisfies the inequality $p^m>e/(p-1)$ where
 $e$ was the absolute ramification index of $R$. Under this condition,
 we may use the Taylor isomorphism of \cite[2.1.5]{Ber2} to compare
 $F_Y\circ f$ and $f'\circ F_X$. This isomorphism can be described in
 the following way. Let $f,f'\colon X\rightarrow Y$ be two morphisms of
 smooth schemes whose reductions over $k$ are the same morphisms, and
 suppose that $Y$ possesses a system of local coordinates
 $\{y_1,\dots,y_d\}$. Let
 \begin{equation*}
  T:=\sum_{\underline{k}\geq0}(f'^*(\underline{y})-f^*(\underline{y}))
   ^{\{\underline{k}\}_{(m)}}\otimes\partial^{\angles{m}{\underline{k}}}
 \end{equation*}
 in $f^*\Dmod{m}{Y}$. This is defined since we have the assumption on
 $m$. Let $\ms{M}$ be a $\Dmod{m}{Y}$-module. Then we can check that the
 isomorphism $\tau_{f,f'}\colon f'^*\ms{M}\xrightarrow{\sim}f^*\ms{M}$
 sends $1\otimes m$ to $T\otimes m$ for a section $m$ of $\ms{M}$.

 Now, when the diagram is not commutative, the calculation of $\chi$
 goes well exactly in the same way except for $(*)$ in the
 proof of the proposition. We use this calculation of Taylor isomorphism
 to compute $(*)$.
\end{rem*}

\subsection{}
\label{calcpushexp}
We use the same notation as paragraph \ref{settingpush}. As the first
application of Proposition \ref{calcberthisom}, we will calculate the
isomorphism
\begin{equation}
 \label{puchcommisom}
 F_Y^*f'_+\ms{M}\xrightarrow{\sim}f_+F_X^*\ms{M}
\end{equation}
concretely. This result is used in \cite{AM} to calculate the
Frobenius structure of geometric Fourier transform defined by
Noot-Huyghe explicitly.

Let $\ms{M}$ be a $\Dmod{m}{X'}$-module. In the proof of
\cite[3.4.4]{Ber2}, the isomorphism
\begin{equation*}
 \xi\colon
  F^*_Y\,(\Dmod{m}{Y'\leftarrow X'}\otimes^{\mb{L}}_{\Dmod{m}{X'}}
  \ms{M}):=f^{-1}\mc{O}_{Y}\otimes_{f'^{-1}\mc{O}_{Y'}}
  (\Dmod{m}{Y'\leftarrow X'}\otimes^{\mb{L}}_{\Dmod{m}{X'}}\ms{M})
  \xrightarrow{\sim}(\Dmod{m+s}{Y\leftarrow X}\otimes^{\mb{L}}
  _{\Dmod{m+s}{X}}F^*_X\ms{M})
\end{equation*}
is defined. Using the projection formula, the isomorphism
(\ref{puchcommisom}) is nothing but $Rf_*(\xi)$.

\begin{prop*}
 Let $\xi_0:=\ms{H}_0(\xi)$.
 For $P\in\Dmod{m}{Y'}$ and $m\in\ms{M}$, we have
 \begin{equation*}
  \xi_0\,\bigl(\underline{y}^{\underline{l}}\otimes(1\otimes
   P\otimes(dy'_\bullet)^\vee\otimes dx'_{\bullet})\otimes m\bigr)=
   \bigl(1\otimes(P^\circ\cdot H\underline{y}^{-(\underline{q}-
   \underline{1})}\cdot\underline{y}^{\underline{l}})\otimes
   (dy_\bullet)^\vee\otimes dx_{\bullet}\bigr)
   \otimes(\underline{x}^{\underline{q}-\underline{1}}\otimes m).
 \end{equation*}
\end{prop*}
\begin{proof}
 Let us review the definition of $\xi_0$.
 Tensoring both sides of $\xi_0$ with $f^{-1}F^{\flat}_Y\Dmod{m}{Y'}$,
 it is equivalent to defining an isomorphism
 \begin{equation}
  \label{coreisopush}
   \Dmod{m}{Y'\leftarrow X'}\otimes_{\Dmod{m}{X'}}\ms{M}
   \xrightarrow{\sim}f^{-1}F_Y^{\flat}\Dmod{m}{Y'}\otimes
   _{f^{-1}\Dmod{m+s}{Y}}(\Dmod{m+s}{Y\leftarrow X}\otimes
   _{\Dmod{m+s}{X}}F^*_X\ms{M})
 \end{equation}
 by \cite[2.5.6]{Ber2}.
 We get an isomorphism
 \begin{equation*}
 \Dmod{m}{Y'\leftarrow X'}\otimes_{\Dmod{m}{X'}}\ms{M}\xrightarrow{\sim}
  F_X^{\flat}\Dmod{m}{Y'\leftarrow X'}\otimes_{\Dmod{m+s}{X}}
  F^*_X\ms{M}\xrightarrow{\sim}f^{-1}F_Y^{\flat}\Dmod{m}{Y'}\otimes
  F^{*}_YF_X^{\flat}\Dmod{m}{Y'\leftarrow X'}\otimes F^*_X\ms{M}
 \end{equation*}
 where the first isomorphism is by \cite[2.5.7]{Ber2} and the second by
 \cite[2.5.6]{Ber2}. Combining this isomorphism with $\chi$, we get the
 isomorphism (\ref{coreisopush}). By Proposition \ref{calcberthisom}, we
 see that (\ref{coreisopush}) sends
 $(1\otimes P\otimes(dy'_\bullet)^\vee\otimes dx'_{\bullet})\otimes m$
 to
 \begin{equation*}
 (1\otimes H\underline{y}^{-\underline{l}})\otimes
  \bigl(x^{\underline{q}-\underline{k}-\underline{1}}\otimes
  (P^\circ\cdot H\underline{y}^{-(\underline{q}-\underline{1})}
  \cdot\underline{y}^{\underline{l}})\otimes(dy_\bullet)^\vee
  \otimes dx_{\bullet}\bigr)
  \otimes(\underline{x}^{\underline{k}}\otimes m)
 \end{equation*}
 where $\underline{k}$ and $\underline{l}$ denote any element in
 $\mb{N}^d$ such that $\leq\underline{q}-\underline{1}$, and the
 proposition follows.
\end{proof}

\subsection{}
We consider the situation \ref{Frobsituations}.\ref{noFrobstr}. As
another application of the explicit description, we will show a
proper base change type result. For the most familiar statement, see
paragraph \ref{properbc}.
Let $m$ be an integer such that $p^m>e/(p-1)$ and $i$ be a non-negative
integer. Consider the following cartesian diagram of smooth schemes over
$R_i$
\begin{equation*}
 \xymatrix{
  X'\ar[r]^{g'}\ar[d]_{f'}\ar[rd]^h&X\ar[d]^f\\
 Y'\ar[r]_g&Y}
\end{equation*}
where $h=g\circ f'$. Let $Z$ be one of $X$, $X'$, $Y$, $Y'$.
We denote by $Z_0$ the reduction of $Z$ over $k$. We assume that
$Z^{(s)}_0$ possesses a smooth lifting $\widetilde{Z}$ over $R_i$.

\begin{lem*}
 \label{bcthmsch}
 We assume that $g$ is smooth. Then there is a canonical equivalence of
 functors
 \begin{equation}
 \label{propbcisom}
  g^!\circ f_+\cong f'_+\circ g'^!\colon D^b_{\mr{qc}}(\Dmod{m}{X})
  \rightarrow D^b_{\mr{qc}}(\Dmod{m}{Y'}).
 \end{equation}
 This equivalence is compatible with raising levels and Frobenius
 pull-backs.
\end{lem*}
\begin{proof}
 Let us first construct the following homomorphism:
 \begin{equation}
  \label{desiredhomten}
   f'^{-1}\Dmod{m}{Y'\rightarrow Y}\otimes^{\mb{L}}_{h^{-1}
   \Dmod{m}{Y}}g'^{-1}\Dmod{m}{Y\leftarrow X}\rightarrow
   \Dmod{m}{Y'\leftarrow X'}\otimes^{\mb{L}}_{\Dmod{m}{X'}}
   \Dmod{m}{X'\rightarrow X}.
 \end{equation}
 For short, we denote $\Dmod{m}{}$ by $\ms{D}$. There exists a canonical
 homomorphism of rings $g^{-1}\ms{D}_{Y}\rightarrow\ms{D}_{Y'}$. Thus,
 we get a homomorphism
 $h^{-1}\ms{D}_Y\otimes_{h^{-1}\mc{O}_Y}\omega_{X'/Y'}\xrightarrow{\alpha}
 f'^{-1}\ms{D}_{Y'}\otimes_{\mc{O}_{Y'}}\omega_{X'/Y'}$ of
 $(f'^{-1}\ms{D}_{Y'},g'^{-1}\ms{D}_X)$-bimodules where
 $\omega_{X'/Y'}:=\omega_{X'}\otimes f'^{-1}\omega_{Y'}$. We have the
 canonical section $1\otimes1$ in $\ms{D}_{X'\rightarrow
 X}\cong\mc{O}_{X'}\otimes_{f^{-1}\mc{O}_X}f^{-1}\ms{D}_X$. This
 section induces a homomorphism
 $f'^{-1}\ms{D}_{Y'}\otimes_{\mc{O}_{Y'}}\omega_{X'/Y'}
 \xrightarrow{\beta}(f'^{-1}\ms{D}_{Y'}\otimes_{\mc{O}_{Y'}}
 \omega_{X'/Y'})\otimes_{\ms{D}_{X'}}\ms{D}_{X'\rightarrow X}$. Combining
 these, we get the homomorphism
 \begin{equation*}
 \beta\circ\alpha\colon
 h^{-1}\ms{D}_Y\otimes_{h^{-1}\mc{O}_Y}\omega_{X'/Y'}\rightarrow
  (f'^{-1}\ms{D}_{Y'}\otimes_{\mc{O}_{Y'}}
 \omega_{X'/Y'})\otimes_{\ms{D}_{X'}}\ms{D}_{X'\rightarrow X}.
 \end{equation*}
 This induces the homomorphism
 \begin{equation}
  \label{nodevtens}
 f'^{-1}\ms{D}_{Y'\rightarrow Y}\otimes_{h^{-1}
  \ms{D}_{Y}}g'^{-1}\ms{D}_{Y\leftarrow X}
  \rightarrow\ms{D}_{Y'\leftarrow X'}
  \otimes_{\ms{D}_{X'}}\ms{D}_{X'\rightarrow X}.
 \end{equation}
 Now, since $f'^{-1}\ms{D}_{Y'\rightarrow Y}$ is a flat right
 $h^{-1}\ms{D}_Y$-module, we get a quasi-isomorphism
 \begin{equation}
  \label{comp1}
  f'^{-1}\ms{D}_{Y'\rightarrow Y}\otimes^{\mb{L}}_{h^{-1}
   \ms{D}_{Y}}g'^{-1}\ms{D}_{Y\leftarrow X}\xrightarrow{\sim}
   f'^{-1}\ms{D}_{Y'\rightarrow Y}\otimes_{h^{-1}\ms{D}_{Y}}
   g'^{-1}\ms{D}_{Y\leftarrow X}.
 \end{equation}
 From now on, we omit $f'^{-1}$ and so on, but we consider sheaves are
 on $X'$. Let us show that
 \begin{equation}
  \label{comp2}
  \Dmod{m}{Y'\leftarrow X'}
  \otimes^{\mb{L}}_{\Dmod{m}{X'}}\Dmod{m}{X'\rightarrow X}
  \xrightarrow{\sim}\Dmod{m}{Y'\leftarrow X'}
  \otimes_{\Dmod{m}{X'}}\Dmod{m}{X'\rightarrow X},
 \end{equation}
 or in other words $\ms{H}^i(  \Dmod{m}{Y'\leftarrow X'}
  \otimes^{\mb{L}}_{\Dmod{m}{X'}}\Dmod{m}{X'\rightarrow X})=0$ for
 $i\neq0$.
 When $m=0$, the argument is standard using the Spencer resolution
 $\Dmod{0}{X'}\otimes\Theta_{X'/X}^{\bullet}$ of
 $\Dmod{0}{X'\rightarrow X}$ where $\Theta_{X'/X}$ denotes the relative
 tangent bundle of $X'\rightarrow X$ (cf.\ \cite[4.3.1]{Ber2}). Let us
 see the general case. Since the verification is local, we may assume
 that the schemes are affine and $s=m$. We get
 \begin{equation*}
  \Dmod{m}{Y'\leftarrow X'}
  \otimes^{\mb{L}}_{\Dmod{m}{X'}}\Dmod{m}{X'\rightarrow X}\cong
  F^*_{Y'}F^{\flat}_X\bigl(\Dmod{0}{\widetilde{Y}'\leftarrow\widetilde{X}'}
  \otimes^{\mb{L}}_{\Dmod{0}{\widetilde{X}'}}\Dmod{0}{\widetilde{X}'
  \rightarrow\widetilde{X}}\bigr)
 \end{equation*}
 by \cite[2.5.6, 3.4.2]{Ber2}.
 This reduces the verification to the $m=0$ case, and the claim
 follows. Combining (\ref{nodevtens}), (\ref{comp1}), (\ref{comp2}), we
 get the desired homomorphism (\ref{desiredhomten}).

 By construction, (\ref{desiredhomten}) is compatible with raising
 levels. Let us show the compatibility of Frobenius.
 We denote by relative Frobenius morphisms of the special fiber $X_0$
 {\it etc.\ }by $F_X$ {\it etc}. Let us show that the following diagram
 is commutative where homomorphisms are isomorphisms.
 \begin{equation}
  \label{pbcfrobcomp}
   \xymatrix{
   F_{Y'}^*\bigl(\Dmod{m}{\widetilde{Y}'\rightarrow\widetilde{Y}}
   \otimes^{\mb{L}}_{\Dmod{m}{\widetilde{Y}}}\Dmod{m}
   {\widetilde{Y}\leftarrow\widetilde{X}}\otimes^{\mb{L}}
   _{\Dmod{m}{\widetilde{X}}}F_X^\flat\Dmod{m}{\widetilde{X}}
   \bigr)\ar[r]\ar[d]
   \ar@/^1pc/[rd]^<>(.5)\alpha
   \ar@/_1pc/[rd]_<>(.5)\beta&
   F^*_{Y'}\bigl(\Dmod{m}{\widetilde{Y}'\leftarrow\widetilde{X}'}
   \otimes^{\mb{L}}_{\Dmod{m}{\widetilde{X}'}}\Dmod{m}
   {\widetilde{X}'\rightarrow\widetilde{X}}
   \otimes^{\mb{L}}_{\Dmod{m}{\widetilde{X}}}
   F_X^\flat\Dmod{m}{\widetilde{X}}\bigr)\ar[d]\\
   \bigl(\Dmod{m+s}{Y'\rightarrow Y}\otimes^{\mb{L}}_{\Dmod{m+s}{Y}}
   \Dmod{m+s}{Y\leftarrow X}\bigr)\otimes^{\mb{L}}_{\Dmod{m+s}{X}}
   F^*_XF^\flat_X\Dmod{m}{\widetilde{X}}\ar[r]&
   \bigl(\Dmod{m+s}{Y'\leftarrow X'}\otimes^{\mb{L}}_{\Dmod{m+s}{X'}}
   \Dmod{m+s}{X'\rightarrow X}\bigr)\otimes^{\mb{L}}_{\Dmod{m+s}{X}}
   F^*_XF^\flat_X\Dmod{m}{\widetilde{X}}
   }
 \end{equation}
 Let $\ms{M}$ be a left $\Dmod{m}{\widetilde{X}}$-complex. By taking
 $\otimes^{\mb{L}}_{\Dmod{m+s}{X}}F^*_X\ms{M}$, (\ref{pbcfrobcomp})
 induces the following commutative diagram:
 \begin{equation*}
  \xymatrix{
   F_{Y'}^*\bigl(\Dmod{m}{\widetilde{Y}'\rightarrow\widetilde{Y}}
   \otimes^{\mb{L}}_{\Dmod{m}{\widetilde{Y}}}\Dmod{m}
   {\widetilde{Y}\leftarrow\widetilde{X}}\otimes^{\mb{L}}
   _{\Dmod{m}{\widetilde{X}}}\ms{M}
   \bigr)\ar[r]\ar[d]&
   F^*_{Y'}\bigl(\Dmod{m}{\widetilde{Y}'\leftarrow\widetilde{X}'}
   \otimes^{\mb{L}}_{\Dmod{m}{\widetilde{X}'}}\Dmod{m}
   {\widetilde{X}'\rightarrow\widetilde{X}}
   \otimes^{\mb{L}}_{\Dmod{m}{\widetilde{X}}}
   \ms{M}\bigr)\ar[d]\\
   \bigl(\Dmod{m+s}{Y'\rightarrow Y}\otimes^{\mb{L}}_{\Dmod{m+s}{Y}}
   \Dmod{m+s}{Y\leftarrow X}\bigr)\otimes^{\mb{L}}_{\Dmod{m+s}{X}}
   F^*_X\ms{M}\ar[r]&
   \bigl(\Dmod{m+s}{Y'\leftarrow X'}\otimes^{\mb{L}}_{\Dmod{m+s}{X'}}
   \Dmod{m+s}{X'\rightarrow X}\bigr)\otimes^{\mb{L}}_{\Dmod{m+s}{X}}
   F^*_X\ms{M},
   }
 \end{equation*}
 which implies the compatibility of Frobenius.

 Let us prove the commutativity of (\ref{pbcfrobcomp}). Since these
 complexes are concentrated in degree $0$, the problem is local on
 $X'$. Thus, we may assume that any scheme possesses a system of local
 coordinates, and the following diagram is commutative.
 \begin{equation*}
  \xymatrix@R=10pt{
   &\widetilde{X}'\ar[rr]\ar[dd]&&\widetilde{X}\ar[dd]\\
  X'\ar[ru]^{F_{X'}}\ar[rr]|<>(.7){g'}\ar[dd]|<>(.3){f'}&&
   X\ar[ru]_{F_X}\ar[dd]|<>(.3)f&\\
  &\widetilde{Y}'\ar[rr]&&\widetilde{Y}\\
  Y'\ar[ru]^{F_{Y'}}\ar[rr]|<>(.7){g}&&Y\ar[ru]_{F_Y}&}
 \end{equation*}
 Let $\{\widetilde{t}_i\}_{1\leq i\leq d_X}$ (resp.\ $\{t_i\}_{1\leq
 i\leq d_X}$, $\{\widetilde{s}_j\}_{1\leq j\leq d_Y}$, $\{s_j\}_{1\leq
 j\leq d_Y}$) be a system of local coordinates of $\widetilde{X}$
 (resp. $X$, $\widetilde{Y}$, $Y$) such that
 $F_X^*(\widetilde{t_i})=t^q_i$, $F_Y^*(\widetilde{s_j})=s^q_j$.
 As usual, we use the notation $dt_\bullet:=dt_1\wedge\dots\wedge
 dt_{d_X}$ and so on. Let $H_{s_j}$ be the Dwork operator (cf.\
 subsection \ref{Garnierrev}) in $\Dmod{m+s}{Y'}$ corresponding to
 $s_j$, and we put $H_s:=\prod_{j=1}^{d_Y}H_{s_j}$.
 Let $x_1,\dots,x_d$ be a local coordinate of $Y'$ over
 $Y$. This can be seen also as a local coordinate of $X'$ over $X$. We
 denote by $H_i$ the Dwork operator in $\Dmod{m+s}{X'}$ (resp.\
 $\Dmod{m+s}{Y'}$) corresponding to $x_i$. We put
 $H_x:=\prod_{i=1}^dH_i$. Using the convention of
 (\ref{identifications}), let
 \begin{equation*}
  \Xi:=1\otimes\bigl((1\otimes P)\otimes(1\otimes1\otimes(d\widetilde{t}
   _\bullet)^\vee\otimes d\widetilde{s}_\bullet)\otimes(Q\otimes\phi)
   \bigr)
 \end{equation*}
 be a section of $F_{Y'}^*\bigl(\Dmod{m}{\widetilde{Y}'\rightarrow
 \widetilde{Y}}\otimes_{\Dmod{m}{\widetilde{Y}}}\Dmod{m}
 {\widetilde{Y}\leftarrow\widetilde{X}}\otimes_{\Dmod{m}{\widetilde{X}}}
 F_X^\flat\Dmod{m}{\widetilde{X}}\bigr)$. To see the commutativity of
 (\ref{pbcfrobcomp}), it suffices to show $\alpha(\Xi)=\beta(\Xi)$.  Let
 $S:=\prod_{1\leq i\leq d}x^{q-1}_iH_ix_i^{-(q-1)}$. By using
 Proposition \ref{calcpushexp},
 \begin{align*}
  \alpha(\Xi)&=\bigl(1\otimes P^\circ\cdot H_s\underline{s}^{-(
  \underline{q}-\underline{1})}\cdot H_x\underline{x}^{-(\underline{q}
  -\underline{1})}\cdot f\otimes(dt_{\bullet})^\vee\otimes ds_\bullet
  \bigr)\otimes(\underline{s}^{\underline{q}-\underline{1}}\,
  \underline{x}^{\underline{q}-\underline{1}}\otimes1)
  \otimes(Q\otimes\phi)\\
  &=\bigl(1\otimes P^\circ\cdot H_s\underline{s}^{-(
  \underline{q}-\underline{1})}\cdot f\otimes(dt_{\bullet})^\vee\otimes
  ds_\bullet\bigr)\otimes{^tS}\cdot(\underline{s}^{\underline{q}
  -\underline{1}}\otimes1)\otimes(Q\otimes\phi)\\
  \beta(\Xi)&=\bigl(1\otimes P^\circ\cdot H_s\underline{s}^{-(
  \underline{q}-\underline{1})}\cdot f\otimes(dt_{\bullet})^\vee\otimes
  ds_\bullet\bigr)\otimes(\underline{s}^{\underline{q}-\underline{1}}
  \otimes1)\otimes(Q\otimes\phi)
 \end{align*}
 To show that these two quantities are equal, it suffices to see that
 the image of ${^tS}$ by the homomorphism $\Dmod{m+s}{X'}\rightarrow
 \Dmod{m+s}{X'\rightarrow X}$ is $1\otimes1$. To show this, it
 suffices to see that ${^tS}(1)=1$ in $\mc{O}_{X'}$. Since the claim is
 stable under base change, we may assume that $\zeta\in R_i$. By
 definition of $H$, it suffices to show that
 \begin{equation*}
  x_i^{-(q-1)}q^{-1}\sum_{\zeta^q=1}\sum_{k\geq 0}(\zeta-1)^k\,
   \partial_i^{[k]}\,(x_i^{k+q-1})=1
 \end{equation*}
 for any $1\leq i\leq d$. The sum is equal to
 \begin{equation*}
  q^{-1}\sum_{\zeta^q=1}\sum_k(\zeta-1)^k\binom{k+q-1}{k}=
   q^{-1}\sum_{\zeta^q=1}\sum_k(1-\zeta)^k\binom{-q}{k}=
   q^{-1}\sum_{\zeta^q=1}\zeta^{-q}=1,
 \end{equation*}
 and the commutativity of (\ref{pbcfrobcomp}) follows.

 Finally, let us complete the proof of (\ref{propbcisom}). We have
 \begin{align*}
 g^!\circ f_+(\ms{M})&:=\ms{D}_{Y'\rightarrow Y}
 \otimes^{\mb{L}}_{g^{-1}\ms{D}_Y}g^{-1}Rf_*\ms{D}_{Y\leftarrow X}
 \otimes_{\ms{D}_X}^{\mb{L}}\ms{M}\\
 &\cong Rf'_*\bigl(f'^{-1}\ms{D}_{Y'\rightarrow Y}\otimes^{\mb{L}}
 _{h^{-1}\ms{D}_Y}g'^{-1}(\ms{D}_{Y\leftarrow X}
 \otimes_{\ms{D}_X}^{\mb{L}}\ms{M})\bigr)\\
 &\xrightarrow{\sim}Rf'_*\bigl(\ms{D}_{Y'\leftarrow X'}
  \otimes^{\mb{L}}_{\ms{D}_{X'}}\ms{D}_{X'\rightarrow X}\otimes^{\mb{L}}
 _{g'^{-1}\ms{D}_X}g'^{-1}\ms{M}\bigr)
 \cong f'_+\circ g'^!(\ms{M}),
 \end{align*}
 where we used the flat base change in the first isomorphism.
Thus, we get the isomorphism.
\end{proof}

\subsection{}
\label{dfntatetwist}
Consider the situation in \ref{Frobsituations}.\ref{noFrobstr}.
Let us define the Tate twist (cf.\ \cite[2.3.8 (i)]{BerPre}). Let
$\ms{X}$ and $\ms{Y}$ are two smooth
formal schemes, and suppose there exist smooth liftings $\ms{X}'$ and
$\ms{Y}'$ of $X^{(s)}_0$ and $Y^{(s)}_0$ where $X_0$ and $Y_0$ are
special fibers of $\ms{X}$ and $\ms{Y}$ as usual. Let $*$ be one of
$\ms{X}$, $\ms{X}'$, $\ms{Y}$, $\ms{Y}'$, and $\mc{A}(*)$ be either
$D^b_{\mr{coh}}(\DdagQ{*})$ or $\LD{*}$. Let
\begin{equation*}
 G\colon\mc{A}(\ms{X})\rightarrow\mc{A}(\ms{Y}),\qquad
  G'\colon\mc{A}(\ms{X}')\rightarrow\mc{A}(\ms{Y}')
\end{equation*}
be $\mb{Q}$-linear functors. Now, suppose given a
equivalence of functors
\begin{equation*}
 \Psi\colon F^*_{\ms{Y}}\circ G'\xrightarrow{\sim}G\circ F^*_{\ms{X}}.
\end{equation*}
The triple $(G,G',\Psi)$ (we often abbreviate this as $(G,\Psi)$ or even
$G$ if no confusion can arise) is called a {\em cohomological functor
with Frobenius isomorphism}.
The natural transform $\Psi$ is called the {\em Frobenius isomorphism}
of the triple. Given $(G,G',\Psi)$ and an integer $d$, we define its
{\em Tate twist} $\Psi(d)$ of the Frobenius isomorphism by
$\Psi(d):=q^{-d}\cdot\Psi$. We often denote by $G(d)$ the triple
$(G,G',\Psi(d))$ for simplicity.

Now, we consider the situation \ref{Frobsituations}.\ref{Frobstr}. For a
$\DdagQ{\ms{X}}$-module $\ms{M}$, we denote a $\DdagQ{\ms{X}}$-module by
$\ms{M}^\sigma$ the base change of $\ms{M}$ by using $\sigma$.
Let $\ms{M}$ be an $F$-$\DdagQ{\ms{X}}$-module (resp.\ complex). By
definition, this is a $\DdagQ{\ms{X}}$-module (resp.\ complex) equipped
with an isomorphism\footnote{
The definition of Frobenius structure here
is slightly different from that of \cite[4.5.1]{Ber2} in the sense that
in {\it loc.\ cit.}, $\Phi$ is an isomorphism $\ms{M}\xrightarrow{\sim}
F^*\ms{M}^{\sigma}$. Since $\Phi$ is an isomorphism, it causes no
difference. We adopted our definition to make it easier to see the
compatibility with the definition of Frobenius structure of
$F$-isocrystals. See also [{\it loc.\ cit.}, Remarque 4.5.1].}
$\Phi\colon F^*\ms{M}^{\sigma}\xrightarrow{\sim}\ms{M}$. For any
integer $d$, we define an
$F$-$\DdagQ{\ms{X}}$-module (resp.\ complex) $\ms{M}(d)$ called the Tate
twist of $\ms{M}$ in the following way. The underlying
$\DdagQ{\ms{X}}$-module (resp.\ complex) is the same as that of
$\ms{M}$. We denote by $\Phi'$ the isomorphism
$F^*\ms{M}(d)^\sigma\rightarrow\ms{M}(d)$ induced by the Frobenius
structure of $\ms{M}$. The Frobenius structure
$\Phi(d)\colon F^*\ms{M}(d)^\sigma\rightarrow\ms{M}(d)$ of $\ms{M}(d)$
is by definition $q^{-d}\Phi'$. Now, let
$(G,\Psi)$ be a cohomological functor with Frobenius isomorphism. Then
we get that $G(\ms{M})$ is naturally equipped with Frobenius structure,
and we get for any integer $d$ that $G(d)(\ms{M})\cong G(\ms{M})(d)\cong
G(\ms{M}(d))$.

Let $(\ms{M},\Phi)$ and $(\ms{N},\Psi)$ be two
$F$-$\DdagQ{\ms{X}}$-modules, and $\varphi\colon\ms{M}\rightarrow\ms{N}$
be a homomorphism of $\DdagQ{\ms{X}}$-modules (where we do not consider
the Frobenius structures). Consider the following diagrams where the
left diagram is that of modules (or sheaves of modules):
\begin{equation*}
 \xymatrix{
  M_1\ar[r]^{\alpha}\ar[d]_{\gamma}&M_2\ar[d]^{\beta}\\
 M_3\ar[r]_{\delta}&M_4,}\hspace{3cm}
  \xymatrix@C=50pt{
  F^*\ms{M}^\sigma\ar[r]^{F^*\varphi^\sigma}\ar[d]_{\Phi}&
  F^*\ms{N}^\sigma\ar[d]^{\Psi}\\\ms{M}\ar[r]_{\varphi}&\ms{N}.}
\end{equation*}
First, pay attention to the left diagram. Let $n$ be a rational
number. We say that the diagram is {\em commutative up to multiplication
by $n$} if $n\cdot(\beta\circ\alpha)=\delta\circ\gamma$ holds. Now,
changing the attention to the right diagram, suppose that the diagram
is commutative up to multiplication by $q^d$. Then we get that $\varphi$
defines a homomorphism $\ms{M}(d)\rightarrow\ms{N}$ as
$F$-$\DdagQ{\ms{X}}$-modules.

\section{Extraordinary pull-back and duality}
In this section, we prove a commutation result of the extraordinary
pull-back functor and the duality functor. The result can be seen as
a part of a ``Poincar\'{e} duality'' in the theory of arithmetic
$\ms{D}$-modules. For the explanation of this interpretation, see \S
\ref{sixfunctor}. By applying the commutation result, we get
\begin{equation*}
 \mb{D}_{\ms{X},Z}(\mc{O}_{\ms{X},\mb{Q}}(^\dag Z))\cong
  \mc{O}_{\ms{X},\mb{Q}}(^\dag Z)(-d)
\end{equation*}
where $\ms{X}$ is a smooth formal scheme, $Z$ is a divisor of its
special fiber, and $d$ is the dimension of $\ms{X}$. Combining this
result with a result of Caro, we are able to compare duality functors of
arithmetic $\ms{D}$-modules and that of overconvergent isocrystals {\em
with Frobenius structures} in a precisely way. At the last part of this
section, we compare the rigid cohomologies and the push-forwards in the
theory of arithmetic $\ms{D}$-modules.

\subsection{}
\label{d-couplesdef}
We consider the situation \ref{Frobsituations}.\ref{noFrobstr}. Let
$\ms{X}$ be a smooth formal
scheme, and $Z$ be a divisor of its special fiber $X_0$. In this
situation, we say that $(\ms{X},Z)$ is a
{\it d-couple}\footnote{This ``d'' stands for divisor.}. Let
$(\ms{Y},W)$ be another d-couple. A morphism of d-couples
$f\colon(\ms{X},Z)\rightarrow(\ms{Y},W)$ is
a morphism of special fibers $f_0\colon X_0\rightarrow Y_0$ such that
$f(X_0\setminus Z)\subset Y_0\setminus W$, and $f^{-1}(W)$ is a
divisor. A {\em strict morphism} of d-couples $f$ is a morphism
$\widetilde{f}\colon\ms{X}\rightarrow\ms{Y}$ whose reduction on the
special fiber is a morphism of d-couples. We say that the morphism
$f_0$ (resp.\ $\widetilde{f}$) is the {\em realization} of the (resp.\
strict) morphism of d-couples $f$.

For a d-couple $(\ms{X},Z)$, let us review the definition of the dual
functor
\begin{equation*}
 \mb{D}_{\ms{X},Z}\colon D^b_{\mr{perf}}(\DdagQ{\ms{X}}(^\dag
  Z))\rightarrow D^b_{\mr{perf}}(\DdagQ{\ms{X}}(^\dag Z)).
\end{equation*}
We note that there exists the canonical equivalence of categories
\begin{equation*}
 D^b_{\mr{perf}}(\DdagQ{\ms{X}}(^\dag Z))\xrightarrow{\sim}
  D^b_{\mr{coh}}(\DdagQ{\ms{X}}(^\dag Z))
\end{equation*}
by \cite[3.2.3]{NH2}. Let $d$ be the dimension of $\ms{X}$. For a
perfect $\DdagQ{\ms{X}}(^\dag Z)$-complex $\ms{C}$, we define the
functor $\mb{D}_{\ms{X},Z}$ by
\begin{equation*}
 \mb{D}_{\ms{X},Z}(\ms{C}):=R\shom_{\DdagQ{\ms{X}}(^\dag Z)}(\ms{C},
  \DdagQ{\ms{X}}(^\dag Z))\otimes_{\mc{O}_{\ms{X}}}
  \omega^{-1}_{\ms{X}}~[d].
\end{equation*}
For fundamental properties of this functor, see \cite{Vir}. Here, we only
note that this functor commutes with Frobenius pull-backs, and induces an
equivalence between the derived categories of perfect complexes (with or
without Frobenius structure).

Let $(\ms{X},Z)$ and $(\ms{Y},W)$ be d-couples, and let
$f_0\colon X_0\rightarrow Y_0$ be a morphism such that
$f_0(X_0\setminus Z)\subset Y_0\setminus W$. Assume that $f_0$ is
smooth. Then $f_0^{-1}(W)$ is a divisor, and in particular, $f_0$
induces a morphism of d-couples
$f\colon(\ms{X},Z)\rightarrow(\ms{Y}.W)$. The functor
$f_0^!\colon D^b_{\mr{coh}}(\DcompQ{m}{\ms{Y}}(W))\rightarrow
D^b_{\mr{coh}}(\DcompQ{m}{\ms{X}}(f^{-1}(W)))$ is defined in \cite[3.2.3
(ii)]{Ber2} and \cite[3.4.6]{BerInt}. By taking the inductive limit as
\cite[4.3.3]{BerInt}, we have the functor
$f_0^!\colon D^b_{\mr{coh}}(\DdagQ{\ms{Y}}(^\dag W))\rightarrow
D^b_{\mr{coh}}(\DdagQ{\ms{X}}(^\dag f^{-1}(W)))$. We define the functor
\begin{equation*}
 f^!\colon D^b_{\mr{coh}}(\DcompQ{m}{\ms{Y}}(W))\rightarrow
  D^b_{\mr{coh}}(\DcompQ{m}{\ms{X}}(Z))
\end{equation*}
by $(Z)\circ f_0^!$, where $(Z)$ denotes the functor
tensoring with $\DcompQ{m}{\ms{X}}(Z)$. By taking the inductive limit,
we also get a functor $f^!\colon D^b_{\mr{coh}}(\DdagQ{\ms{Y}}(^\dag
W))\rightarrow D^b_{\mr{coh}}(\DdagQ{\ms{X}}(^\dag Z))$.

\subsection{}
Let $m$ be an integer such that $p^m>e/(p-1)$ (cf.\ \cite[A.4]{Ber2}).
We denote $\Bmod{m}{\ms{X}}(Z)$ and
$\Bcomp{m}{\ms{X}}(Z)$ (cf.\ \cite[4.2.4]{Ber1})
by $\Bmod{m}{\ms{X}}$ and $\Bcomp{m}{\ms{X}}$,
$\Bcomp{m}{\ms{X}}\widehat{\otimes}\Dcomp{m}{\ms{X}}$ by
$\Dtild{m}{\ms{X}}$, and
$\indlim_m\Dtild{m}{\ms{X}}$($=\Ddag{\ms{X}}(^\dag Z)$) by
$\widetilde{\ms{D}}^\dag_{\ms{X}}$.
We put $\widetilde{\omega}_{\ms{X}}:=\Bcomp{m}{\ms{X}}\widehat{\otimes}
_{\mc{O}_{\ms{X}}}\omega_{\ms{X}}$. We denote by
$D^*_{\mr{qc}}(\Dtild{m}{\ms{X}})$ ($*\in\{-,b\}$) the full
subcategory of $D^*(\Dcomp{m}{\ms{X}})$ consisting of quasi-coherent
complexes (cf.\ \cite[3.2.1]{BerInt}), and
$D^*_{\mb{Q},\mr{qc}}(\Dtild{m}{\ms{X}})$ by the category obtained by
localizing $D^*_{\mr{qc}}(\Dtild{m}{\ms{X}})$ with respect to isogenies
(cf.\ \cite[3.3.2]{BerInt}). Finally, we denote by $D^*(\ms{X})$
($*\in\{+,-,b\}$) the derived category of $R$-modules on $\ms{X}$.

\begin{lem*}
 \label{tensboundedlem}
 Let $\ms{X}$ be a smooth formal scheme. Let $\ms{M}$ be
 a complex in $D^b_{\mr{perf}}(\DtildQ{m}{\ms{X}})^{\mr{g}}$, and
 $\ms{N}$ be a complex in $\DbqcQ(\Dtild{m}{\ms{X}})^{\mr{g}}$. Then the
 complex $\ms{M}\widehat{\otimes}^{\mb{L}}_{\BcompQ{m}{\ms{X}}}\ms{N}$
 is bounded.
\end{lem*}
\begin{proof}
 We will use the notation of SGA6 Exp.\ I.\:
 Since $D^b_{\mr{perf}}\subset D_{\mr{ftd}}$ by \ref{perfcatdfn}, we may
 assume that $\pa(\ms{M})\subset[0,a]$. Let $n$ be an integer such that
 $\ms{H}^i(\ms{N})=0$ for $i<n$. It suffices to show that
 $\ms{H}^i(\ms{M}\widehat{\otimes}^{\mb{L}}
 _{\BcompQ{m}{\ms{X}}}\ms{N})=0$ for $i<n-1$. Since this is local, and
 we may assume that $\ms{X}$ is affine. We will assume $\ms{X}$ to be
 affine in the following.

 For a positive integer $r$, we say that a finitely generated
 $\Dtild{m}{\ms{X}}$-module $\ms{P}$ is $r$-nearly projective if there
 exists a $\Dtild{m}{\ms{X}}$-module $\ms{Q}$, an integer $b$, and a
 short exact sequence
 $0\rightarrow\ms{P}\oplus\ms{Q}\rightarrow(\Dtild{m}{\ms{X}})^{\oplus
 b}\rightarrow\ms{R}\rightarrow0$ such that $\pi^r\ms{R}=0$. For any
 finitely generated projective $\DtildQ{m}{\ms{X}}$-module $\ms{P}'$,
 there exists an integer $r$ and an $r$-nearly projective
 $\Dtild{m}{\ms{X}}$-module $\ms{P}$ such that
 $\ms{P}\otimes\mb{Q}\cong\ms{P}'$. This shows that there exists a
 complex $\ms{P}_\bullet$ of $r$-nearly projective
 $\Dtild{m}{\ms{X}}$-modules concentrated in $[0,a]$ such that
 $\ms{P}_\bullet\otimes\mb{Q}$ is
 quasi-isomorphic to $\ms{M}$. Thus, it suffices to show that
 for any $r$-nearly projective $\Dtild{m}{\ms{X}}$-module $\ms{P}$ and
 $\ms{N}\in D^b_{\mr{qc}}(\Dtild{m}{\ms{X}})$ such that
 $\ms{H}^i(\ms{N})=0$ for $i<n$, we get
 \begin{equation*}
  \varpi^r\ms{H}^i(\ms{P}\widehat{\otimes}^{\mb{L}}_{\Bcomp{m}{\ms{X}}}
   \ms{N})=0
 \end{equation*}
 for $i<n-1$. Since $\ms{P}$ is $r$-nearly projective,
 $\varpi^r\ms{H}^i(\ms{P}\otimes_{\Bcomp{m}{\ms{X}}}^{\mb{L}}\ms{N})=0$
 for any $i<n$. This shows that
 \begin{equation*}
  \varpi^r\ms{H}^i(\ms{P}_j\otimes_{\Bmod{m}{X_j}}^{\mb{L}}
   \ms{N})=0
 \end{equation*}
 for any $j$ and $i<n-1$, where $\ms{P}_j:=\ms{P}\otimes_RR_j$. Now,
 it remains to take $R\underleftarrow{l}_{X*}$, but
 since this functor
 is a right derived functor, we get the claim.
\end{proof}

\subsection{}
Now, let us state a key proposition in this section.
Let $\ms{X}$ be a smooth formal scheme, and $\ms{X}'$ be a smooth
lifting of $X_0^{(s)}$.

\begin{prop*}
 \label{keypropcal}
 Let $\ms{M}$ be a complex in
 $D^b_{\mr{perf}}(\DtildQ{m}{\ms{X}})^{\mr{g}}$, and $\ms{N}$ be a
 complex in $\DbqcQ(\Dtild{m}{\ms{X}})^{\mr{g}}$. We denote
 $\mb{D}_{\ms{X},Z}$ by $\mb{D}$. Then there exists the following
 quasi-isomorphism in $D^+(\ms{X})$
 \begin{equation*}
  \Psi\colon
  R\shom_{\DtildQ{m}{\ms{X}}}(\BcompQ{m}{\ms{X}},\mb{D}(\ms{M})
   \widehat{\otimes}^{\mb{L}}_{\BcompQ{m}{\ms{X}}}\ms{N})
   \xrightarrow{\sim} R\shom_{\DtildQ{m}{\ms{X}}}(\ms{M},\ms{N}).
 \end{equation*}
 Let $\ms{F}$ be a complex in $D^b_{\mr{perf}}(\DtildQ{m}{\ms{X}'})
 ^{\mr{g}}$, and $\ms{G}$ be a complex in
 $\DbqcQ(\Dtild{m}{\ms{X}'})^{\mr{g}}$. Consider the following diagram:
 \begin{equation*}
  \xymatrix{
   R\shom_{\DtildQ{m}{\ms{X}'}}(\BcompQ{m}{\ms{X}'},\mb{D}(\ms{F})
   \widehat{\otimes}^\mb{L}_{\BcompQ{m}{\ms{X}'}}\ms{G})
   \ar[r]^<>(.5){\Psi}\ar[d]_{\sim}&
   R\shom_{\DtildQ{m}{\ms{X}'}}(\ms{F},\ms{G})\ar[d]^{\sim}\\
  R\shom_{\DtildQ{m+s}{\ms{X}}}(\BcompQ{m+s}{\ms{X}},\mb{D}
   (F^*\ms{F})\widehat{\otimes}^{\mb{L}}_{\BcompQ{m+s}{\ms{X}}}
   F^*\ms{G})\ar[r]_<>(.5){\Psi}&R\shom_{\DtildQ{m+s}{\ms{X}}}
   (F^*\ms{F},F^*\ms{G})
   }
 \end{equation*}
 where the vertical homomorphisms are canonical isomorphisms of
 complexes which are defined by the theorem of Frobenius descent
 {\normalfont \cite[4.1.3]{Ber2}}. This diagram is {\em commutative up
 to multiplication by $q^d$} where $d$ denotes the dimension of
 $\ms{X}$.
\end{prop*}

\begin{rem*}
 We note that the complex
 $R\shom_{\DtildQ{m}{\ms{X}}}(\BcompQ{m}{\ms{X}},\mb{D}(\ms{M})
 \widehat{\otimes}^{\mb{L}}_{\BcompQ{m}{\ms{X}}}\ms{N})$ makes sense
 thanks to Lemma \ref{tensboundedlem}.
\end{rem*}

The proof of the proposition will be given in paragraph
\ref{proofpropsec}, and we
will start preparations of the proof from the next paragraph. Unless
otherwise stated, $\ms{M}$, $\ms{N}$, $\ms{F}$, $\ms{G}$ are {\it not}
the sheaves in the proposition.

\subsection{}
\label{dbqcqnorm}
First, we will prove the following small lemma.

\begin{lem*}
 Let $\ms{M}$ be a complex in $D^b_{\mb{Q},\mr{coh}}(\Dtild{m}{\ms{X}})$
 and $\ms{N}$ be one in $\DbqcQ(\Dtild{m}{\ms{X}})$. Then, there is a
 canonical isomorphism
 \begin{equation*}
 \mr{Hom}_{\DbqcQ(\Dtild{m}{\ms{X}})}(\ms{M},\ms{N})\xrightarrow{\sim}
  \mr{Hom}_{\DtildQ{m}{\ms{X}}}(\ms{M},\ms{N}).
 \end{equation*}
\end{lem*}
\begin{proof}
 For $\ms{F}$ and $\ms{G}$ in $\DbqcQ(\Dtild{m}{\ms{X}})$, we define
 \begin{equation*}
  R\shom_{\DbqcQ(\Dtild{m}{\ms{X}})}(\ms{F},\ms{G}):=
   R\shom_{\Dtild{m}{\ms{X}}}(\ms{F},\ms{G})\otimes\mb{Q}.
 \end{equation*}
 There exists the canonical homomorphism
 \begin{equation*}
  R\shom_{\DbqcQ(\Dtild{m}{\ms{X}})}(\ms{M},\ms{N})\rightarrow
   R\shom_{\DtildQ{m}{\ms{X}}}(\ms{M}\otimes\mb{Q},\ms{N}
   \otimes\mb{Q}),
 \end{equation*}
 and it suffices to show that this is an isomorphism.
 Since the problem
 is local, we may assume that $\ms{X}$ is affine. It suffices to show
 the claim in the case where $\ms{M}$ is projective. Since $\ms{M}$ is a
 direct factor of $(\Dtild{m}{\ms{X}})^{\oplus n}$ for some $n$, we are
 reduced to showing the case $\ms{M}=\Dtild{m}{\ms{X}}$.
 In this case, the lemma is straightforward.
\end{proof}

\subsection{}
To compare Frobenius pull-backs, we need to construct a certain
isomorphism in $D^-(\ms{X})$. Let $\ms{N}$ be a complex in
$D^b_{\mr{perf}}(\DtildQ{m}{\ms{X}'})^{\mr{d}}$ ({\it e.g.\
}$\widetilde{\omega}_{\ms{X}',\mb{Q}}$ by
(\ref{spencerresbmod})), and
$\ms{M}$ be a complex in  $D^-(\DtildQ{m}{\ms{X}})^{\mr{g}}$. The
isomorphism we will construct is the following:
\begin{equation}
 \label{canhomflsh}
 \ms{N}\otimes^{\mb{L}}_{\DtildQ{m}{\ms{X}'}}\ms{M}\xrightarrow
 {\sim}F^\flat\ms{N}\otimes^{\mb{L}}_{\DtildQ{m+s}{\ms{X}}}
 F^*\ms{M}.
\end{equation}
Let $\ms{Y}$ be a smooth formal scheme, and let $\ms{F}$ be a complex in
$D^b_{\mr{perf}}(\DtildQ{m}{\ms{Y}})^{\mr{d}}$, and $\ms{E}$ be a complex
in $D^-(\DtildQ{m}{\ms{Y}})^{\mr{g}}$.
First, there exists an isomorphism
$\mb{D}^{\mr{d}}\circ\mb{D}^{\mr{d}}(\ms{F})\cong\ms{F}$ where
\begin{equation*}
 \mb{D}^{\mr{d}}(\ms{F}):=R\shom_{\DtildQ{m}{\ms{Y}}}
  (\ms{F},\DtildQ{m}{\ms{Y}})\otimes\omega_{\ms{Y}}
\end{equation*}
by \cite[II, 3.6]{Vir}. Using this, we get
\begin{align*}
 \ms{F}\otimes^\mb{L}_{\DtildQ{m}{\ms{Y}}}\ms{E}&\cong
 R\shom_{\DtildQ{m}{\ms{Y}}}
 (\mb{D}^{\mr{d}}(\ms{F}),\DtildQ{m}{\ms{Y}}\otimes\omega_\ms{Y})
 \otimes^\mb{L}_{\DtildQ{m}{\ms{Y}}}\ms{E}\\
 &\cong R\shom_{\DtildQ{m}{\ms{Y}}}(\mb{D}^{\mr{d}}(\ms{F}),
 \omega_{\ms{Y}}\otimes\ms{E})
\end{align*}
where the second isomorphism is by \cite[2.1.17 (i)]{CarD}.
Now, we define (\ref{canhomflsh}) in the following way:
\begin{align*}
 \ms{N}\otimes^{\mb{L}}_{\DtildQ{m}{\ms{X}'}}\ms{M}&\cong
 R\shom_{\DtildQ{m}{\ms{X}'}}(\mb{D}^{\mr{d}}(\ms{N}),
 \omega_{\ms{X}'}\otimes\ms{M})\xrightarrow{\sim}
 R\shom_{\DtildQ{m+s}{\ms{X}}}(F^\flat\mb{D}^{\mr{d}}
 (\ms{N}),F^\flat(\omega_{\ms{X}'}\otimes\ms{M}))\\
 &\cong R\shom_{\DtildQ{m+s}{\ms{X}}}(\mb{D}^{\mr{d}}
 (F^{\flat}\ms{N}),\omega_{\ms{X}'}\otimes F^*\ms{M})\cong
 F^\flat\ms{N}\otimes^{\mb{L}}_{\DtildQ{m+s}{\ms{X}}}F^*\ms{M}.
\end{align*}
Here the second isomorphism follows by the theorem of Frobenius
descent.

\subsection{}
\label{lemcrucons}
We will construct the homomorphism $\Psi$ in the proposition step by
step. Let $\Theta_{\ms{X}}$ be the tangent sheaf on $\ms{X}$, and we put
$\Theta_{\ms{X}}^i:=\wedgeb^i\Theta_{\ms{X}}$. 
First, note that there exists the Spencer resolution
\begin{equation}
 \label{spencerresbmod}
 \DtildQ{m}{\ms{X}}\otimes\Theta^\bullet_{\ms{X}}\rightarrow
  \BcompQ{m}{\ms{X}}.
\end{equation}
This can be seen in exactly the same way as the proof of
\cite[4.3.3]{Ber2}. Indeed, $\Bmod{m}{\ms{X},\mb{Q}}\otimes\Dmod{m}
{\ms{X},\mb{Q}}$ is flat over $\Dmod{m}{\ms{X},\mb{Q}}$. Since
$\Bmod{m}{\ms{X}}\otimes\Dmod{m}{\ms{X}}$ is noetherian, the $p$-adic
completion $\Dtild{m}{\ms{X}}$ is flat over
$\Bmod{m}{\ms{X}}\otimes\Dmod{m}{\ms{X}}$. Thus,
$\DtildQ{m}{\ms{X}}$ is flat over $\Dmod{0}{\ms{X},\mb{Q}}$. It remains
to show that $\DtildQ{m}{\ms{X}}\otimes_{\Dmod{0}{\ms{X},
\mb{Q}}}\mc{O}_{\ms{X},\mb{Q}}\cong\BcompQ{m}{\ms{X}}$, whose proof is
straightforward. This shows that $\BcompQ{m}{\ms{X}}$ is perfect as a
complex.

Let $\ms{M}$ be a bounded $\DtildQ{m}{\ms{X}}$-complex. We
have the following quasi-isomorphisms
\begin{align*}
 R\shom_{\DtildQ{m}{\ms{X}}}(\BcompQ{m}{\ms{X}},\ms{M})&\cong 
 R\shom_{\DtildQ{m}{\ms{X}}}(\BcompQ{m}{\ms{X}},\DtildQ{m}{\ms{X}})
 \otimes^{\mb{L}}_{\DtildQ{m}{\ms{X}}}\ms{M}\cong\widetilde{\omega}
 _{\ms{X},\mb{Q}}\otimes^{\mb{L}}_{\DtildQ{m}{\ms{X}}}\ms{M}\,[-d]
\end{align*}
where the first isomorphism holds by \cite[2.1.17]{CarD}, and the second
by (\ref{spencerresbmod}) and \cite[(3.5.5.1)]{BerInt}. Now the
point where we need to use the explicit computation of Frobenius
isomorphism of \S\ref{sect1} is in the proof of the following lemma.

\begin{lem*}
 Let $\ms{M}$ be a complex in $D^b(\DtildQ{m}{\ms{X}'})$. Consider the
 following diagram in
 $D^b(\ms{X})$.
 \begin{equation}
  \label{comwantdercom}
   \xymatrix@C=50pt{
   R\shom_{\DtildQ{m}{\ms{X}'}}(\BcompQ{m}{\ms{X}'},
   \ms{M})\ar[r]^<>(.5){\sim}\ar[d]_{\sim}&
   \widetilde{\omega}_{\ms{X}',\mb{Q}}
   \otimes^\mb{L}_{\DtildQ{m}{\ms{X}'}}\ms{M}[-d]\ar[d]^{\sim}\\
  R\shom_{\DtildQ{m+s}{\ms{X}}}(\BcompQ{m+s}{\ms{X}},
   F^*\ms{M})\ar[r]_<>(.5){\sim}&\widetilde{\omega}_{\ms{X},\mb{Q}}\otimes
   _{\DtildQ{m+s}{\ms{X}}}^{\mb{L}}F^*\ms{M}[-d]
   }
 \end{equation}
 Here the right vertical homomorphism is {\normalfont
 (\ref{canhomflsh})} composed with the canonical isomorphism
 $F^\flat\omega_{\ms{X}'}\cong\omega_{\ms{X}}$, and the left vertical
 homomorphism is induced by the functor $F^*$. Then this diagram is
 commutative up to multiplication by $q^d$.
\end{lem*}
\begin{proof}
 Let $G$ be the functor $\shom_{\DtildQ{m+s}{\ms{X}}}(F^*
 \BcompQ{m}{\ms{X}'},\bullet)$, and $G'$ be the functor
 $\shom_{\DtildQ{m+s}{\ms{X}}}(\BcompQ{m+s}{\ms{X}},\bullet)$.
 By (\ref{spencerresbmod}), $R^iG$ and $R^iG'$ vanishes for $i>d$.
 We define a functor $H$ to be $R^dG$, and $H'$ to be $R^dG'$.
 The canonical isomorphism $F^*\BcompQ{m}{\ms{X}'}
 \xrightarrow{\sim}\BcompQ{m+s}{\ms{X}}$ induces an isomorphism of
 functors $H\xrightarrow{\sim}H'$. By \cite[I, 7.4]{Har}, we get the
 following commutative diagram of functors.
 \begin{equation*}
  \xymatrix{
   RG\ar[r]^<>(.5){\sim}\ar[d]_{\sim}&LH[-d]\ar[d]^{\sim}\\
  RG'\ar[r]_<>(.5){\sim}&LH'[-d]}
 \end{equation*}
 We note that flat $\DtildQ{m+s}{\ms{X}}$-modules belong to the set
 $P$ of {\it loc.\ cit.} For a flat $\DtildQ{m+s}{\ms{X}}$-module
 $\ms{P}$, we get a canonical isomorphism $H(\ms{P})\cong
 F^*\widetilde{\omega}_{\ms{X}'}\otimes_{\DtildQ{m+s}{\ms{X}}}\ms{P}$ and
 $H'(\ms{P})\cong\widetilde{\omega}_{\ms{X}}\otimes_{\DtildQ{m+s}{\ms{X}}}
 \ms{P}$, which induces canonical isomorphisms of functors
 $LG\cong F^\flat\widetilde{\omega}_{\ms{X}'}\otimes_{\DtildQ{m+s}{\ms{X}}}
 ^{\mb{L}}$, and $LG'\cong\widetilde{\omega}_{\ms{X}}\otimes
 _{\DtildQ{m+s}{\ms{X}}}^{\mb{L}}$. Consider the following diagram:
 \begin{equation*}
  \xymatrix{
   R\shom_{\DtildQ{m}{\ms{X}'}}(\BcompQ{m}{\ms{X}'},
   \ms{M})\ar[rr]\ar[d]\ar@{}[rrd]|{\circlearrowleft}&&
   \widetilde{\omega}_{\ms{X}'}\otimes^\mb{L}_{\Dtild{m}{\ms{X}'}}
   \ms{M}[-d]\ar[d]\\
  R\shom_{\DtildQ{m+s}{\ms{X}}}(F^*\BcompQ{m}{\ms{X}'},
   F^*\ms{M})\ar[r]\ar[d]\ar@{}[rd]|{\circlearrowleft}&
   LH(F^*\ms{M})[-d]\ar[r]\ar[d]\ar@{}[rd]|{\ccirc{a}}&
   F^\flat\widetilde{\omega}_{\ms{X}'}\otimes_{\Dtild{m+s}{\ms{X}}}^{\mb{L}}
   F^*\ms{M}[-d]\ar[d]\\
  R\shom_{\DtildQ{m+s}{\ms{X}}}(\BcompQ{m}{\ms{X}},F^*\ms{M})\ar[r]&
   LH'(F^*\ms{M})[-d]\ar[r]&\widetilde{\omega}_{\ms{X}}\otimes_{
   \Dtild{m+s}{\ms{X}}}^{\mb{L}}F^*\ms{M}[-d]
  }
 \end{equation*}
 where all the arrows are isomorphisms, and $\circlearrowleft$ denotes
 that the marked square is commutative. Thus, to show that the big
 diagram is commutative up to some multiplication is equivalent to
 showing that the diagram $\ccirc{a}$ is commutative up to the same
 multiplication. Since the homomorphisms of the diagram $\ccirc{a}$ are
 induced by a diagram of functors between $H$, $H'$,
 $F^*\widetilde{\omega}_{\ms{X}'}\otimes$,
 $\widetilde{\omega}_{\ms{X}}\otimes$, it suffices to show the
 commutativity up to the same multiplication for this diagram of
 functors. Thus the problem is local.

 We may assume that $\ms{X}'$ is affine, and possesses a
 system of local coordinates $\{y_1,\dots,y_d\}$. Moreover, we can take
 a system of local coordinates $\{x_1,\dots,x_d\}$ of $\ms{X}$ and a
 lifting $F\colon\ms{X}\rightarrow\ms{X}'$ of relative Frobenius
 morphism such that $F^*(y_i)=x^q_i$. Under this situation, let us show
 that the diagram is commutative up to multiplication by $q^d$. From now
 on, we do not make any difference between quasi-coherent modules and its
 global sections. It suffices to show the commutativity in the case
 where $\ms{M}$ is flat over $\DcompQ{m}{\ms{X}'}$.

 Let $F_*\colon\Theta_\ms{X}\rightarrow F^*\Theta_{\ms{X}'}$ be the
 canonical homomorphism. We have the following homomorphism
 $\DtildQ{m+s}{\ms{X}}\otimes\Theta^k_{\ms{X}}\rightarrow
 F^*(\DtildQ{m}{\ms{X}'}\otimes\Theta^k_{\ms{X}'})$ sending
 $P\otimes(dy_{i_1}\wedge\dots\wedge
 dy_{i_k})$ to $P\cdot(1\otimes qx^{q-1}_{i_1}\dots qx^{q-1}_{i_k}
 (dx_{i_1}\wedge\dots\wedge dx_{i_k}))$. This defines, in fact, a
 homomorphism of complexes
 \begin{equation}
  \label{spencerhom}
  \DtildQ{m+s}{\ms{X}}\otimes\Theta^\bullet_{\ms{X}}\rightarrow
   F^*(\DtildQ{m}{\ms{X}'}\otimes\Theta^\bullet_{\ms{X}'})
 \end{equation}
 by the proof of \cite[4.3.5]{Ber2}. It suffices to show that the
 diagram of modules
 \begin{equation*}
  \xymatrix{
   \Omega^d_{\ms{X}'}\otimes_{\mc{O}_{\ms{X}'}}\ms{M}
   \ar[r]\ar[d]&
   \omega_{\ms{X}'}\otimes_{\DcompQ{m}{\ms{X}'}}\ms{M}
   \ar[d]\\
  \Omega^d_{\ms{X}}\otimes_{\mc{O}_{\ms{X}}}F^*\ms{M}
   \ar[r]&
   \omega_{\ms{X}}\otimes_{\DcompQ{m}{\ms{X}}}F^*\ms{M}
   }
 \end{equation*}
 is commutative up to multiplication by $q^d$, where the left
 vertical homomorphism is induced by (\ref{spencerhom}). Since the right
 vertical homomorphism sends $(dy_1\wedge\dots\wedge dy_d)\otimes m$ to
 $x_1^{q-1}\dots x_d^{q-1}\cdot(dx_1\wedge\dots\wedge dx_d)
 \otimes(1\otimes m)$ by using Proposition \ref{mainpropbalhar} and
 Theorem \ref{maincalc}, we get the claim, and conclude the proof of the
 lemma.
\end{proof}

\subsection{}
\label{LemBjoCar}
We have the following lemma whose proof is similar to that of
\cite[2.1.27]{CarD}, and we leave it to the reader.
\begin{lem*}
 Let $\ms{N}$ be a complex in $\Dmqc(\Dtild{m}{\ms{X}})^{\mr{d}}$,
 and $\ms{M}$ and $\ms{M}'$ be two complexes in
 $\Dmqc(\Dtild{m}{\ms{X}})^{\mr{g}}$. Then there is a canonical
 isomorphism
\begin{equation*}
 \ms{N}\,\widehat{\otimes}^{\mb{L}}_{\Dtild{m}{\ms{X}}}\,
  (\ms{M}\,\widehat{\otimes}^{\mb{L}}_{\Bcomp{m}{\ms{X}}}\,\ms{M}')
  \cong(\ms{N}\,\widehat{\otimes}_{\Bcomp{m}{\ms{X}}}
  ^{\mb{L}}\,\ms{M})\,\widehat{\otimes}^{\mb{L}}_{\Dtild{m}{\ms{X}}}\,
  \ms{M}'.
\end{equation*}
\end{lem*}

\usubsection{Proof of Proposition \ref{keypropcal}}
\label{proofpropsec}
Here, we use the notation in the proposition. We apply Lemma
\ref{lemcrucons} to
$\mb{D}(\ms{M})\widehat{\otimes}^\mb{L}_{\BcompQ{m}{\ms{X}}}\ms{N}$.
When we omit bases of tensor products, they are taken over
$\BcompQ{m}{\ms{X}}$. Then we get isomorphisms
\begin{align*}
 &R\shom_{\DtildQ{m}{\ms{X}}}(\BcompQ{m}{\ms{X}},\mb{D}(\ms{M})
 \widehat{\otimes}^\mb{L}\ms{N})\cong\widetilde{\omega}_{\ms{X},\mb{Q}}
 \otimes^{\mb{L}}_{\DtildQ{m}{\ms{X}}}(\mb{D}(\ms{M})
 \widehat{\otimes}^{\mb{L}}\ms{N})[-d]\\&\qquad
 \cong(\widetilde{\omega}_{\ms{X},\mb{Q}}\otimes^{\mb{L}}
 \mb{D}(\ms{M}))\otimes^{\mb{L}}_{\DtildQ{m}{\ms{X}}}\ms{N}[-d]
 \cong R\shom_{\DtildQ{m}{\ms{X}}}(\ms{M},\DtildQ{m}{\ms{X}})
 \otimes^{\mb{L}}_{\DtildQ{m}{\ms{X}}}\ms{N}\\&\qquad
 \cong R\shom_{\DtildQ{m}{\ms{X}}}(\ms{M},\ms{N})
\end{align*}
where the first isomorphism is the one in paragraph \ref{lemcrucons},
the second isomorphism is defined by Lemma \ref{LemBjoCar}, and we
used the fact that $\ms{M}$ is a perfect complex in the last isomorphism
(cf.\ \cite[2.1.12, or 2.1.17]{CarD}). This is nothing but $\Psi$ that
we are looking for. The second, third, and the last isomorphisms are
compatible with Frobenius. Thus the statement of Frobenius follows by
Lemma \ref{lemcrucons}.
\hspace{\fill}$\blacksquare$

\subsection{}
Let $\ms{X}$ be a smooth formal scheme. We denote by
$\tdp(\DtdagQ{\ms{X}})$ the full subcategory of
$D^b_{\mr{perf}}(\DtdagQ{\ms{X}})$ consisting of a complex $\ms{M}$ such
that there exists a complex $\ms{M}'$ in
$D^b_{\mr{perf}}(\DtildQ{m}{\ms{X}})$ for some $m$ and an isomorphism
$\DtdagQ{\ms{X}}\otimes_{\DtildQ{m}{\ms{X}}}\ms{M}'\cong\ms{M}$.

\begin{lem*}
 \label{perflift}
 Assume that $\ms{X}$ is quasi-compact.
 For any complex $\ms{M}$ in $D^b_{\mr{coh}}(\DtildQ{m}{\ms{X}})$, there
 exists an integer $m'\geq m$ such that
 $\DtildQ{m'}{\ms{X}}\otimes_{\DtildQ{m}{\ms{X}}}\ms{M}$ is in
 $D^b_{\mr{perf}}(\DtildQ{m'}{\ms{X}})$. In particular, we have the
 canonical equivalence of categories
 \begin{equation*}
  \tdp(\DtdagQ{\ms{X}})\xrightarrow{\sim}D^b_{\mr{perf}}(\DtdagQ{\ms{X}})
   \xrightarrow{\sim}D^b_{\mr{coh}}(\DtdagQ{\ms{X}}).
 \end{equation*}
\end{lem*}
\begin{proof}
 Let us see the first claim. Since $\ms{X}$ is quasi-compact, the
 problem is local, and we may assume that $\ms{X}$ is affine.
 Since $\ms{X}$ is affine, we can take $\ms{M}$ to be a bounded complex
 such that each term is a coherent $\DtildQ{m}{\ms{X}}$-module. For
 $m'\geq m$, we
 denote $\ms{M}^{(m')}:=\DtildQ{m'}{\ms{X}}\otimes\ms{M}$ and
 $\ms{M}^\dag:=\DtdagQ{\ms{X}}\otimes\ms{M}$.
 Now, there exists a bounded finite locally projective
 $\DtdagQ{\ms{X}}$-complex $\ms{L}$
 and a quasi-isomorphism of complexes
 $\varphi\colon\ms{L}\rightarrow\ms{M}^\dag$ since $\DtdagQ{\ms{X}}$ has
 finite Tor-dimension by the result of Noot-Huyghe in \cite{NH2}. For a
 sufficiently large $m'$, this complex can be descended to level
 $m'$. Namely there exists a bounded locally
 finite projective $\DtildQ{m}{\ms{X}}$-complex $\ms{L}^{(m')}$
 and a homomorphism of complexes
 $\varphi'\colon\ms{L}^{(m')}\rightarrow\ms{M}^{(m')}$ such that
 $\DtdagQ{\ms{X}}\otimes\varphi'\cong\varphi$.
 The homomorphism $\varphi'$ may not be a quasi-isomorphism, but since
 the complexes are bounded and each term is coherent, there exists
 $m''\geq m'$ such that the homomorphism
 $\DtildQ{m''}{\ms{X}}\otimes\varphi'$ becomes a quasi-isomorphism,
 which concludes the proof. The latter statement follows from
 \cite[4.2.4]{BerInt}.
\end{proof}

\begin{rem*}
 We do not know if $\tdp(\DtdagQ{\ms{X}})$ and
 $D^b_{\mr{perf}}(\DtdagQ{\ms{X}})$ coincide or not in general.
\end{rem*}

\begin{thm}
 \label{smoothpoincare}
 Let $f\colon(\ms{X},Z)\rightarrow(\ms{Y},W)$ be a morphism of d-couples
 whose realization is smooth. We assume that $X_0^{(s)}$ and $Y_0^{(s)}$
 can be lifted to smooth formal schemes $\ms{X}'$ and $\ms{Y}'$. Let us
 denote by $d\geq 0$ the relative dimension of $X_0$ over $Y_0$. Then
 there is a canonical equivalence of cohomological functors from
 $\tdp(\DdagQ{\ms{Y}}(^\dag W))$ to
 $\tdp(\DdagQ{\ms{X}}(^\dag Z))$ with Frobenius isomorphisms
 \begin{equation}
  \label{commduexpu}
  (\mb{D}_{\ms{X},Z}\circ f^!)(d)[2d]\xrightarrow{\sim}
  f^!\circ\mb{D}_{\ms{Y},W}.
 \end{equation}
\end{thm}

\begin{proof}
 Let $\ms{M}^\dag$ be a complex in
 $\tdp(\DtdagQ{\ms{Y}})$. Then by definition, there
 exists $\ms{M}$ in $D^b_{\mr{perf}}(\DtildQ{m}{\ms{Y}})$ for some $m$
 and
 $\DtdagQ{\ms{Y}}\otimes_{\DtildQ{m}{\ms{Y}}}\ms{M}\cong\ms{M}^\dag$.
 First, let us define the homomorphism. By Proposition \ref{keypropcal},
 we get a homomorphism $\mr{Hom}_{\DtildQ{m}{\ms{Y}}}(\ms{M},\ms{M})
 \xleftarrow{\sim}\mr{Hom}_{\DtildQ{m}{\ms{Y}}}(\BcompQ{m}{\ms{Y}},
 \mb{D}(\ms{M})\widehat{\otimes}\ms{M})$. By Lemma \ref{dbqcqnorm} and
 the functoriality of the extraordinary pull-back functor $f^!$, we
 get
 \begin{equation*}
  \mr{Hom}_{\DtildQ{m}{\ms{Y}}}(\BcompQ{m}{\ms{Y}},\mb{D}
 (\ms{M})\widehat{\otimes}\ms{M})\rightarrow\mr{Hom}_{\DtildQ{m}
 {\ms{X}}}(f^!\BcompQ{m}{\ms{Y}},f^!(\mb{D}(\ms{M})
 \widehat{\otimes}\ms{M})).
 \end{equation*}
 This homomorphism is compatible with Frobenius pull-backs by the
 functoriality of the isomorphism \cite[3.2.4]{Ber2}.
 We get $f^!\BcompQ{m}{\ms{Y}}\cong\BcompQ{m}{\ms{X}}[d_f]$. Moreover,
 $f^!(\mb{D}(\ms{M})\widehat{\otimes}\ms{M})\cong
 f^!\circ\mb{D}(\ms{M})~\widehat{\otimes}~f^!\ms{M}[-d_f]$. This
 isomorphism is also compatible with Frobenius pull-backs. Thus we get
 \begin{equation*}
  \mr{Hom}_{\DtildQ{m}{\ms{X}}}(f^!\BcompQ{m}{\ms{Y}},
   f^!(\mb{D}(\ms{M})\widehat{\otimes}\ms{M}))\cong
   \mr{Hom}_{\DtildQ{m}{\ms{X}}}(\BcompQ{m}{\ms{X}}[d_f],
   f^!\mb{D}(\ms{M})~\widehat{\otimes}~f^!\ms{M}[-d_f]).
 \end{equation*}
 Now, using the proposition once again, we get an isomorphism
 \begin{equation*}
  \mr{Hom}_{\DtildQ{m}{\ms{X}}}(\BcompQ{m}{\ms{X}}[d_f],
   f^!\mb{D}(\ms{M})~\widehat{\otimes}~f^!\ms{M}[-d_f])
   \xrightarrow{\sim}\mr{Hom}_{\DtildQ{m}{\ms{X}}}(\mb{D}
   f^!\ms{M},f^!\mb{D}(\ms{M})[-2d_f]).
 \end{equation*}
 Composing all of them, we obtain a homomorphism
 \begin{equation*}
  \mr{Hom}_{\DtildQ{m}{\ms{Y}}}(\ms{M},\ms{M})\rightarrow
   \mr{Hom}_{\DtildQ{m}{\ms{X}}}(\mb{D}f^!\ms{M},f^!\mb{D}
   (\ms{M})[-2d_f]).
 \end{equation*}
 The image of the identity is the homomorphism we wanted.  By using
 \cite[4.3.3, 4.3.11]{BerInt}, we get the following diagram.
 \begin{equation*}
  \xymatrix{
   \mr{Hom}_{\DtildQ{m}{\ms{Y}'}}(\ms{M},\ms{M})\ar[r]\ar[d]_{F^*_Y}&
   \mr{Hom}_{\DtildQ{m}{\ms{X}'}}(\mb{D} f'^!(\ms{M}),f'^!
   \mb{D}(\ms{M})[-2d_f])\ar[d]^{F^*_X}\\
  \mr{Hom}_{\DtildQ{m}{\ms{Y}}}(F_Y^*\ms{M},F_Y^*\ms{M})\ar[r]&
   \mr{Hom}_{\DtildQ{m}{\ms{X}}}(\mb{D} f^!(F_Y^*\ms{M}),
   f^!\mb{D}(F_Y^*\ms{M})[-2d_f])
   }
 \end{equation*}
 This diagram is commutative up to multiplication by $q^{-d_Y}
 \cdot q^{d_X}=q^{d_f}$ by the commutativity of Proposition
 \ref{keypropcal}. Thus we obtain the homomorphism $\mb{D}\circ
 f^!(\ms{M})(d_f)[2d_f]\rightarrow f^!\circ\mb{D}(\ms{M})$. By tensoring
 with $\DtdagQ{\ms{X}}$, we get
 \begin{equation*}
  \mb{D}\circ f^!(\ms{M}^\dag)(d_f)[2d_f]\rightarrow f^!\circ
   \mb{D}(\ms{M}^\dag)
 \end{equation*}
 by using \cite[3.4.6 (iii)]{BerInt} and \cite[I.5.4]{Vir}.
 By construction, this does not depend on the choice of $\ms{M}$.
 It remains to show that this homomorphism is an isomorphism when $f$ is
 smooth.

 It suffices to show the equality for $\ms{M}=\DtdagQ{\ms{Y}}$, and we
 can forget about Frobenius pull-backs. We get
 $R\shom_{\DtdagQ{\ms{X}}}(\DtdagQ{\ms{X}\rightarrow\ms{Y}},
 \DtdagQ{\ms{X}})[d_f]\cong\DtdagQ{\ms{Y}
 \leftarrow\ms{X}}$. Indeed,
 \begin{align*}
  R\shom_{\DtdagQ{\ms{X}}}(\DtdagQ{\ms{X}\rightarrow\ms{Y}},
  \DtdagQ{\ms{X}})[d_f]&\cong
  R\shom_{\DdagQ{\ms{X}}}(\DdagQ{\ms{X}}\otimes\Theta_{\ms{X}/\ms{Y}}
  ^{\bullet},\DdagQ{\ms{X}})~[d_f]\\&\cong
  \Omega_{\ms{X}/\ms{Y}}^{\bullet}\otimes\DtdagQ{\ms{X}}\cong
  \DtdagQ{\ms{Y}\leftarrow\ms{X}}.
 \end{align*}
 Thus, we get
 \begin{align*}
  &\mb{D}\circ f^!(\DtdagQ{\ms{Y}})[2d_f]\cong\mb{D}(\DtdagQ{\ms{X}
  \rightarrow\ms{Y}})[d_f]\\&\qquad\cong\DtdagQ{\ms{Y}\leftarrow\ms{X}}
  \otimes\omega_{\ms{X}}^{-1}[d_f]\cong f^!\DtdagQ{\ms{Y}}\otimes
  \omega_{\ms{Y}}^{-1}\cong f^!\circ\mb{D}(\DtdagQ{\ms{Y}}).
 \end{align*}
 We can see that this isomorphism coincides with the homomorphism we
 have constructed, and the theorem follows.
\end{proof}

\begin{rem*}
 We may be able to weaken the assumption of the theorem. The theorem
 should hold only by assuming that
 $X_0\setminus Z\rightarrow Y_0\setminus W$ is smooth. Moreover, we
 may be able to see the theorem as a solution of a part of
 ``Cauchy-Kovalevskaya type problem''. Classically, this observation
 first appeared in Kashiwara's thesis \cite{Ka}, and the problem was
 interpreted in terms of the language of $\ms{D}$-modules. We expect
 that the analogous theorem also holds in our setting: if the morphism
 $f$ is ``non-characteristic'' to a coherent
 $F$-$\DdagQ{\ms{X}}$-module, then we get the isomorphism
 (\ref{commduexpu}).
\end{rem*}

\begin{cor}
 \label{constantcase}
 Let $\ms{X}$ be a smooth formal scheme. Let $d$ be the
 dimension of $\ms{X}$, and $Z$ be a divisor of the special fiber of
 $\ms{X}$. Then we get a canonical isomorphism
 \begin{equation*}
  \mb{D}_{\ms{X},Z}(\mc{O}_{\ms{X},\mb{Q}}(^\dag Z))
   \xrightarrow{\sim}\mc{O}_{\ms{X},\mb{Q}}(^\dag Z)(-d)
 \end{equation*}
 where $d$ denotes the dimension of $\ms{X}$.
\end{cor}
\begin{proof}
 Apply Theorem \ref{smoothpoincare} in the case where
 $\ms{Y}=\mr{Spf}(R)$.
\end{proof}

\begin{rem*}
 The question to calculate $\mb{D}(\mc{O}_{\ms{X},\mb{Q}})$ was posed by
 Caro in \cite[4.3.3]{Caro}, saying that ``{\em En effet, lorsque $X=P$,
 on retrouve l'isomorphisme canonique:
 $\mb{D}_{\mc{P},T}(\mc{O}_{\mc{P}}(^\dag T)_{\mb{Q}})\rightarrow
 \mc{O}_{\mc{P}}(^\dag T)_{\mb{Q}}$. Je n'ai pas de
 contre-exemple mais la compatibilit\'{e} \`{a} Frobenius de ce dernier
 isomorphisme me para\^{i}t inexacte.}''
\end{rem*}

\subsection{}
Let $\ms{X}$ be a smooth formal scheme of dimension $d$, $Z$ be a
divisor of the special fiber $X_0$ of $\ms{X}$. For an overconvergent
$F$-isocrystal $M$ on $X_0$, we denote by $M^\vee$ the dual
overconvergent $F$-isocrystal of $M$.

\begin{cor*}
 \label{compduals}
 Let $\mr{sp}\colon\ms{X}_K\rightarrow\ms{X}$ be the specialization map,
 and let $M$ be an overconvergent $F$-isocrystal on $X_0\setminus
 Z$. Then,
 \begin{equation*}
  (\mb{D}_{\ms{X},Z}(\mr{sp}_*(M))\cong\mr{sp}_*(M^\vee)(-d).
 \end{equation*}
\end{cor*}
\begin{proof}
 Apply Corollary \ref{constantcase} to \cite[2.3.37]{CarD}.
\end{proof}

\begin{rem*}
 \label{compLcarLE}
 This corollary completes the comparison of $L$-functions of isocrystals
 and arithmetic $\ms{D}$-modules \cite[3.3.1]{CarL}. Namely, we get
 \begin{equation*}
  L(Y,E,t)=L(\ms{Y},\mr{sp}_*E,\,q^{d_X}\cdot t)
 \end{equation*}
 using the notation of {\it loc.\ cit}.
 However, in {\it loc.\ cit.\ }the definition of the
 Frobenius structure of the push-forward is modified in order to make
 the relative duality compatible with Frobenius (cf.\
 \cite[1.2.11]{CarL}), and the definition may not be the same as that of
 Berthelot. Still, we will show that this duality is compatible with
 Frobenius in the next section (cf.\ Corollary \ref{adointcomp}), and
 the Frobenius structure of the push-forward is in fact the same as that
 of Berthelot.
\end{rem*}

\subsection{}
Now, we will compare the rigid cohomologies and the push-forwards of
arithmetic $\ms{D}$-modules with Frobenius structure when varieties can
be lifted to smooth formal schemes. If we do not consider Frobenius
structure, they coincide up to shifts of degree, which is a result of
Berthelot (cf.\ \cite[4.3.6.3]{BerInt}). If we consider Frobenius
structure, we need a Tate twist. This twist naturally appears in the
philosophy of six functors (cf.\ paragraph \ref{comrigdpushconcl}).

Let us fix the notation. We consider the situation
\ref{Frobsituations}.\ref{Frobstr}. Let $\ms{X}$ be a smooth formal
scheme and let $p\colon\ms{X}\rightarrow\ms{S}:=\mr{Spf}(R)$
be the structural morphism. Let $X_0$ be the special fiber of $\ms{X}$
as usual, $\ms{X}'$ be a lifting of $X_0^{(s)}$, and
$\ms{X}_K$ be the Raynaud generic fiber. We denote by
$\mr{sp}\colon\ms{X}_K\rightarrow\ms{X}$ the specialization map of
topoi. Let $\ms{M}$ be an $F$-$\DdagQ{\ms{X}}$-module. We define the
``rigid cohomology'' of $\ms{M}$ in the following way. Let
$p_K\colon\ms{X}_K\rightarrow\mr{Spm}(K)$ be the structural morphism.
We define 
\begin{equation*}
 H^i_{\mr{rig}}(X_0,\ms{M}):=R^ip_{K*}(\Omega^{\bullet}_{\ms{X}_K}
  \otimes_{\mc{O}_{\ms{X}_K}}\mr{sp}^*\ms{M}).
\end{equation*}
We define the Frobenius structure in the following way. There exists an
isomorphism
\begin{equation*}
 \varphi\colon\Omega^\bullet_{\ms{X}'_K}\otimes_{\mc{O}_{\ms{X}'}}
  \ms{M}^\sigma\xrightarrow{\sim}\Omega^\bullet_{\ms{X}_K}
  \otimes_{\mc{O}_{\ms{X}}}F^*\ms{M}^{\sigma}
  \xrightarrow{\sim}\Omega^\bullet_{\ms{X}_K}\otimes
  _{\mc{O}_{\ms{X}}}\ms{M}
\end{equation*}
in $D^+(\ms{X})$ where the first isomorphism follows from
\cite[4.3.5]{Ber2} and we used the Frobenius structure of $\ms{M}$ in
the second isomorphism. Thus, we get an isomorphism
\begin{equation*}
 H^i_{\mr{rig}}(X_0,\ms{M})\xrightarrow{\sim}H^i_{\mr{rig}}(X^{(s)}_0,
  \ms{M}^\sigma)\xrightarrow[\varphi]{\sim}H^i_{\mr{rig}}(X_0,\ms{M}),
\end{equation*}
where the first isomorphism is the base change homomorphism.
This is the induced Frobenius structure on the cohomology.

Suppose moreover that $\ms{X}$ is {\em smooth proper} and let $Z$ be a
divisor of $X_0$. When $M$ is an overconvergent $F$-isocrystal on
$X_0\setminus Z$ along $Z$, the rigid cohomology of $\mr{sp}_*(M)$ is
isomorphic to the usual rigid cohomology of $M$.

\begin{thm}
 \label{comprigidaD}
 We preserve the notation. We suppose that $\ms{X}$ is purely of
 dimension $d$. Let $\ms{M}$ be a coherent
 $F$-$\DdagQ{\ms{X}}$-module. Then we get
 \begin{equation*}
  H^ip_+\ms{M}\cong H^{i+d}_{\mr{rig}}(X_0,\ms{M})(d).
 \end{equation*}
\end{thm}
\begin{proof}
 Let $\ms{M}'$ be a coherent $\DdagQ{\ms{X}'}$-module, and consider the
 following diagram of complexes in $D^b(\ms{X})$.
 \begin{equation*}
  \xymatrix{
   F^*(\Omega_{\ms{X}'}^{\bullet}\otimes_{\mc{O}_{\ms{X}'}}
   \ms{M}')[d]\ar[r]\ar[d]&F^*\DdagQ{\ms{S}'\leftarrow\ms{X}'}
   \otimes^{\mb{L}}_{\DdagQ{\ms{X}'}}\ms{M}'\cong F^*(\omega_{\ms{X}'}
   \otimes^{\mb{L}}_{\DdagQ{\ms{X}'}}\ms{M}')\ar[d]\\
  \Omega_{\ms{X}}^{\bullet}\otimes_{\mc{O}_{\ms{X}}}F^*\ms{M}'[d]
   \ar[r]&\DdagQ{\ms{S}\leftarrow\ms{X}}\otimes^{\mb{L}}
   _{\DdagQ{\ms{X}}}F^*\ms{M}'\cong\omega_{\ms{X}}\otimes^{\mb{L}}
   _{\DdagQ{\ms{X}}}F^*\ms{M}'
   }
 \end{equation*}
 where the horizontal arrows are induced by \cite[4.2.1.1]{BerInt}.
 We need to see that this diagram is commutative up to multiplication by
 $q^d$. Indeed, this diagram is nothing but Lemma \ref{lemcrucons} by
 taking into account the proof of \cite[4.3.5]{Ber2}.

 Now to know the Frobenius actions on the cohomologies, apply $\ms{M}'$
 to be $\ms{M}^{\sigma}$. We only need to take $Rp_*$ to the four
 sheaves in the diagram with this $\ms{M}'$, and we get the theorem.
\end{proof}

\begin{rem}
 \label{remvarcahb}
 (i) We can also compare in the relative situations. Namely, when we are
 given a smooth morphism of smooth formal schemes
 $\ms{X}\rightarrow\ms{Y}$, we are able to compare the relative rigid
 cohomology and the push-forward as arithmetic $\ms{D}$-module. Required
 methods are exactly the same, so we leave the precise
 formulation and calculation to the readers.

 (ii) In \cite[6.2]{NH}, Noot-Huyghe cited a calculation of
 Baldassarri-Berthelot \cite{BB}. However, the definition of Frobenius
 structures in \cite{NH} and \cite{BB} are not the same, and we need
 a Tate twist here. Precisely, Noot-Huyghe used
 cohomological functors of the arithmetic $\ms{D}$-module theory to
 define the Frobenius structure. On the other
 hand Baldassarri-Berthelot used the relative rigid cohomologies to
 define the Frobenius structure on the Fourier transform. Thus we need
 to add the Tate twist $(N)$ on the right hand side of the isomorphism
 in \cite[6.2]{NH}, namely
 \begin{equation*}
  \mc{F}_{\pi}(\mc{O}^\dag_{Y,\mathbf{Q}})[2-N]\cong\mc{H}^{\dag N}_{X_0}
   (\mc{O}^\dag_{Y^\vee,\mathbf{Q}})(N)
 \end{equation*}
 using the notation of {\it loc.\ cit}.

 (iii) D. Caro pointed out to the author in personal communications that
 we need a suitable Tate twist in \cite[Proposition
 2.3.12]{CarInv}. If we put this Tate twist, this proposition can be
 seen as a generalization of our theorem. He also pointed out that some
 modifications might be needed in {\it loc.\ cit.}\ Theorem 3.3.4, in
 whose proof he used the proposition.
\end{rem}

\section{Complementary results}
In this section, we will prove three complementary results; 1)
commutation of the dual functor and the tensor product, 2) the
K\"{u}nneth formula, and 3) the compatibility of the relative duality
isomorphism by Virrion with Frobenius. The first commutation result is
another application of Theorem \ref{comprigidaD}, and the
proofs of 2) and 3) are independent from the other part of this
paper. Although the K\"{u}nneth
formula for arithmetic $\ms{D}$-modules seems to be well-known to
experts, we could not find any  appropriate reference. We think that
this would be a good occasion to include the proof. The compatibility of
relative duality is needed to establish the Poincar\'{e} duality.

\nsubsection*{Commutation of the dual functor and tensor product}

\subsection{}
\label{isom1}
We consider the situation \ref{Frobsituations}.\ref{noFrobstr}.
Let $X$ be a smooth scheme over $\mr{Spec}(R_i)$ for some $i$, and let
$Z$ be a divisor. We put $\BcompQ{m}{\ms{X}}:=\BcompQ{m}{\ms{X}}(Z)$,
and $\DtildQ{m}{\ms{X}}:=\BcompQ{m}{\ms{X}}\widehat{\otimes}
_{\mc{O}_\ms{X}}\Dcomp{m}{\ms{X}}$. We denote the
dual functor with respect to $\DtildQ{m}{\ms{X}}$ by
$\mb{D}$. First, we get the following lemma.

\begin{lem*}
 Let $\ms{M}$ be a complex in $D^b_{\mr{perf}}(\DtildQ{m}{\ms{X}})$, and
 $\ms{N}$ be a coherent $\DtildQ{m}{\ms{X}}$-module which is also
 coherent as a $\BcompQ{m}{\ms{X}}$-module. We denote the dimension of
 $\ms{X}$ by $d$. Then, we have the following isomorphism
 \begin{equation*}
  \kappa\colon\mb{D}(\ms{M})\otimes_{\BcompQ{m}{\ms{X}}}
   ^{\mb{L}}\ms{N}\cong R\shom_{\DtildQ{m}{\ms{X}}}(\ms{M},
   \DtildQ{m}{\ms{X}}\otimes^{\mb{L}}_{\BcompQ{m}{\ms{X}}}\ms{N})
   \otimes\omega^{-1}_{\ms{X}}\,[d]
 \end{equation*}
 of complexes in $D^b(\DtildQ{m}{\ms{X}})$. Here, the right module
 structure of
 $\DtildQ{m}{\ms{X}}\otimes^{\mb{L}}_{\BcompQ{m}{\ms{X}}}\ms{N}$ is
 defined by that of $\DtildQ{m}{\ms{X}}$, and the left structure by
 {\normalfont \cite[1.1.7]{Ber2}}.
 Moreover, this isomorphism is
 compatible with Frobenius.
\end{lem*}

\begin{proof}
 By \cite[4.4.2]{Ber1}, $\ms{N}$ is a locally projective
 $\BcompQ{m}{\ms{X}}$-module, and we do not need to take the derived
 tensor products. Let $\ms{M}'$ be a {\em right}
 $\DtildQ{m}{\ms{X}}$-module. Then
 $\ms{M}'\otimes_{\BcompQ{m}{\ms{X}}}\ms{N}$ possesses a right
 $\DtildQ{m}{\ms{X}}$-module structure. Indeed let $\ms{U}$ be an affine
 open formal subscheme of $\ms{X}$. Let $M':=\Gamma(\ms{U},\ms{M}')$,
 $N:=\Gamma(\ms{U},\ms{N})$, $B:=\Gamma(\ms{U},\BcompQ{m}{\ms{X}})$,
 and $D:=\Gamma(\ms{U},\DtildQ{m}{\ms{X}})$. Then it suffices to define
 a right $D$-module structure on $M'\otimes_{B}N$. For
 $a\otimes b\in M'\otimes_BN$ and $P\in D$, it suffices to define
 $(a\otimes b)\cdot P$. Take $S\subset M'$ to be the finite
 $\Gamma(\ms{U},\DtildQ{m}{\ms{X}})$-submodule generated by $a$. Then
 $S\otimes N$ is naturally a $D$-module considering
 \cite[4.4.7]{Ber1} and \cite[1.1.7]{Ber2}. This defines $(a\otimes
 b)\cdot P$.

 Now back to the proof, by using \cite[I, 1.2.2]{Vir}, we get an
 isomorphism $\kappa$ of complexes in $D^b(\BcompQ{m}{\ms{X}})$.
 It suffices to show that this isomorphism is an isomorphism of
 $\DtildQ{m}{\ms{X}}$-complexes. For a bi-$\DtildQ{m}{\ms{X}}$-module
 $\ms{I}$, we have a canonical homomorphism
 \begin{equation*}
  \shom_{\DtildQ{m}{\ms{X}}}(\ms{M},\ms{I})\otimes_{\BcompQ{m}{\ms{X}}}
   \ms{N}\rightarrow\shom_{\DtildQ{m}{\ms{X}}}(\ms{M},\ms{I}
   \otimes_{\BcompQ{m}{\ms{X}}}\ms{N}).
 \end{equation*}
 {\it A priori}, this is a homomorphism of
 $\BcompQ{m}{\ms{X}}$-modules. By the argument above, both sides of the
 homomorphism possess the right $\DtildQ{m}{\ms{X}}$-module
 structures. To finish the proof, is suffices to see that the
 homomorphism is $\DtildQ{m}{\ms{X}}$-linear. The verification is
 straightforward.
\end{proof}

\subsection{}
Now, consider the situation \ref{Frobsituations}.\ref{Frobstr}.
Let $\ms{X}$ be a smooth formal scheme over
$\mr{Spf}(R)$, and $Z$ be a divisor of the special fiber. Let $\ms{N}'$
be a coherent $F$-$\DdagQ{\ms{X}}(^\dag Z)$-module which is also
coherent as an $\mc{O}_{\ms{X},\mb{Q}}(^\dag Z)$-module. By abuse of
language, we say that $\ms{N}'$ is a convergent $F$-isocrystal
overconvergent along $Z$ (cf.\ Notation \ref{abuselangocisoc}). We put
$N':=\mr{sp}^*(\ms{N}')$, which is an
overconvergent $F$-isocrystal
in the usual sense. Then thanks to Corollary
\ref{compduals}, we get the following isomorphisms of bimodules. We omit
subscripts $\mb{Q}$ and denote by $(Z)$ instead of $(^\dag Z)$ in the
next equality to save the space.
\begin{align}
 \label{isom3}
 &R\shom_{\mc{O}_{\ms{X}}(Z)}
 (\mb{D}_{\ms{X},Z}(\ms{N}'),\Ddag{\ms{X}}(Z))\\
 \notag&\qquad\cong
 \Ddag{\ms{X}}(Z)\otimes_{\mc{O}_{\ms{X}}(Z)}
 R\shom_{\mc{O}_{\ms{X}}(Z)}(\mr{sp_*}(N'^\vee)(-d),
 \mc{O}_{\ms{X}}(Z))\cong
 \Ddag{\ms{X}}(Z)\otimesddag_{\mc{O}_{\ms{X}}(Z)}\ms{N}'(d)
\end{align}
Here the right module structure of the first module is defined by
\cite[1.1.7]{Ber2} using the right module structure of
$\DdagQ{\ms{X}}(^\dag Z)$, and the left structure by using that of
$\DdagQ{\ms{X}}(^\dag Z)$. This is compatible with Frobenius, which
means that the following diagram of canonical isomorphisms is
commutative. We again omit $\mb{Q}$ and denote by $(Z)$ in the
following.
\begin{equation*}
 \xymatrix{
  F^\flat R\shom_{\mc{O}_{\ms{X}'}(Z)}(\mb{D}_{\ms{X},Z}(\ms{N}'),
  \Ddag{\ms{X}'}(Z))\ar[d]
  \ar[r]&F^\flat\bigl(\Ddag{\ms{X}'}(Z)\otimesddag_{\mc{O}_{\ms{X}'}(Z)}
  \ms{N}'(d)\bigr)\ar[d]\\
  R\shom_{\mc{O}_{\ms{X}}(Z)}(\mb{D}_{\ms{X},Z}(F^*\ms{N}'),
  F^\flat\Ddag{\ms{X}'}(Z))\ar[r]&
  F^\flat\Ddag{\ms{X}'}(Z)\otimesddag_{\mc{O}_{\ms{X}}(Z)}
  F^*\ms{N}'(d)
 }
\end{equation*}
Here the homomorphisms are the canonical ones except for the right
vertical homomorphism, which is $q^{-d}$ times the canonical
homomorphism.

Let $\ms{M}$ be an object in $F$-$D^b_{\mr{perf}}(\DdagQ{\ms{X}}(^\dag
Z))$. We get the following isomorphisms compatible with Frobenius
structures:
\begin{align}
 \label{commdual}
 &\mb{D}_{\ms{X},Z}(\ms{M})\otimesddag_{\mc{O}_{\ms{X},\mb{Q}}
 (^\dag Z)}\ms{N}'\cong R\shom_{\DdagQ{\ms{X}}(^\dag
 Z)}(\ms{M},\DdagQ{\ms{X}}(^\dag Z)\otimesddag_{\mc{O}_{\ms{X},\mb{Q}}
 (^\dag Z)}\ms{N}')\otimes\omega_{\ms{X}}^{-1}[d]\\
 \notag&\qquad\cong R\shom_{\DdagQ{\ms{X}}(^\dag Z)}\bigl(
 \ms{M},R\shom_{\mc{O}_{\ms{X},\mb{Q}}(^\dag Z)}
 (\mb{D}_{\ms{X},Z}(\ms{N}'),\DdagQ{\ms{X}}(^\dag Z))\bigr)
 \otimes\omega_{\ms{X}}^{-1}(-d)[d]\\
 \notag&\qquad\cong R\shom_{\DdagQ{\ms{X}}(^\dag Z)}\bigl(
 \ms{M}\otimesddag_{\mc{O}_{\ms{X},\mb{Q}}(^\dag Z)}
 \mb{D}_{\ms{X},Z}(\ms{N}'),\DdagQ{\ms{X}}(^\dag Z)\bigr)
 \otimes\omega_{\ms{X}}^{-1}(-d)[d]\\
 \notag&\qquad\cong\mb{D}_{\ms{X},Z}(\ms{M}\otimesddag
 _{\mc{O}_{\ms{X},\mb{Q}}(^\dag Z)}\mb{D}_{\ms{X},Z}(\ms{N}'))(-d)
\end{align}
where the first isomorphism by Lemma \ref{isom1}, the second by
(\ref{isom3}), and the third by using \cite[2.1.34]{CarD}. Now, we get
the following proposition.

\begin{prop}
 \label{tenstwist}
 Let $\ms{X}$ be a smooth formal scheme over
 $\mr{Spf}(R)$, and $Z$ be a divisor of its special fiber. Let $\ms{M}$
 be a coherent $F$-$\DdagQ{\ms{X}}(^\dag Z)$-module, and $\ms{N}$ be an
 overconvergent $F$-isocrystal along $Z$. Then we get
 \begin{equation}
  \label{twodualfunequ}
  (\mb{D}_{\ms{X},Z}(\ms{M})\otimesddag_{\mc{O}_{\ms{X},\mb{Q}}
   (^\dag Z)}\mb{D}_{\ms{X},Z}(\ms{N}))(d)\cong
   \mb{D}_{\ms{X},Z}(\ms{M}\otimesddag_{\mc{O}_{\ms{X},\mb{Q}}
   (^\dag Z)}\ms{N})
 \end{equation}
 which is compatible with Frobenius structures.
\end{prop}
\begin{proof}
 Take $\ms{N}':=\mb{D}_{\ms{X},Z}(\ms{N})$, and the isomorphism
 (\ref{commdual}) induces the isomorphism we are looking for.
\end{proof}

\begin{rem*}
 We are {\em not} able to expect the isomorphism (\ref{twodualfunequ})
 in general. For example, consider the closed immersion
 $i\colon\{0\}\hookrightarrow\widehat{\mb{A}}_R$. Then taking 
 $\ms{M}=\ms{N}=i_+K$, we do not have such an isomorphism.
\end{rem*}

\nsubsection*{The K\"{u}nneth formula}
\subsection{}
Now, we will show the K\"{u}nneth formula. The ideas used to show the
formula is due to P. Berthelot.
Let $S$ be a $\mb{Z}_{(p)}$-scheme, and $X$, $Y$, $T$
be smooth formal schemes over $S$. Consider the following commutative
diagram:
\begin{equation}
 \label{kunnethdiag}
 \xymatrix@R=10pt@C=30pt{
  &Z\ar[dr]^{q_Y}\ar[dl]_{q_X}\ar[dd]_p&
  \\X\ar[dr]_{p_X}&&Y\ar[dl]^{p_Y}\\
 &T&}
\end{equation}
where $Z=X\times_TY$. Suppose that $p_X$ and $p_Y$ are smooth.

\subsection{}
\label{calckunneth}
Let $\mc{B}_T$, (resp.\ $\mc{B}_X$, $\mc{B}_Y$) be a commutative
$\mc{O}_T$-algebra (resp.\ $q^*_X\mc{B}_T$-algebra,
$q^*_Y\mc{B}_Y$-algebra) endowed with an action of $\Dmod{m}{T}$ (resp.\
$\Dmod{m}{X}$, $\Dmod{m}{Y}$) compatible with that of $\mc{O}_T$ (resp.\
$\mc{O}_X$, $\mc{O}_Y$). For
a $\mc{B}_X$-module $\ms{F}$ and a $\mc{B}_Y$-module $\ms{G}$, we put
$\ms{F}\boxtimes_{\mc{B}_T}\ms{G}:=q_X^*\ms{F}\otimes_{p^*\mc{B}_T}
q^*_Y\ms{G}$. Let
$\mc{B}_Z:=\mc{B}_X\boxtimes_{\mc{B}_T}\mc{B}_Y$. We
put $\Dtmod{m}{*}:=\mc{B}_*\otimes\Dmod{m}{*}$ where
$*\in\{X,Y,Z\}$. We note that $p^*\Dtmod{m}{T}$ is a
sub-$\mc{O}_{Z}$-algebra of $\Dtild{m}{Z}$. We get the following lemma.

\begin{lem*}
 We preserve the notation.

 (i)  There exists the canonical isomorphism
 \begin{equation*}
  \Dtmod{m}{Z}\cong q_X^{*}\Dtmod{m}{X}\otimes_{p^*\Dtmod{m}{T}}
   q_Y^{*}\Dtmod{m}{Y}.
 \end{equation*}

 (ii) There exists the canonical isomorphism
 \begin{equation*}
  \Dtmod{m}{T\leftarrow Z}\cong q_X^{*}\Dtmod{m}{T\leftarrow X}
   \otimes_{p^*\Dtmod{m}{T}}q_Y^{*}\Dtmod{m}{T\leftarrow Y}.
 \end{equation*}
\end{lem*}
\begin{proof}
 The natural homomorphisms $q^*_X\Dtmod{m}{X}\rightarrow\Dtmod{m}{Z}$ and
 $q_Y^*\Dtmod{m}{Y}\rightarrow\Dtmod{m}{Z}$ induces the homomorphism
 $q_X^{*}\Dtmod{m}{X}\otimes_{p^*\Dtmod{m}{T}}q_Y^{*}\Dtmod{m}{Y}
 \rightarrow\Dtmod{m}{Z}$. To see that this is an isomorphism, we may
 assume that $T$ possesses a system of local coordinate and $X$ and $Y$
 possesses a system of local coordinate over $T$. Then proof is
 straightforward, and we leave the reader for the detail. Let us see
 (ii). We know the following isomorphisms
 \begin{equation*}
  \omega_{Z/T}\cong q_X^*\omega_{X/T}\otimes_{\mc{O}_Z}q_Y^*
   \omega_{Y/T}.
 \end{equation*}
 We get
 \begin{align*}
  &\Dtmod{m}{T\leftarrow Z}\cong p^*\Dtmod{m}{T}
  \otimes_{\mc{O}_Z}\omega_{Z/T}\\
  &\qquad\cong\Bigl(q^*_Xp_X^*\Dtmod{m}{T}\otimes_{p^*\Dtmod{m}{T}}
  q^*_Yp_Y^*\Dtmod{m}{T}\Bigr)\otimes_{\mc{O}_Z}
  (q_X^{*}\omega_{X/T}\otimes q_Y^*\omega_{Y/T})\\
  &\qquad\cong q_X^*\Dtmod{m}{T\leftarrow X}\otimes_{p^*\Dtmod{m}{T}}
  q_Y^{*}\Dtmod{m}{T\leftarrow Y}.
 \end{align*}
\end{proof}

\begin{lem}
 \label{flatlemkunneth}
 Let $\ms{M}$ be a left flat $\Dtmod{m}{X}$-module, and $\ms{N}$ be a
 left flat $\Dtmod{m}{Y}$-module. Then we get that
 $\ms{M}\boxtimes_{\mc{B}_T}\ms{N}$ is a flat left
 $\Dtmod{m}{Z}$-module.
\end{lem}
\begin{proof}
 In the case where $T=\mr{Spf}(R)$, the verification is left to the
 readers. To see the lemma, since the verification is local, we may
 assume that $T$ is affine. Let $i\colon Z=X\times_TY\hookrightarrow
 W:=X\times Y$ be the canonical inclusion. Since $T$ is separated, this
 is a closed immersion. By the $T=\mr{Spf}(R)$ case, we get that
 $\ms{M}\boxtimes_S\ms{N}$ is a flat $\Dtmod{m}{W}$-module.
 We put
 $\Dtmod{m}{Z\rightarrow W}:=i^*\Dtmod{m}{W}$, which is a
 $(\Dtmod{m}{Z},i^{-1}\Dtmod{m}{W})$-module as usual. Let $\ms{F}$ be a
 right $\Dtmod{m}{Z}$-module. Then we get
 \begin{equation*}
  \ms{F}\otimes_{\Dtmod{m}{Z}}(\ms{M}\boxtimes_{\mc{B}_T}\ms{N})
   \cong(\ms{F}\otimes_{\Dtmod{m}{Z}}\Dtmod{m}{Z\rightarrow
   W})\otimes_{i^{-1}\Dtmod{m}{W}}i^{-1}(\ms{M}\boxtimes_S
   \ms{N}).
 \end{equation*}
 Since $\Dtmod{m}{Z\rightarrow W}$ is flat over $\Dtmod{m}{Z}$ and
 $\ms{M}\boxtimes_S\ms{N}$ is flat over $\Dtmod{m}{W}$, we get the
 lemma.
\end{proof}

Using this preparation, we get the following K\"{u}nneth formula.
\begin{prop}
 We preserve the notation. Let $\ms{M}$ (resp.\ $\ms{N}$) be a complex
 in $D^-_{\mr{qc}}(\Dtmod{m}{X})$ (resp.\
 $D^-_{\mr{qc}}(\Dtmod{m}{Y})$). Then we get a canonical isomorphism in
 $D^-_{\mr{qc}}(\Dtmod{m}{T})$
 \begin{equation}
  \label{kunnethisom}
  p_+(\ms{M}\boxtimes^{\mb{L}}_{\mc{B}_T}\ms{N})\cong p_{X+}(\ms{M})
  \otimes^{\mb{L}}_{\mc{B}_T}p_{Y+}(\ms{N}).
 \end{equation}
\end{prop}
\begin{proof}
 Let $\ms{F}$ (resp.\ $\ms{G}$) be a quasi-coherent $\Dtmod{m}{X}$-module
 (resp.\ $\Dtmod{m}{Y}$-module). Then by Lemma \ref{calckunneth}, we get
 a canonical isomorphism of $p^{-1}\Dtmod{m}{T}$-modules
 \begin{equation}
  \label{isomextprodmod}
   \Dtmod{m}{T\leftarrow Z}\otimes_{\Dtmod{m}{Z}}(\ms{F}
   \boxtimes_{\mc{B}_T}\ms{G})\cong(\Dtmod{m}{T\leftarrow X}
   \otimes_{\Dtmod{m}{X}}\ms{F})\boxtimes_{\mc{B}_T}
   (\Dtmod{m}{T\leftarrow Y}\otimes_{\Dtmod{m}{Y}}\ms{G}).
 \end{equation}
 Let $\ms{L}_\bullet$ be a flat resolution of $\ms{F}$ as a
 $\Dtmod{m}{X}$-module, and $\ms{M}_\bullet$ be a flat resolution of
 $\ms{G}$ as a $\Dtmod{m}{Y}$-module. Then we get that
 $\ms{L}_i\boxtimes_{\mc{B}_T}\ms{M}_j$ is a flat $\Dtmod{m}{Z}$-module
 for any $i$ and $j$ by Lemma \ref{flatlemkunneth}. Thus we get
 \begin{align*}
  &\Dtmod{m}{T\leftarrow Z}\otimes^{\mb{L}}_{\Dtmod{m}{Z}}(\ms{F}
   \boxtimes^{\mb{L}}_{\mc{B}_T}\ms{G})\cong
  \Dtmod{m}{T\leftarrow Z}\otimes_{\Dtmod{m}{Z}}(\ms{L}_\bullet
   \boxtimes^{\mb{L}}_{\mc{B}_T}\ms{M}_\bullet)\\
  &\qquad\cong(\Dtmod{m}{T\leftarrow X}\otimes_{\Dtmod{m}{X}}
  \ms{L}_\bullet)\boxtimes_{\mc{B}_T}(\Dtmod{m}{T\leftarrow Y}
   \otimes_{\Dtmod{m}{Y}}\ms{M}_\bullet)\\
  &\qquad\cong(\Dtmod{m}{T\leftarrow X}\otimes^{\mb{L}}_{\Dtmod{m}{X}}
  \ms{F})\boxtimes^{\mb{L}}_{\mc{B}_T}(\Dtmod{m}{T\leftarrow Y}
   \otimes^{\mb{L}}_{\Dtmod{m}{Y}}\ms{G}).
 \end{align*}
 By using the K\"{u}nneth formula for quasi-coherent sheaves, we get the
 proposition.
\end{proof}

\subsection{}
Now, let us consider Frobenius. Suppose $T$ is endowed with a
quasi-coherent $m$-PD-ideal $(\mf{a},\mf{b},\alpha)$, and
$p\in\mf{a}$. With this hypothesis, we are able to
consider Frobenius pull-back even if there are no liftings of relative
Frobenius morphisms.

\begin{lem*}
 Suppose that $X_0$ and $Y_0$ be the reductions of $X$ and $Y$
 respectively, and $X'$, $Y'$ be liftings of $X_0^{(s)}$,
 $Y_0^{(s)}$. We take $Z':=X'\times_TY'$. Then the isomorphism
 {\normalfont(\ref{kunnethisom})} is compatible with Frobenius
 isomorphisms. Moreover, if $\mf{a}$ is $m$-PD-nilpotent, it is
 compatible with Frobenius isomorphism even if there are no liftings.
\end{lem*}
\begin{proof}
 The verification uses only standard arguments, so we leave the details
 to the readers.
\end{proof}

\subsection{}
Finally, by taking inverse limit and inductive limit, we get the
following K\"{u}nneth formula for $\ms{D}^\dag$-modules.

\begin{prop*}
 Consider the following diagram
 \begin{equation*}
   \xymatrix@R=10pt@C=30pt{
  &\ms{Z}\ar[dr]^{q_Y}\ar[dl]_{q_X}\ar[dd]_p&
  \\\ms{X}\ar[dr]_{p_X}&&\ms{Y}\ar[dl]^{p_Y}\\
 &\ms{T}&}
 \end{equation*}
 where $\ms{T}$ is a smooth formal scheme, $p_X$ and $p_Y$ are smooth,
 and $\ms{Z}:=\ms{X}\times_{\ms{T}}\ms{Y}$. Let $D$ be a divisor of the
 special fiber of $\ms{T}$, $D_X$ (resp.\ $D_Y$) be a divisor of the
 special fiber of $\ms{X}$ (resp.\ $\ms{Y}$) such that $D_X\supset
 p^{-1}_X(D)$ (resp.\ $D_Y\supset p^{-1}_Y(D)$). Let $\ms{M}$ (resp.\
 $\ms{N}$) be a complex in
 $\underrightarrow{LD}^b_{\mb{Q},\mr{qc}}(\widehat{\ms{D}}
 ^{(\bullet)}_\ms{X}(D_X))$ (resp.\ $\underrightarrow{LD}^b
 _{\mb{Q},\mr{qc}}(\widehat{\ms{D}}^{(\bullet)}_\ms{Y}(D_Y))$). Then
 we get the canonical isomorphism
 \begin{equation*}
    p_+(\ms{M}\btimesdag{}^\dag_{\mc{O}_{\ms{T}}(^\dag D)}\ms{N})
     \cong p_{X+}(\ms{M})\otimesdag{}^\dag_{\mc{O}_{\ms{T}}(^\dag D)}
     p_{Y+}(\ms{N})
 \end{equation*}
 in $\underrightarrow{LD}^b_{\mb{Q},\mr{qc}}(
 \widehat{\ms{D}}^{(\bullet)}_\ms{T}(D))$. This isomorphism is
 compatible with the Frobenius isomorphisms.
\end{prop*}

\begin{rem}
 To get the proposition directly, we can also proceed as
 follows. Consider the following cartesian diagram.
 \begin{equation*}
  \xymatrix{
   \ms{X}\times_\ms{T}\ms{Y}\ar[r]^i\ar[d]_p\ar@{}[rd]|{\square}&
   \ms{X}\times\ms{Y}\ar[d]^q\\
  \ms{T}\ar[r]_{i_T}&\ms{T}\times\ms{T}}
 \end{equation*}
 We see easily that $q_+(\ms{M}\boxtimes^{\mb{L}}\ms{N})\cong
 q_+(\ms{M})\boxtimes^{\mb{L}}_{\mc{O}_{\ms{T}\times\ms{T}}}q_+(\ms{N})$. Let
 $\ms{S}:=\mr{Spf}(R)$. By using this and \cite[2.1.9]{Caroc}, we get
 \begin{equation*}
  i_{T+}p_+i^!(\ms{M}\btimesdag{}^\dag_{\ms{S}}\ms{N})\cong
   i_{T+}i_T^!(q_{X+}\ms{M}\btimesdag{}^\dag_{\ms{S}}q_{Y+}\ms{N})
   \cong i_{T+}(q_{X+}\ms{M}\otimesdag{}^\dag_{\mc{O}_{\ms{T}}(^\dag D)}
   q_{Y+}\ms{N}).
 \end{equation*}
 Taking $H^0i_T^!$, we get what we want.
\end{rem}


\nsubsection*{Compatibility of Frobenius pull-backs with relative
duality}

\subsection{}
\label{settingreldual}
We will show that the relative duality homomorphism by Virrion is
compatible with Frobenius pull-back.

First, let us fix the situation. We consider the situation
\ref{Frobsituations}.\ref{noFrobstr}. Let
$f\colon\ms{X}\rightarrow\ms{Y}$ be a
{\it proper} morphism of smooth formal schemes, and $W$ be a divisor of
the special fiber of $\ms{Y}$ such that $Z:=f^{-1}(W)$ is a
divisor. Under this situation, Virrion \cite{Vir2} defined the trace
homomorphism
\begin{equation}
 \label{tracemapvir}
 \mr{Tr}_{+,f}\colon f_+\omega_{\ms{X},\mb{Q}}[d_X]\rightarrow
  \omega_{\ms{Y},\mb{Q}}[d_Y]
\end{equation}
where $d_X$ and $d_Y$ denotes the dimension of $\ms{X}$ and
$\ms{Y}$. Using the trace map, for an object $\ms{E}$ in
$D^b_{\mr{perf}}(\DdagQ{\ms{X}}(^\dag Z))$, she also constructed the
relative duality isomorphism
\begin{equation}
 \label{reldualhom}
 \chi\colon\mb{D}_{\ms{Y},W}f_+(\ms{E})\xrightarrow{\sim}f_+
 \mb{D}_{\ms{X},Z}(\ms{E})
\end{equation}
in $D^b_{\mr{perf}}(\DdagQ{\ms{Y}}(^\dag W))$ (see also
\cite[1.2.7]{CarL}). Now, we assume that there exist liftings
$\ms{X}'$ and $\ms{Y}'$ of $X_0^{(s)}$ and $Y_0^{(s)}$ where $X_0$ and
$Y_0$ are special fibers as usual. We also assume that there exists a
lifting $f'\colon\ms{X}'\rightarrow\ms{Y}'$ of the morphism
$X_0^{(s)}\rightarrow Y_0^{(s)}$ induced by $f$. These assumptions
automatically hold when we consider the situation
\ref{Frobsituations}.\ref{Frobstr}.

\subsection{}
Before stating the theorem, we will prepare a commutative
diagram, which is needed in the proof of the compatibility. We freely
use the notation of \cite{Har}. Let $X$, $X'$, $Y$, $Y'$ be locally
noetherian schemes. Suppose we are given the following commutative
diagram.
\begin{equation*}
 \xymatrix{
  X'\ar[r]^{f'}\ar[d]_{u'}&Y'\ar[d]^u\\
 X\ar[r]_f&Y
  }
\end{equation*}
We assume that all the morphisms are proper, $u$ is finite flat, and
all the schemes admit dualizing complexes (cf.\ \cite[V, \S 2]{Har}). By
the transitivity of trace map \cite[VI, 4.2(a) TRA1]{Har}, we get the
following commutative diagram.
\begin{equation*}
 \xymatrix@C=70pt@R=20pt{
  u_*f'_*f'^{\triangle}u^{\triangle}\ar[r]^{\mr{Tr}_{f'}}
  \ar[d]_{\mr{Tr}_{u'}}&u_*u^\triangle\ar[d]^{\mr{Tr}_{u}}\\
 f_*f^{\triangle}\ar[r]_{\mr{Tr}_{f}}&\mr{id}
  }
\end{equation*}
Here, $\mr{Tr}_{u'}$ denotes the composition $u_*f'_*f'^{\triangle}
u^{\triangle}\cong f_*u'_*u'^{\triangle}f^{\triangle}\xrightarrow
{f_*\circ\mr{Tr}_{u'}\circ f^\triangle}f_*f^{\triangle}$.
We remind that this diagram consists of homomorphisms in the category of
{\it complexes} by \cite[VII, 2.1]{Har}.
Let $u^{\cdot}$ denotes the {\it functor} $u'$ in \cite[VI, 4.1]{Har} to
avoid confusions with the {\it morphism} $u'$. We note
that $u^\cdot\cong u^\flat\cong u^!$ in the derived category since $u$
is finite and flat. We have the canonical homomorphism
$\mr{id}\rightarrow u^\cdot u_*$. Let
$c\colon f'_*f'^{\triangle}u^{\triangle}\rightarrow u^\cdot u_*f'_*
f'^{\triangle}u^{\triangle}\rightarrow u^\cdot f_*f^{\triangle}$
where the second morphism is that induced by the left vertical morphism
$\mr{Tr}_{u'}$ in the diagram above. Let $c'\colon
u^{\triangle}\rightarrow u^\cdot u_*u^{\triangle}\rightarrow u^\cdot$
where the second homomorphism is the trace map. Taking $u^{\cdot}$ to
the above diagram, we get the following commutative diagram of
complexes.
\begin{equation}
 \label{dualtr}
 \xymatrix@C=70pt@R=20pt{
  f'_*f'^{\triangle}u^{\triangle}\ar[r]^{\mr{Tr}_{f'}\circ
  u^{\triangle}}\ar[d]_{c}&u^{\triangle}\ar[d]^{c'}\\
 u^\cdot f_*f^{\triangle}\ar[r]_{u^{\cdot}\circ\mr{Tr}_{f}}
  &u^\cdot}
\end{equation}
The morphism $\mr{id}\rightarrow u^\cdot u_*\cong u^\flat u_*$ is
nothing but the adjunction homomorphism, and $c'$ is the identity in the
derived category.

\begin{thm}
 \label{virdualcom}
 We preserve the assumptions and notations of paragraph
 {\normalfont\ref{settingreldual}}.
 Then the relative duality isomorphism {\normalfont (\ref{reldualhom})}
 is compatible with Frobenius pull-back.
\end{thm}
\begin{proof}
 Recall that $d_X$ (resp.\ $d_Y$) was the dimension of $\ms{X}$ (resp.\
 $\ms{Y}$). We let $d$ to be the relative dimension, namely
 $d:=d_X-d_Y$. We denote by $f^{(m)}_+$ the push-forward of level
 $m$. We will use the push-forward for right modules as in
 \cite[III]{Vir2}: recall that $f_+^{(m)}(\ms{M}):=
 Rf_*(\ms{M}\otimes^{\mb{L}}_{\DcompQ{m}{\ms{X}}}\DcompQ{m}
 {\ms{X}\rightarrow\ms{Y}})$ for a
 coherent right $\DcompQ{m}{\ms{X}}$-module $\ms{M}$. It suffices to
 show that the trace homomorphism (\ref{tracemapvir}) is compatible with
 Frobenius pull-backs by \cite[1.5]{CarDF}. By the result of Caro
 \cite[2.2.7]{CarDF}, we know the compatibility in
 the case where $f$ is a closed immersion. Thus, we are reduced to
 showing the case where $f$ is smooth using the standard
 factorization $\ms{X}\rightarrow\ms{X}\times\ms{Y}\rightarrow\ms{Y}$
 and the transitivity of the trace map \cite[III, 5.5]{Vir2}. In the
 following, we assume $f$ to be smooth.

 Let $i$ be a non-negative integer, and $X$ and $Y$ be the reductions of
 $\ms{X}$ and $\ms{Y}$ over $R_i$.
 It suffices to show the compatibility for these $X$ and
 $Y$, namely we need to prove the commutativity of the following
 diagram.
 \begin{equation*}
  \xymatrix@C=70pt{
   f^{(m+s)}_+\omega_{X}[d_X]\ar[r]^<>(.5){\mr{Tr}_{+,f}}
   \ar[d]&\omega_{Y}[d_Y]\ar[d]\\
  F_Y^{\flat}f'^{(m)}_+\omega_{X'}[d_X]\ar[r]_<>(.5)
   {F_Y^{\flat}\mr{Tr}_{+,f'}}&F_Y^{\flat}\omega_{Y'}[d_Y]}
 \end{equation*}
 Since $f$ is smooth and the relative dimension is $d$, we get that for
 any point $y$ in $\ms{Y}$, the dimension of $f^{-1}(y)$ is equal to
 $d$. Since $f$ is proper, $R^if_*(\ms{F})=0$ for any  quasi-coherent
 sheaf $\ms{F}$ and $i>d$ (cf.\ for example \cite[I Lemma
 (3)]{Kl}). Thus, by the definition of $f_+$, the canonical homomorphism
 \begin{equation*}
  R^df_*\omega_{X}\twoheadrightarrow H^df^{(m)}_+
   \omega_{X}
 \end{equation*}
 is surjective, and $H^if^{(m)}_+\omega_X=0$ for $i>d$, in other words
 $H^i(f^{(m)}_+\omega_{\ms{X},\mb{Q}}[d_X])=0$ for $i>-d_Y$.
 Thus, we get
 \begin{equation*}
  \mr{Hom}(f^{(m)}_+\omega_{X}[d_X],\omega_{Y}[d_Y])
   \cong\mr{Hom}(H^df^{(m)}_+\omega_{X},\omega_{Y}).
 \end{equation*}
 This implies that to give the trace map
 $f^{(m)}_+\omega_X[d_X]\rightarrow\omega_{Y}[d_Y]$ is equivalent to
 giving the homomorphism
 $H^df^{(m)}_+\omega_{X}\rightarrow\omega_{Y}$. We can retrieve the
 trace map by the composition
 \begin{equation*}
  f^{(m)}_+\omega_{X}[d_X]\rightarrow H^df^{(m)}_+\omega_{X}
   [d_Y]\rightarrow\omega_{Y}[d_Y].
 \end{equation*}
 Note that we have the following commutative diagram by \cite[III,
 5.4]{Vir2}.
 \begin{equation*}
  \xymatrix{
   R^df_*\omega_{X}\ar[rr]\ar[dr]_{\mr{Tr}_f}&&
   H^df^{(m)}_+\omega_{X}\ar[dl]^{\mr{Tr}_{+,f}}\\&
   \omega_{Y}&
   }
 \end{equation*}
 It suffices to see that the following diagram is commutative.
 \begin{equation*}
  \xymatrix{
   H^{d}f^{(m+s)}_+\omega_{X}\ar[r]\ar[d]&\omega_{Y}
   \ar[d]\\F_Y^\flat H^{d}f'^{(m)}_+\omega_{X'}\ar[r]&F^\flat_Y
   \omega_{Y'}
   }
 \end{equation*}
 This shows that the problem is local with respect to $Y$.
 Assume that there exist liftings $F_X\colon X\rightarrow X'$ and
 $F_Y\colon Y\rightarrow Y'$ of relative Frobenius morphisms such that
 the two morphisms $f'\circ F_X, F_Y\circ f\colon X\rightarrow Y'$
 coincide. Under this particular situation, the theorem is reduced to
 showing the following diagram
 \begin{equation*}
  \xymatrix@C=60pt{
   R^df_*\omega_{X}\ar[r]^<>(.5){\mr{Tr}_f}\ar[d]&
   \omega_{Y}\ar[d]\\F_Y^\flat R^df'_*\omega_{X'}
   \ar[r]_<>(.5){F^\flat_Y\mr{Tr}_{f'}}&F^\flat_Y
   \omega_{Y'}
   }
 \end{equation*}
 is commutative, where the left vertical homomorphism is the base change
 homomorphism. This is nothing but (\ref{dualtr}). In particular, the
 theorem holds in the case where $f$ is finite \'{e}tale.

 Let $x$ be a point of $X$ (which may not be closed). For a right
 $\Dmod{m}{X}$-module $\ms{F}$ and an integer $i$, we put
 \begin{equation*}
   \ms{H}^{(m),i}_{x}(\ms{F}):=\indlim_{x\in U}~j_{U*}\,R^i\Gammas^{(m)}
   _{U\cap\overline{\{x\}}}(\ms{F})
 \end{equation*}
 where $U$ runs over open neighborhoods of $x$, $j_U\colon
 U\hookrightarrow X$ is the inclusion, and $\Gammas^{(m)}$
 denotes the level $m$ local cohomology functor defined in
 \cite[4.4.4]{BerInt}. We note that
 \begin{equation}
  \label{indlimcoinloccoh}
  \indlim_{m'}\ms{H}^{(m'),i}_{x}(\ms{F})\cong
  \ms{H}^{i}_{x}(\ms{F})
 \end{equation}
 by \cite[1.5.4]{Ber1}, and $\ms{H}^{i}_{x}(\omega_X)$ can be seen as a
 quasi-coherent right $\Dmod{m}{X}$-module. The Frobenius isomorphism 
 $\ms{H}^{(m+s),i}_x(\omega_X)\xrightarrow{\sim}F_X^\flat
 \ms{H}^{(m),i}_x(\omega_{X'})$ induces the isomorphism
 \begin{equation*}
  \phi^i_x\colon\ms{H}^i_x(\omega_X)\xrightarrow{\sim}F_X^\flat
   \ms{H}^i_x(\omega_{X'}).
 \end{equation*}
 This induces the following isomorphism of Cousin complexes (cf.\
 \cite[IV, \S 3]{Har}).
 \begin{equation*}
  \xymatrix{
   \omega_X\ar[r]\ar[d]_\sim&\ms{H}^0_{Z_0/Z_1}(\omega_X)\ar[r]
   \ar[d]_\sim^{\sum\phi^0_x}&\ms{H}^1_{Z_1/Z_2}(\omega_X)\ar[r]
   \ar[d]_\sim^{\sum\phi^1_x}&\dots\\
  F_X^\flat\omega_{X'}\ar[r]&F_X^\flat\ms{H}^0_{Z_0/Z_1}(\omega_{X'})
   \ar[r]&F_X^\flat\ms{H}^1_{Z_1/Z_2}(\omega_{X'})\ar[r]&\dots
   }
 \end{equation*}
 We denote by $\ms{C}_X^\bullet$ the Cousin complex of $\omega_X$.
 
 Let $y$ be a point of codimension $i$ in $Y$, and let $x$ be a closed
 point of the fiber $f^{-1}(y)$ in $X$.  For closed subsets of schemes,
 let us endow with the reduced induced scheme structure.
 Then since $f$ is smooth, there exists
 an open subscheme $U$ of $X$, such that $f'\colon
 Z:=\overline{\{x\}}\cap U\rightarrow W:=\overline{\{y\}}\cap V$ is
 finite \'{e}tale where $V:=f(U)$, and $W$ is smooth. Consider the
 following commutative diagram.
 \begin{equation*}
  \xymatrix{
   Z\ar[r]^{i}\ar[d]_{f'}&U\ar[d]^{f}\\W\ar[r]_{i'}&V}
 \end{equation*}
 The trace map $\mr{Tr}_{+,f'}\colon
 f'^{(m)}_+\omega_Z\rightarrow\omega_W$ can be identified with the usual
 trace map $\mr{Tr}_{f'}$ by the isomorphism
 $f'_*\cong f'_+$ since $f'$ is finite \'{e}tale. Since
 $\ms{H}^{(m),i}_Z(\omega_X)\cong i^{(m)}_+(\omega_Z)$ and 
 $\ms{H}^{(m),i}_W(\omega_Y)\cong i'^{(m)}_+(\omega_W)$, the functor
 $i'^{(m)}_+\circ\mr{Tr}_{+,f'}$ induces the homomorphism
 \begin{equation}
  \label{pointtracelevelm}
   H^0f^{(m)}_+\ms{H}^{(m),d+i}_{x}(\omega_X)\rightarrow
   \ms{H}^{(m),i}_{y}(\omega_Y)
 \end{equation}
 by taking inductive limit over $V$. Note that since $f'$ is finite
 \'{e}tale, this trace map is compatible with Frobenius by the result of
 the first part of this proof. By taking the inductive to
 (\ref{pointtracelevelm}) over $m$ and using the identification
 (\ref{indlimcoinloccoh}), we get a homomorphism
 \begin{align}
  \label{pointwisetrace}
  \mr{Tr}^i_x\colon&H^0f^{(m)}_+\ms{H}^{d+i}_{x}(\omega_X)\cong
  \indlim_{m'}H^0f^{(m)}_+\ms{H}^{(m'),d+i}_{x}(\omega_X)\\
  \notag&\qquad\rightarrow
  \indlim_{m'} H^0f^{(m')}_+\ms{H}^{(m'),d+i}_{x}(\omega_X)
  \xrightarrow{(\ref{pointtracelevelm})}\indlim_{m'}
  \ms{H}^{(m'),i}_{y}(\omega_Y)\cong\ms{H}^{i}_{y}(\omega_Y).
 \end{align}
 This homomorphism is compatible with Frobenius as well since
 (\ref{pointtracelevelm}) is. The composition
 \begin{equation}
  \label{compcuspot}
   f_*\ms{H}^{d+i}_{x}(\omega_X)\rightarrow H^0f^{(m)}_+
   \ms{H}^{d+i}_{x}(\omega_X)\xrightarrow{\mr{Tr}^i_x}
   \ms{H}^{i}_{y}(\omega_Y)
 \end{equation}
 is the usual trace map by construction.
 
 There exists the surjection
 $\Dmod{m}{X}\rightarrow\Dmod{m}{X\rightarrow Y}$ sending $1$ to
 $1\otimes1$. This gives us a flat resolution
 $\ms{L}^\bullet\rightarrow\Dmod{m}{X\rightarrow Y}$ such that
 $\ms{L}^0:=\Dmod{m}{X}$ and $\ms{L}^i=0$ for $i>0$. The double complex
 $f_*(\ms{C}^\bullet_X\otimes_{\Dmod{m}{X}}\ms{L}^\bullet)$ induces the
 spectral sequence
 \begin{equation*}
  E^{a,b}_1=f_*H_{-b}(\ms{H}^a_{Z_a/Z_{a+1}}(\omega_X)\otimes
   ^{\mb{L}}\Dmod{m}{X\rightarrow Y})\Rightarrow
   H^{a+b}f^{(m)}_+\omega_X.
 \end{equation*}
 Note that $E^{a,b}_1\cong H_{-b}f^{(m)}_+\ms{H}^a_{Z_a/Z_{a+1}}
 (\omega_X)$. The trace map (\ref{compcuspot}) $E_1^{d+i,0}\rightarrow
 \ms{H}^i_{Z_i/Z_{i+1}}(\omega_Y)$ induces the homomorphism of complexes
 $f_*(\ms{C}_X^{d+i}\otimes\ms{L}_{\bullet})\rightarrow
 \ms{H}^i_{Z_i/Z_{i+1}}(\omega_Y)=\ms{C}^i_Y$, and this induces the
 homomorphism of double complexes
 \begin{equation}
  \label{homdoublecomptra}
   f_*(\ms{C}_X^{d+\bullet}\otimes\ms{L}^{\bullet})\rightarrow
   \ms{C}^{\bullet}_Y.
 \end{equation}
 This homomorphism defines the homomorphism
 $\gamma\colon f_+\omega_X[d_X]\rightarrow\omega_Y[d_Y]$. Let us show
 that $\gamma=\mr{Tr}_{+,f}$.
 The canonical homomorphism
 $\ms{C}_X^\bullet\rightarrow\ms{C}_X^\bullet\otimes\ms{L}^0$ induces the
 homomorphism of double complexes
 \begin{equation}
  \label{homdoublecomptra2}
   f_*\ms{C}^\bullet_X\rightarrow
   f_*(\ms{C}_X^{d+\bullet}\otimes\ms{L}^{\bullet}).
 \end{equation}
 Let 
 \begin{equation*}
  {}_{\mr{I}}E^{a,b}_1:=
   \begin{cases}
    f_*\ms{C}^a_X&\mbox{if $b=0$}\\
    0&\mbox{if $b\neq0$}
   \end{cases}
   \hspace{2cm}
   {}_{\mr{II}}E^{a+d,b}_1:=
  \begin{cases}
   \ms{C}^a_Y&\mbox{if $b=0$}\\
   0&\mbox{if $b\neq0$}.
  \end{cases}
 \end{equation*}
 Then we
 get the trivial spectral sequences ${}_{\mr{I}}E_1^{a,b}\Rightarrow
 R^{a+b}f_*\omega_X$ and
 ${}_{\mr{II}}E_1^{a,b}\Rightarrow{}_{\mr{II}}E^n$ where
 ${}_{\mr{II}}E^d:=\omega_Y$ and $0$ otherwise. The homomorphisms
 (\ref{homdoublecomptra}) and (\ref{homdoublecomptra2}) induce the
 homomorphisms ${}_{\mr{I}}E_1^{a,b}\rightarrow E^{a,b}_1\rightarrow
 {}_{\mr{II}}E_1^{a,b}$ of spectral sequences.
 We get the following homomorphisms of complexes of $E_1$-terms of these
 spectral sequences.
 \begin{equation}
  \label{commtraceres}
  \xymatrix{
   f_*\ms{H}^{d-1}_{Z_{d-1}/Z_{d}}(\omega_X)\ar[r]\ar[d]&
   f_*\ms{H}^{d}_{Z_d/Z_{d+1}}(\omega_X)\ar[r]\ar[d]&
   f_*\ms{H}^{d+1}_{Z_{d+1}/Z_{d+2}}(\omega_X)\ar[d]\ar[r]&\dots&
   Rf_*\omega_X\ar[d]\\
  E^{d-1,0}_1\ar[r]\ar[d]&
   E^{d,0}_1\ar[r]\ar[d]^{\sum\mr{Tr}^0_x}&
   E^{d+1,0}_1\ar[d]^{\sum\mr{Tr}^1_x}\ar[r]&\dots&
   f^{(m)}_+\omega_X\ar[d]^{\gamma}\\
  0\ar[r]&
   \ms{H}^{0}_{Z_0/Z_1}(\omega_Y)\ar[r]&
   \ms{H}^{1}_{Z_1/Z_2}(\omega_Y)\ar[r]&\dots&\omega_Y[-d]
   }
 \end{equation}
 Here the right homomorphisms are the homomorphisms of complexes of
 corresponding spectral sequences.
 To show that $\gamma$ is the trace map, it suffices to show that
 $H^{-d_Y}(\gamma)$ is the trace map. Consider the homomorphisms
 \begin{equation*}
  R^df_*\omega_X\rightarrow H^df^{(m)}_+\omega_X
   \xrightarrow{H^d\gamma}\omega_Y
 \end{equation*}
 induced by the $E^d$-terms of the homomorphism of the spectral
 sequences. The composition is the usual trace map $\mr{Tr}_f$ since
 (\ref{compcuspot}) is (cf.\ \cite[VI, 4.2]{Har} for the construction of
 the classical trace map). Moreover the
 first homomorphism is the natural map which is surjective. Thus the
 second homomorphism is nothing but the trace map of Virrion.
 
 Since it suffices to see the Frobenius compatibility for $H^d\gamma$,
 this shows that it suffices to show the Frobenius compatibility for the
 lower homomorphism of complexes of the diagram
 (\ref{commtraceres}). Thus it is reduced to showing the Frobenius
 compatibility for the homomorphism
 \begin{equation*}
  \sum\mr{Tr}^i_x\colon
   E_1^{d+i,0}\rightarrow\ms{H}^i_{Z_i/Z_{i+1}}(\omega_Y)
 \end{equation*}
 for any $i$. It is enough to show the compatibility for
 $\mr{Tr}^i_x$ for each $x\in X$ and $i$, which we have already
 verified at (\ref{pointwisetrace}).
\end{proof}

\subsection{}
\label{adointcomp}
We preserve the assumptions and notations from paragraph
\ref{settingreldual}.
Let $\ms{M}'$ be an object in $D^b_{\mr{coh}}(\DdagQ{\ms{X}'}(^\dag
Z))$, and $\ms{N}'$ be an object in $D^b_{\mr{coh}}(\DdagQ{\ms{Y}'}
(^\dag W))$. We put $\ms{M}:=F_X^*\ms{M}$ and $\ms{N}:=F^*_X\ms{N}'$.
Since $F^*_Y$ induces an equivalence between $D^b_{\mr{coh}}
(\DdagQ{\ms{Y}'}(^\dag W))$ and $D^b_{\mr{coh}}(\DdagQ{\ms{Y}}
(^\dag W))$, we get
\begin{align*}
 R\mr{Hom}_{\DdagQ{\ms{Y}}(^\dag W)}(f_+(\ms{M}),\ms{N})
  &\xrightarrow{\sim}R\mr{Hom}_{\DdagQ{\ms{Y}}(^\dag W)}
  (F_Y^*f_+(\ms{M}'),F_Y^*\ms{N}')\\&\xrightarrow{\sim}
  R\mr{Hom}_{\DdagQ{\ms{Y}'}(^\dag W)}(f'_+(\ms{M}'),\ms{N}')
\end{align*}
where the first isomorphism is induced by the isomorphism of functors
$F^*_Y\circ f'_+\cong f_+\circ F^*_X$. In the same way, we get an
isomorphism $R\mr{Hom}_{\DdagQ{\ms{X}}(^\dag Z)}(\ms{M},f^!\ms{N})
\xrightarrow{\sim}R\mr{Hom}_{\DdagQ{\ms{X}}(^\dag Z)}(\ms{M}',f^!
\ms{N}')$.

\begin{cor*}
 We preserve the assumptions and notations. The adjoint isomorphism is
 compatible with Frobenius, in other words, the following diagram
 is commutative.
 \begin{equation*}
  \xymatrix{
   R\mr{Hom}_{\DdagQ{\ms{Y}}(^\dag W)}(f_+(\ms{M}),\ms{N})\ar[d]\ar[r]
   &R\mr{Hom}_{\DdagQ{\ms{X}}(^\dag Z)}(\ms{M},f^!\ms{N})\ar[d]\\
  R\mr{Hom}_{\DdagQ{\ms{Y}}(^\dag W)}(f'_+(\ms{M}'),\ms{N}')\ar[r]&
   R\mr{Hom}_{\DdagQ{\ms{X}}(^\dag Z)}(\ms{M}',f^!\ms{N}')
   }
 \end{equation*}
 Here the horizontal isomorphism are the adjoint formula isomorphism
 {\normalfont \cite[IV, 4.2]{Vir}}, and the vertical isomorphisms are
 those we have just defined.
\end{cor*}
\begin{proof}
 We only need to check the compatibility with Frobenius of the
 isomorphisms used in the proof of \cite[IV, 4.1]{Vir2}. The
 compatibility of the isomorphism \cite[IV, 1.1 (i)]{Vir2} is nothing
 but \cite[2.1.19]{CarD}. The compatibility of \cite[IV, 3.4]{Vir2} is
 Theorem \ref{virdualcom}.
\end{proof}

\section{Cohomological operations in arithmetic $\ms{D}$-modules}
\label{sixfunctor}
In this last section, we will collect results on six operations in
the theory of arithmetic $\ms{D}$-modules with Frobenius structures in
the liftable case. Before starting, recall the notations and
terminologies of paragraph \ref{d-couplesdef}.

\begin{quote}
 In this section, any (formal) scheme is assumed to be of {\em finite
 type} over its basis.
\end{quote}
\subsection{}
In this section, we consider the situation
\ref{Frobsituations}.\ref{noFrobstr} if we do not consider the Frobenius
structure, and \ref{Frobsituations}.\ref{Frobstr} if we use modules with
Frobenius structure.
Let $f\colon(\ms{X},Z)\rightarrow(\ms{Y},W)$ be a morphism of
d-couples. We put $Z':=f^{-1}(Z)\subset Z$, which is a divisor by
the definition of morphisms of d-couples. For a coherent
($F$-)$\DdagQ{\ms{Y}}(^\dag W)$-complex $\ms{M}$, recall that
\begin{equation*}
 f^!\ms{M}:=\DdagQ{\ms{X}}(^\dag Z)\otimes_{\DdagQ{\ms{X}}(^\dag Z')}
  f_0^!(\ms{M})
\end{equation*}
in ($F$-)$\LDn{\ms{X}}(Z))$ (cf.\ paragraph \ref{d-couplesdef}). Also
recall that we denote by $\mb{D}_{\ms{X},Z}$ the dual functor with
respect to ($F$-)$\DdagQ{\ms{X}}(^\dag Z)$-modules. Let $\ms{M}$ be a
coherent ($F$-)$\DdagQ{\ms{Y}}(^\dag W)$-module (or perfect complex),
and suppose that $f^!\circ\mb{D}_{\ms{Y},W}(\ms{M})$ is a perfect
complex. Then we put
\begin{equation*}
 f^+(\ms{M}):=(\mb{D}_{\ms{X},Z}\circ f^!\circ\mb{D}_{\ms{Y},W})
  (\ms{M})
\end{equation*}
in ($F$-)$D^b_{\mr{coh}}(\DdagQ{\ms{X}}(^\dag Z))$.
When the realization of $f$ is smooth, this functor is defined for any
perfect ($F$-)$\DdagQ{\ms{Y}}(^\dag W)$-complexes by
\cite[4.3.3]{BerInt}. If Berthelot's conjecture (cf.\
\cite[5.3.6]{BerInt}) is valid, this functor is defined for any
holonomic $F$-$\DdagQ{\ms{Y}}(^\dag W)$-complexes.

\subsection{}
\label{defexpush}
Let $f\colon(\ms{X},Z)\rightarrow(\ms{Y},W)$ be a morphism of d-couples
such that the realization is proper. Let $Z':=f^{-1}(W)\subset Z$.
We denote by $f_{0,Z',+}$ the proper push-forward
from ($F$-)$\LDn{\ms{X}}(Z'))$ to ($F$-)$\LDn{\ms{Y}}(W))$.
Let $\ms{M}$ be a coherent $\DdagQ{\ms{X}}(^\dag Z)$-complex. We denote by
$j_+\ms{M}$ the underlying $\DdagQ{\ms{X}}(^\dag Z')$-complex of
$\ms{M}$. We define
\begin{equation*}
 f_+(\ms{M}):=f_{0,Z',+}(j_+\ms{M})
\end{equation*}
in $\LDn{\ms{Y}}(W))$.
Let $\ms{M}$ be a perfect ($F$-)$\DdagQ{\ms{X}}(^\dag Z)$-complex such
that $\mb{D}_{\ms{X},Z}(\ms{M})$ is a coherent ($F$-)$\DdagQ{\ms{X}}
(^\dag Z')$-complex.
In this case, we say that $\ms{M}$ is $f_!$-admissible. Then we define
\begin{equation*}
 f_!(\ms{M}):=(\mb{D}_{\ms{Y},W}\circ f_+\circ\mb{D}_{\ms{X},Z})
  (\ms{M})
\end{equation*}
in ($F$-)$D^b_{\mr{coh}}(\DdagQ{\ms{Y}}(^\dag W))$. When $Z'=Z$, any
perfect complex is $f_!$-admissible. If Berthelot's conjecture is
valid, any holonomic module is $f_!$-admissible. Another example we have
in mind is the geometric Fourier transform \cite{NH} (see also
Definition \ref{Fouriertransdfn}) or that with compact support (cf.\
Definition \ref{defcompsuppfour}).

\subsection{}
Let $(\ms{X},Z)$ be a d-couple. For coherent $\DdagQ{\ms{X}}(^\dag
Z)$-modules $\ms{M}$ and $\ms{N}$, we denote
$\ms{M}\otimesddag_{\mc{O}_{\ms{X},\mb{Q}}(^\dag Z)}\ms{N}$ simply by
$\ms{M}\otimesddag\ms{N}$. This is an object in $\LDn{\ms{X}}(Z))$. Now,
let $f$ be a morphism of d-couples. We have defined functors
$\otimesddag$, $\mb{D}_{\ms{X},Z}$, and $f_+$, $f_!$, $f^+$,
$f^!$. These functors are expected to fit in the framework of six
functors if we consider the category of holonomic complexes. Let us
explain this shortly. Consider the category d-couples such that the
morphisms consist of strict morphisms of d-couples
(cf.\ paragraph \ref{d-couplesdef}) whose realizations
are proper. For a d-couple $(\ms{X},Z)$, we consider the category of
holonomic $F$-$\DdagQ{\ms{X}}(^\dag Z)$-complexes denoted by
$\mc{C}_{(\ms{X},Z)}$. If the Berthelot conjecture holds, the category
$\mc{C}_{(\ms{X},Z)}$ is stable under six operations. Philosophically,
considering $\mc{C}_{(\ms{X},Z)}$ means to consider a ``good category''
of coefficients on $\ms{X}\setminus Z$. See \cite[5.3.6]{BerInt} for
some explanations. The following theorems are stating that fundamental
relations of these functors hold in this framework.

\subsection{}
We use the notation of paragraph \ref{defexpush}. We get the following
theorem.

\begin{thm*}
 Assume that $\ms{M}$ is $f_!$-admissible. Then there exists a
 canonical homomorphism
 \begin{equation}
  \label{compsuppnonsupp}
  f_!(\ms{M})\rightarrow f_+(\ms{M})
 \end{equation}
 compatible with Frobenius pull-backs. Moreover, when $Z'=Z$, this
 homomorphism is an isomorphism.
\end{thm*}
\begin{proof}
 By using Theorem \ref{virdualcom}, we get $\mb{D}_{\ms{Y},W}\circ
 f_{Z',+}\cong f_{Z',+}\circ\mb{D}_{\ms{X},Z'}$. The extension of scalar
 $\DdagQ{\ms{X}}(^\dag Z)\otimes_{\DdagQ{\ms{X}}(^\dag Z')}$ induces the
 functor $\mb{D}_{\ms{X},Z'}\circ j_+\rightarrow
 j_+\circ\mb{D}_{\ms{X},Z}$. Using the isomorphism
 $\mb{D}_{\ms{X},Z}\circ\mb{D}_{\ms{X},Z}\cong\mr{id}$ compatible with
 the Frobenius isomorphisms by \cite[II, 3.5]{Vir}, and combining these
 morphisms of functors we get the homomorphism
 (\ref{compsuppnonsupp}). The latter assertion is now clear.
\end{proof}

\subsection{}
Now, we will show the Poincar\'{e} duality theorem for arithmetic
$\ms{D}$-module theory in the style of SGA4 Exp.\ XVIII Th\'{e}or\`{e}me
3.2.5.

Let $f\colon(\ms{X},Z)\rightarrow(\ms{Y},W)$ be a strict morphism of
d-couples such that the realization is proper. Moreover, we assume that
$f^{-1}_0(W)=Z$.
There exists the following isomorphism thanks to \cite[IV, 7.4]{Vir2}.
\begin{equation}
 \label{Viradjoint}
  R\mr{Hom}_{\DdagQ{\ms{X}}(^\dag Z)}(\ms{M},f^!\ms{N})\xrightarrow{\sim}
 R\mr{Hom}_{\DdagQ{\ms{Y}}(^\dag W)}(f_+\ms{M},\ms{N})
\end{equation}
We get that this isomorphism is compatible with Frobenius by Corollary
\ref{adointcomp}.

\begin{rem*}
 The isomorphism (\ref{Viradjoint}) should hold without assuming that
 $f_0^{-1}(W)=Z$ if we replace $f_+$ by $f_!$. For this, we need to
 assume the Berthelot conjecture. In the following, we freely use this
 conjecture. Let us sketch a proof. It suffices to show the case where
 $f\colon(\ms{X},Z)\rightarrow(\ms{X},Z')$ such that the realization is
 the identity and $Z'\subset Z$. We can see easily that it suffices to
 show that for $\ms{M}\in\mc{C}_{(\ms{X},Z')}$ and
 $\ms{N}\in\mc{C}_{(\ms{X},Z)}$, the homomorphism induced by scalar
 extension
 \begin{equation*}
  R\mr{Hom}_{\DdagQ{\ms{X}}(^\dag Z')}(\ms{M},f_+\ms{N})\rightarrow
   R\mr{Hom}_{\DdagQ{\ms{X}}(^\dag Z)}(f^!\ms{M},\ms{N})
 \end{equation*}
 is an isomorphism. We can reduce the verification to the following two
 cases; when the support of $\ms{M}$ is contained in $Z$, and when
 $\ms{M}$ is a $\DdagQ{\ms{X}}(^\dag Z)$-module. To see the former case,
 use the theorem of Berthelot-Kashiwara \cite[5.3.3]{BerInt}. To see the
 latter case, it suffices to show that the homomorphism
 \begin{equation*}
  R\mr{Hom}_{\DdagQ{\ms{X}}(^\dag Z')}(\mc{O}_{\ms{X},\mb{Q}}(^\dag Z'),
   \mb{D}_{\ms{X},Z'}(\ms{M})\otimesddag f_+\ms{N})\rightarrow
   R\mr{Hom}_{\DdagQ{\ms{X}}(^\dag
   Z)}(\mc{O}_{\ms{X},\mb{Q}}(^\dag Z),\mb{D}_{\ms{X},Z}(\ms{M})
   \otimesddag\ms{N})
 \end{equation*}
 is an isomorphism. Using the Spencer resolution, it suffices to show
 that the canonical homomorphism $\mb{D}_{\ms{X},Z'}(\ms{M})\otimesddag
 f_+\ms{N}\rightarrow\mb{D}_{\ms{X},Z}(\ms{M})\otimesddag\ms{N}$ is
 an isomorphism. The verification is easy.
\end{rem*}

To complete the Poincar\'{e} duality we need to calculate $f^!$
in the case where $f$ is smooth. Namely, we get the following.

\begin{thm*}
 Let $f\colon(\ms{X},Z)\rightarrow(\ms{Y},W)$ be a morphism of d-couples
 such that the realization is {\em smooth}. Then there is a canonical
 isomorphism of cohomological functors with Frobenius isomorphisms
 \begin{equation*}
  f^!\cong f^+(d)[2d]\colon D^b_{\mr{coh}}(\DdagQ{\ms{Y}}(^\dag W))
   \rightarrow D^b_{\mr{coh}}(\DdagQ{\ms{X}}(^\dag Z))
 \end{equation*}
 where $d$ denotes the relative dimension of $f$.
\end{thm*}
\begin{proof}
 Put the functor $\mb{D}_{\ms{X},Z}$ to the both sides of the
 isomorphism of Theorem \ref{smoothpoincare}.  Using the involutivity
 \cite[II.3.5]{Vir} of $\mb{D}_{\ms{X},Z}$, we get the claim.
\end{proof}

\subsection{}
By using the comparison between dual functor of arithmetic
$\ms{D}$-modules and that of isocrystals, we can prove a purity type
result. Namely we have

\begin{thm*}[Purity]
 Let $(\ms{X},Z)\rightarrow(\ms{Y},W)$ be a morphism of
 d-couple. Moreover, suppose that the
 realization of $f$ is a {\em closed immersion}. Let $\ms{M}$ be a
 convergent $F$-isocrystal on $\ms{Y}$ overconvergent along $W$. Then
 $f^+(\ms{M})$ is defined, and we get
 \begin{equation*}
  f^!(\ms{M})\cong f^+(\ms{M})(-d)[-2d]
 \end{equation*}
 where $d$ denotes the codimension of $\ms{X}$ in $\ms{Y}$.
\end{thm*}
\begin{proof}
 We know that $f^*(\mr{sp}^*(\ms{M})^\vee)\cong(f^*\mr{sp}^*(\ms{M}))
 ^\vee$. Together with the comparison theorem of duality functors
 Corollary \ref{compduals} and the compatibility of pull-backs
 \cite[4.1.8]{Caro}, the theorem follows.
\end{proof}

\subsection{}
\label{properbc}
Consider the following cartesian diagram of d-couples.
\begin{equation*}
 \xymatrix{
  (\ms{X}',Z')\ar[r]^{g'}\ar[d]_{f'}\ar@{}[rd]|\square&
  (\ms{X},Z)\ar[d]^f\\
 (\ms{Y}',W')\ar[r]_g&(\ms{Y},W)}
\end{equation*}
Here, cartesian means it is cartesian as a diagram of underlying formal
schemes and $\ms{X}'\setminus Z'\cong(\ms{X}\setminus
Z)\times_{(\ms{Y}\setminus W)}(\ms{Y}'\setminus W')$. Now, we get the
following base change theorem.

\begin{thm*}[Proper base change]
 We preserve the notation. We get a canonical equivalence of
 functors
 \begin{equation*}
  g^!\circ f_+\cong f'_+\circ g'^!\colon \LDn{\ms{X}}(Z))
  \rightarrow\LDn{\ms{Y}'}(W')).
 \end{equation*}
 This isomorphism is compatible with Frobenius pull-backs.
\end{thm*}
\begin{proof}
 Using the standard factorization, it suffices to show the theorem in the
 cases where $g$ is a closed immersion and a smooth morphism.
 When $g$ is smooth, we get the theorem by Lemma \ref{bcthmsch}. When
 $g$ is a closed immersion, this is a result of Caro
 \cite[2.2.18]{Caroc}.
\end{proof}

\begin{rem*}
 Let $\ms{M}$ be an object in $D^b_{\mr{coh}}(\DdagQ{\ms{X}}(^\dag
 Z))$. When $g^+\circ f_!(\ms{M})$ and $f'_!\circ g'^+(\ms{M})$ are
 defined, the above equivalence and the isomorphism
 $\mb{D}\circ\mb{D}\cong\mr{id}$ induces an isomorphism
 \begin{equation*}
  g^+\circ f_!(\ms{M})\cong f'_!\circ g'^+(\ms{M}).
 \end{equation*}
\end{rem*}

\subsection{}
We preserve the notation. Let $\ms{M}$ and $\ms{N}$ be perfect
($F$-)$\DdagQ{\ms{Y}}(^\dag W)$-complexes. We assume that
$\mb{D}_{\ms{Y},W}(\ms{M})\otimesddag\mb{D}_{\ms{Y},W}(\ms{N})$ is
also perfect ($F$-)$\DdagQ{\ms{Y}}(^\dag W)$-complex. Then, we define
the twisted tensor product of $\ms{M}$ and $\ms{N}$ denoted by
$\ms{M}\otimesddagt\ms{N}$ to be
\begin{equation*}
 \mb{D}_{\ms{Y},W}(\mb{D}_{\ms{Y},W}(\ms{M})\otimesddag\mb{D}
  _{\ms{Y},W}(\ms{N})).
\end{equation*}
One of the reasons we introduce this twisted tensor product is the
following. For $\ms{M}$ and $\ms{N}$ in $\LD{\ms{Y}}$, we get that
\begin{equation}
 \label{tensextpul}
 f^!(\ms{M}\otimesddag\ms{N})[d_f]\cong f^!\ms{M}\otimesddag f^!\ms{N}
\end{equation}
where $d_f:=\dim(\ms{X})-\dim(\ms{Y})$. This is compatible with the
Frobenius structures. However, if we replace $f^!$ by $f^+$, the
equality (\ref{tensextpul}) does not hold in general. Nevertheless, if
we also replace $\otimesddag$ by $\otimesddagt$, the equality
holds in turn. Namely, (\ref{tensextpul}) induces an isomorphism
\begin{equation*}
 f^+(\ms{M}\otimesddagt\ms{N})[-d_f]\cong f^+\ms{M}\otimesddagt
  f^+\ms{N}
\end{equation*}
if the both sides are defined.
A consequence of Lemma \ref{tenstwist} is the following.

\begin{prop*}
 Let $(\ms{X},Z)$ be a d-couple, and $\ms{M}$ and $\ms{N}$ be coherent
 $F$-$\DdagQ{\ms{X}}(^\dag Z)$-complexes. Assume further that $\ms{M}$
 is an overconvergent
 isocrystal along the divisor $Z$. Then we get
 \begin{equation*}
 \ms{M}\otimesddagt\ms{N}\cong(\ms{M}\otimesddag\ms{N})(d)
 \end{equation*}
 where $d$ denotes the dimension of $\ms{X}$.
\end{prop*}

\subsection{}
\label{comrigdpushconcl}
Finally, let us compare the rigid cohomology with the push-forward as
arithmetic $\ms{D}$-modules.

Let $\ms{X}$ be a {\em proper smooth} formal scheme of dimension $d$,
$Z$ be a divisor of the special fiber of $\ms{X}$, $\ms{U}$ be the
complement, and $U_0$ be its special fiber. We denote by
$f\colon(\ms{X},Z)\rightarrow(\mr{Spf}(R),\emptyset)$ the morphism of
d-couples induced by the structural morphism of $\ms{X}$. Let $\ms{M}$
be a coherent $\DdagQ{\ms{X}}(^\dag Z)$-module which is an
overconvergent isocrystal along $Z$. Suppose that it is coherent as
a $\DdagQ{\ms{X}}$-module.

By Corollary \ref{comprigidaD}, we get the canonical isomorphism
\begin{equation*}
 H^if_+\ms{M}\cong H^{d+i}_{\mr{rig}}(U_0,\mr{sp}^*(\ms{M}))(d).
\end{equation*}

To see the relation for cohomologies with compact supports, we use the
Poincar\'{e} duality of rigid cohomology. In the curve case,
Poincar\'{e} duality is proven in \cite{Cr}. In the general case,
we could not find any literature explicitly stating the Poincar\'{e}
duality with Frobenius structure. However in \cite[8.3.14]{LS}, the
coupling is defined, and in \cite{Ked}, the perfectness of the couple is
proven. Thus we get the following isomorphism
\begin{equation*}
 H^i_{\mr{rig}}(U_0,M)^\vee\cong H^{2d-i}_{\mr{rig},c}(U_0,M^\vee(d))
\end{equation*}
for an overconvergent $F$-isocrystal $M$ on the smooth variety $U_0$
over $k$. Using this, we get
\begin{equation*}
 H^if_{!}\ms{M}\cong(H^{-i}f_+\mb{D}_{\ms{X},Z}(\ms{M}))^{\vee}
  \cong\bigl(H^{d-i}_{\mr{rig}}(U_0,\mr{sp}^*(\ms{M})^{\vee}(-d))
  (d)\bigr)^{\vee}\cong
  H^{d+i}_{\mr{rig},c}(U_0,\mr{sp}^*(\ms{M}))(d).
\end{equation*}
Here the second isomorphism follows from Corollary \ref{compduals}.
Summing up, we get
\begin{equation*}
 H^if_+\ms{M}\cong H^{d+i}_{\mr{rig}}(U_0,\mr{sp}^*(\ms{M}))(d),
  \qquad
  H^if_!\ms{M}\cong H^{d+i}_{\mr{rig},c}(U_0,\mr{sp}^*(\ms{M}))(d).
\end{equation*}
In particular, we note that there exist canonical isomorphisms
\begin{equation*}
 H^if_+f^+(K)\cong H^i_{\mr{rig}}(U_0/K),\qquad
  H^if_!f^+(K)\cong H^i_{\mr{rig},c}(U_0/K)
\end{equation*}
compatible with Frobenius isomorphism.

\appendix
\section{Comparison theorem between Fourier transform and Fourier transform with compact
support. C.~Huyghe}
\usubsection*{Introduction}
Let $V$ be a discrete valuation ring of inequal characteristics $(0,p)$, containing an
element $\pi$, the $\pi$ of Dwork, satisfying the equation $\pi^{p-1}=-p$, $\SS=\spf(V)$ the
formal spectrum of $V$, $\XX$ the formal affine line over $\SS$, $\XX^{\vee}$ the dual
affine line.
 Let us introduce $\YY$ and $\YY^{\vee}$ two copies of the formal projective line over $R$
compactifying $\XX$ and $\XX^{\vee}$ and denote $\infty_{\YY}$ (resp. $\infty_{\YY^{\vee}}$)
the complementary divisors. 
In \cite{Ber1} Berthelot constructed sheaves of arithmetic differential operators 
 with overconvergent singularities along a divisor ({\it e.g.}\
 $\DD^{\dagger}_{\YY}(\infty_{\YY})$ in
our situation). For several reasons these sheaves have to be thought 
as sheaves over the open subset which is the complementary of the considered divisor,
thus, in our case, 
over the formal affine line. 
In \cite{NH0} one constructed the Fourier transform 
of $\DD^{\dagger}_{\YY}(\infty_{\YY})$-modules, using
the Dwork exponential module as kernel and 
we checked the compatibility with the so-called na\"\i ve Fourier transform. The aim of
this note is to define the compact support Fourier transform in dimension $1$ and to prove 
that it coincides with the Fourier transform without compact support
introduced in {\it loc.\ cit.}.

This result is part from the unpublished part 4.4 of \cite{NH0}, where it is
proven in dimension $N$. The dimension $1$ case is technically easier and sufficient for 
\cite{AM} while giving a good idea of the proof in dimension $N$. That's why we
restrict to this case here. The extra work in dimension $N$ consists
into proving a generalization of the division lemma \ref{division_lemma}
and to deal with longer complexes of length $N+1$.

\usubsection{Preliminaries}
\subsubsection{}\label{preliminaries}
Denote $K$ the fraction field of $V$.  
For $l\in\Ze$, $|l|$ is the usual archimedean absolute value of $l$. The
$p$-adic valuation
of an element $a$ of a $p$-adically separated ring 
is $v_p(a)$. For $a\in K$, $$|a|_p=p^{-v_p(a)}.$$

If $\FF$ is a sheaf of abelian 
groups over a topological space we denote $\FF_{\Qr}=\Qr\ot_{\Ze}\FF.$

The product $\ZZ=\YY\times \YY^{\vee}=\widehat{\mathbb{P}}^1_\SS\times
\widehat{\mathbb{P}}^1_\SS$
is endowed with the ample divisor $\infty=\infty_{\YY}
\times \widehat{\mathbb{P}}^1_\SS \bigcup \widehat{\mathbb{P}}^1_\SS\times \infty_{\YY^{\vee}}$. 

When needed, we will use $[u_0,u_1]$ and $[v_0,v_1]$ as homogeneous coordinates 
over the two copies of $\widehat{\mathbb{P}}^1_\SS$, $u_0=0$ and $v_0=0$ are equations 
of the infinite divisor over each copy of $\widehat{\mathbb{P}}^1_\SS$,
$x=u_1/u_0$ and $y=v_1/v_0$ will be coordinates 
over the affine plane complementary of the $\inft$ divisor over $\ZZ$. These two 
coordinates $x$ and $y$ over $\XX=\widehat{\mathbb{A}}^1_\SS$ and
$\XX^{\vee}=\widehat{\mathbb{A}}^1_\SS$
should be considered as dual to each other.

\subsubsection{}\label{def_dag}
Let $\VV$ be a smooth formal scheme over $\SS$, endowed with a relative divisor $D$
(meaning that $D$ induces a divisor of the special fiber), and $\UU=\VV \backslash D$. 
Then the direct image by specialization of the constant overconvergent $F$-isocrystal over $U$, the
special fiber of $\UU$, is a
sheaf over $\VV$ denoted by $\OO_{\VV}(^{\dagger }D)$ (4.4 of \cite{Ber1}). If $\VV$ is
affine and if 
$f$ is an equation of $D$
over $\VV$, 
we have the following description
$$\Ga(\VV,\OO_{\VV}(^{\dagger }D))=\left\{\sum_{l\in\Ne} \frac{a_l}{f^l},\,a_l\in \Ga(\VV,\OO_{\VV,\Qr})\, |
\textrm{ and }\exists 
C'>0\, | \, v_p(a_{l})-\frac{l}{C'}\rig +\inft \textrm{ if } l\rig +\inft \right\}.$$

Consider the ring $\DD^{\dagger}_{\VV}$ of arithmetic differential 
operators over $\VV$. Suppose that $x_1,\ldots,x_n$ are coordinates on $\VV$, 
denote $\der_{x_i}$ the corresponding derivations for $i\in\{1,\ldots,n\}$, and
$$\der_{x_i}^{[k_i]}=\frac{\der_{x_i}^{k_i}}{k_i!},
\textrm{ and }\uder^{[\uk]}=\der_{x_1}^{[k_1]}\cdots \der_{x_n}^{[k_n]},$$
then we have the following description
\begin{multline*}\Ga(\VV,\DD^{\dagger}_{\VV})=\left\{\sum_{\uk\in\Ne^n}a_{\uk}\uder^{[\uk]},\,\,a_{\uk}\in 
\Ga(\VV,\OO_{\VV,\Qr})\,
\textrm{ and }\exists 
C>0\, | \, v_p(a_{\uk})-\frac{|\uk|}{C}\rig +\inft \textrm{ if } |\uk|\rig +\inft
\right\}.
\end{multline*}

In 4.2 of \cite{Ber1}, Berthelot introduces also rings of differential operators 
with overconvergent singularities, which are sheaves of $\OO_{\VV}(^{\dagger }D)$-modules. 
Suppose that $x_1,\ldots,x_n$ are coordinates on $\VV$
and that the divisor $D$ is defined by the equation $f=0$ on $\VV$, then we have 
the following description
\begin{multline*}
\Ga(\VV,\DD^{\dagger}_{\VV}(D))=\left\{\sum_{l\in\Ne,\uk\in\Ne^n}\frac{a_{l,\uk}}{f^l}\uder^{[\uk]},\,
a_{l,\uk}\in \Ga(\VV,\OO_{\VV,\Qr})\,\right.\\\left.\phantom{\sum_{\ul,\uk}a_{\uk}\uder^{[\uk]}} \textrm{ and }\exists 
C>0\, | \, v_p(a_{l,\uk})-\frac{l+|\uk|}{C}\rig +\inft \textrm{ if } l+|\uk| \rig +\inft
\right\}.
\end{multline*}

All these sheaves are weakly complete, in the sense that they are inductive limit of 
sheaves of $p$-adic complete rings as described in the following subsection. Let us 
stress here upon the fact that these sheaves are always sheaves of $K$ vector spaces 
(even if there is no $\Qr$ in their notation). We made this convention to avoid 
too heavy notations. Note also that this won't be the case with the sheaves that  
we will introduce in the following subsections.
\subsubsection{} 
In this subsection, we use 
notations of \ref{def_dag}. Let us fix $m\in\Ne$, elements $x_1,\ldots, x_n\in\Ga(\VV,\OO_{\VV})$  
which are coordinates on $\VV$ and $f\in\Ga(\VV,\OO_{\VV})$ such that $D\bigcap \VV=V(f)$. 

Let us introduce here some coefficients. First 
we define the application $\nu_m\colon\mathbf{Z}\rig \mathbf{N}$ by the following way. 
If $k<0$, we set $\nu_m(k)=0$. Let $k\in \mathbf{N}$,
and $q$ and $r$ be the quotient and the remainder of the 
division of $k$ by $p^{m+1}$. If $r=0$, we set $\nu_m(k)=q$, otherwise we set
$\nu_m(k)=q+1$. We extend this application to $\mathbf{Z}^r$, 
by $\num((k_1,\ldots,k_r))=\num(k_1)+\cdots +\nu_m(k_r)$. 

We denote also $q_k^{(m)}$ the quotient of 
the division of a positive integer $k$ by $p^m$, and 
$$\uder^{\la \uk \ra_{(m)}}=q_{k_1}^{(m)}!\cdots q_{k_n}^{(m)}!\,\der_{x_1}^{[k_1]}\cdots\der_{x_n}^{[k_n]}.$$
If the choice of $m$ is clear, we will omit it in the notation.

Berthelot defines a sheaf 
$\what{\BB}^{(m)}_{\VV}$ by setting locally
\begin{multline*}
\Ga(\VV,\what{\BB}^{(m)}_{\VV})=\left\{\sum_l \frac{a_l}{f^l},\,a_l\in \Ga(\VV,\OO_{\VV})\,
|\, v_p(a_{l})\geq \nu_m(l) \textrm{ and } 
v_p(a_{l})- \nu_m(l)\rig +\infty \textrm{ if } l\rig +\inft \right\}.
\end{multline*}

Then 
there are canonical injective morphisms $\what{\BB}^{(m)}_{\VV}\subset \what{\BB}^{(m+1)}_{\VV}$
and
$$\OO_{\VV}(^{\dagger }D)=\varinjlim_m \what{\BB}^{(m)}_{\VV,\Qr}.$$

 Berthelot 
defines also sheaves of rings differential operators $\DD^{(m)}_{\VV}(D)$ and their
$p$-adic completion $\what{\DD}^{(m)}_{\VV}(D)$ over 
$\VV$ by 
\begin{align*}\Ga(\VV,\DD^{(m)}_{\VV}(D))=
\left\{\sum_{l\in\Ne,\uk\in\Ne^n}\frac{a_{l,\uk}}{f^l}\uder^{\la \uk \ra_{(m)}},\, 
a_{l,\uk}\in \Ga(\VV,\OO_{\VV}) \, |\, v_p(a_{l,\uk})\geq  \nu_m(l)
\right\},
\end{align*}
where the sums are finite
and
\begin{multline*}\Ga(\VV,\what{\DD}^{(m)}_{\VV}(D))=
\left\{\sum_{l\in\Ne,\uk\in\Ne^n}\frac{a_{l,\uk}}{f^l}\uder^{\la \uk \ra_{(m)}},\, 
a_{l,\uk}\in \Ga(\VV,\OO_{\VV})\, |\, v_p(a_{l,\uk})\geq  \nu_m(l)\right. \\
\left. \phantom{\sum_{\ul,\uk}a_{\uk}\uder^{[\uk]}}
\textrm{ and } v_p(a_{l,\uk})- \nu_m(l)\rig +\infty \textrm{ if } |\uk|+l\rig +\inft
\right\}.
\end{multline*}

Then there are canonical injective morphisms $\what{\DD}^{(m)}_{\VV,\Qr}(D) \subset \what{\DD}^{(m+1)}_{\VV,\Qr}(D)$
and
$$\DD^{\dagger}_{\VV}(D)=\varinjlim_m \what{\DD}^{(m)}_{\VV,\Qr}(D).$$

If $D=\emptyset$ the previous sheaves are simply denoted $\DD^{(m)}_{\VV}$ and
$\what{\DD}^{(m)}_{\VV}$.

\subsubsection{} We finally recall the following inequalities for $|\uk|$ and $|\ul|$
elements of $\mathbf{N}^n$, $l$ and $r$ in $\mathbf{N}$.
\label{estimations1}
\begin{gather*}
\frac{|\uk|}{p-1}-n \log_p(|\uk|+1)-n\leq v_p(\uk!) \leq
\frac{|\uk|}{p-1} \\
\frac{|\uk|}{p^m(p-1)}-n \log_p(|\uk|+1)-n\frac{p}{p-1}\leq v_p(q_{\uk}^{(m)}!) \leq
\frac{|\uk|}{p^m(p-1)}\\
\frac{|\ul|}{p^{m+1}}\leq \nu_m(\ul)\leq \frac{|\ul|}{p^{m+1}}+n \\
             0\leq \nu_m(\ul)-\nu_m(|\ul|)\leq n\\
0\leq v_p\left(\binom{l}{r}\right)\leq p(\log_p(l)+1).
\end{gather*}

\subsubsection{}
By definition an induced $\DD^{\dagger}_{\VV}(D)$-module is a
$\DD^{\dagger}_{\VV}(D)$ of the type 
$$\DD^{\dagger}_{\VV}(D)\ot_{\OO_{\VV}(^{\dagger }D)}\EE,$$
where $\EE$ is a coherent $\OO_{\VV}(^{\dagger }D)$-module and where the 
$\DD^{\dagger}_{\VV}(D)$-module structure comes from the one of
$\DD^{\dagger}_{\VV}(D)$.

\subsubsection{}\label{eq_cat_Pn}
Coming back to the situation of the introduction,
let $p_1$ and $p_2$ be the two projections $\ZZ \rig \YY$ and $\ZZ \rig \YY^{\vee}$, 
 and $\infty'=p_2^{-1}\infty_{\YY^{\vee}}$.

Let us now recall how to describe the structure of the category of left
$\DD^{\dagger}_{\YY}(\infty_{\YY})$-modules, resp.
$\DD^{\dagger}_{\ZZ}(\infty)$-coherent modules, resp.
$\DD^{\dagger}_{\ZZ}(\infty')$-coherent modules.

Let us first start with coherent
$\DD^{\dagger}_{\YY}(\infty_{\YY})$-modules (see \cite{HuCRAS2}). First note that 
$$ \Ga(\YY,\OO_{\YY}(^{\dagger}\infty_{\YY}))=\left\{\sum_l b_l x^l, b_l\in K,
 \textrm{ and }\exists C>0, \eta<1 \,|\, |b_{l}|_p<C\eta^{l}\right\}$$
and set  
$$A_1(K)^{\dagger}=\left\{\sum_{k\in\Ne,l\in\Ne}a_{l,k}x^{l}\der_x^{[k]}, a_{l,k}\in K,
\textrm{ and }\exists C>0, \eta<1 \,|\, |a_{\uk,\ul}|_p<C\eta^{l+k}\right\},$$
the weak completion of the Weyl algebra. It is a coherent algebra and we have (\cite{HuCRAS2})

\begin{thmH}\label{equivP0} The functor $\Ga(\YY,. )$ (resp. $R\Ga(\YY,.)$) establishes an equivalence of
categories between the category of left coherent $\DD^{\dagger}_{\YY}(\infty_{\YY})$-modules 
(resp. $D^b_{coh}(\DD^{\dagger}_{\YY}(\infty_{\YY}))$) and the 
category of left coherent $A_1(K)^{\dagger}$-modules (resp. $D^b_{coh}(A_1(K)^{\dagger})$.
\end{thmH}

In particular, $A_1(K)^{\dagger}\simeq \Ga(\YY,\DD^{\dagger}_{\YY}(\infty_{\YY}))$ 
and every coherent $\DD^{\dagger}_{\YY}(\infty_{\YY})$ admits globally over $\YY$
a resolution by globally projective and finitely generated $\DD^{\dagger}_{\YY}(\infty_{\YY})$-modules.
This resolution can be taken finite since $A_1(K)^{\dagger}$ has finite cohomological
dimension (\cite{NH2}).

Consider now the situation over $\ZZ$ for $m$ fixed. In theorem 3 of \cite{HuCRAS2}, we 
proved that the elements $\der_x^{{\la k_1 \ra}_{(m)}}$ are global sections 
over $\widehat{\mathbb{P}}^1_{\SS}$ of the sheaves $\DD^{(m)}_{\widehat{\mathbb{P}}^1_{\SS}}$, so that
  $$\der_x^{{\la k_1 \ra}_{(m)}}\der_y^{{\la k_2
\ra}_{(m)}}\in\Ga(\ZZ,\DD^{(m)}_{\ZZ}).$$
 Moreover, an easy 
computation (2.1 of \cite{HuCRAS2}) shows that elements 
$p^{\nu_m(l_1)}x^{l_1}$ for $l_1\geq 0$ are global sections of $\Bcm_{\widehat{\mathbb{P}}^1_{\SS}}$, 
implying that 
$$ p^{\nu_m(\ul)}x^{l_1}y^{l_2}\in \Ga(\ZZ,\Bcm_{\ZZ}).$$
Define
\begin{multline*}
\what{E}^{(m)}=\left\{\sum_{\uk\in\Ne^2,\ul\in\Ne^2}a_{\ul,\uk}x^{l_1}y^{l_2}\der_x^{\la k_1 \ra}
\der_y^{\la k_2 \ra},\,a_{\ul,\uk}\in V,\,\textrm{ and } v_p(a_{\ul,\uk})\geq \nu_m(\ul)\right.\\
\left. \phantom{\sum_{\ul,\uk}a_{\uk}\uder^{[\uk]}}
 \,|\,v_p(a_{\ul,\uk})-\nu_m(\ul) \rig +\infty
  \textrm{ if } 
|\ul|+|\uk| \rig +\infty \right\}.
\end{multline*}
 From these observations and the fact that 
$\Dcm_{\ZZ}$ is a sheaf of $p$-adically complete algebras, we see that 
$$\what{E}^{(m)}\subset \Ga(\ZZ,\Dcm_{\ZZ}(\infty)).$$

Consider the weak completion of the Weyl algebra in $2$ variables
$$A_2(K)^{\dagger}=\left\{\sum_{\uk\in\Ne^2,\ul\in\Ne^2}a_{\ul,\uk}x^{l_1}y^{l_2}\der_x^{[k_1]}
\der_y^{[k_2]}, a_{l,k}\in K,\textrm{ and }\exists
 C>0, \eta<1 \,|\, |a_{\uk,\ul}|<C\eta^{|\ul|+|\uk|}\right\},$$
which is coherent from \cite{HuCRAS2}.
It is easy to see that 
$$A_2(K)^{\dagger}=\varinjlim_m \what{E}^{(m)}_{\Qr} .$$
 We endow $A_2(K)^{\dagger}$ with the inductive limit topology coming from this filtration.

Because the divisor $\infty$ over $\ZZ$ is ample, we can apply 4.5.1 of \cite{Hucomp} 
which tells us that $$\Ga(\ZZ,\DD^{\dagger}_{Z}(\infty))\simeq A_2(K)^{\dagger}.$$

Moreover we have the following theorem
(4.5.1 of \cite{Hucomp} and 5.3.4 of \cite{HuFour}).
\begin{thmH}\label{equivP1} The functor $\Ga(\ZZ,. )$ (resp. $R\Ga(\ZZ,.)$) establishes an equivalence of
categories between the category of left coherent $\DD^{\dagger}_{\ZZ}(\infty)$-modules 
(resp. $D^b_{coh}(\DD^{\dagger}_{\ZZ}(\infty))$) and the 
category of left coherent $A_2(K)^{\dagger}$-modules (resp. $D^b_{coh}(A_2(K)^{\dagger}))$.
\end{thmH}

Consider now the scheme $\ZZ=\YY\times \YY^{\vee}$ endowed with the divisor
$\infty'=p_2^{-1}(\infty_{\YY^{\vee}})$ and $\DD^{\dagger}_{\ZZ}(\infty')$ the ring of arithmetic differential
operators with overconvergent coefficients along $\infty'$.
 In order to deal with coherent $\DD^{\dagger}_{\ZZ}(\infty')$-modules, denote
$$B_2(K)^{\dagger}=\left\{\sum_{l_2\in\Ne,\uk\in\Ne^2}a_{l_2,\uk}y^{l_2}\der_x^{[k_1]}
\der_y^{[k_2]},\,a_{l_2,\uk}\in K, \textrm{ and } \,\exists C>0, \eta<1 \,|\,
|a_{l_2,\uk}|_p<C\eta^{l_2+|\uk|}\right\},$$
which we endow with the induced topology of $A_2(K)^{\dagger}$. Consider 
$\what{F}^{(m)}=\what{E}^{(m)}\bigcap B_2(K)^{\dagger}$. As before we observe that 
$$\what{F}^{(m)}\subset \Ga(\ZZ,\what{\DD}^{(m)}_{\ZZ}(\infty')),$$ which leads to the following 
inclusion
$$\varinjlim_m  \what{F}^{(m)}_{\Qr}=B_2(K)^{\dagger}\subset
\Ga(\ZZ,\DD^{\dagger}_{\ZZ}(\infty')).$$
We apply 2.3.3 of \cite{NH} to see that this is actually an equality.

We will also use the following division lemma (4.3.4.2 of \cite{NH}).
\begin{thmH}\label{decomposition_ov_Z}
\begin{itemize} 
\item[i.] For any $P\in A_2(K)^{\dagger}$ there exists a unique $(Q,R)\in
A_2(K)^{\dagger}\times B_2(K)^{\dagger}$ such that 
$P=Q(-\der_y+\pi x)+R.$
\item[ii.] The maps $P \mapsto Q$ and $P \mapsto R$ are continuous. More precisely, if $P\in
\what{E}^{(m)}_{\Qr}$, then $Q\in \what{E}^{(m+2)}_{\Qr}$, and $R\in
\what{E}^{(m+2)}_{\Qr}\bigcap B_2(K)^{\dagger}$.
\end{itemize}
\end{thmH}
\usubsection{Fourier transforms}
\nnsubsubsection{Kernel of the Fourier transform.}
In dimension $1$, the duality
pairing $\XX \times \XX^{\vee} \rig \what{\mathbb{A}}^1_\SS $ is given by $t \mapsto xy$
where $t$ is the global coordinate on the right-hand side. It extends 
to $\delta\colon\ZZ=\YY\times \YY^{\vee} \rig \what{\mathbb{P}}^1_\SS$, by the formula $t^{-1}\mapsto
x^{-1}y^{-1}$ on neighborhoods of $\infty$. Let $L_{\pi}$ be the realization over $\what{\mathbb{P}}^1_\SS$
of the overconvergent Dwork $F$-isocrystal. It is given by a connection
$\nabla(1)=-\pi dt$. We define 
$K_{\pi}$ to be the $\DD^{\dagger}_{\ZZ,\Qr}(\infty)$-module associated to the
overconvergent $F$-isocrystal $\delta^*(L_{\pi})$. The module $K_{\pi}$ is thus isomorphic 
to $\OO_{\ZZ}(^{\dagger}\infty)$ 
with a connection on $\YY \times \YY^{\vee}$ defined by $$\nabla(1)=-\pi(xdy+ydx).$$

\nnsubsubsection{Explicit descriptions of cohomological operations.} 
\label{acyclicity}
 Let us consider the
following diagram
$$\xymatrix@C=25pt@R=15pt{ & \ZZ=\YY\times
 \YY^{\vee}\ar@{->}[dr]^<>(.5){p_2}
 \ar@{->}[dl]_<>(.5){p_1} &  \\
          \YY & & \YY^{\vee}.}$$
For $M,N$ two $\OO_{\ZZ}(^{\dagger}\infty)$-modules, denote 
$$M\widetilde{\ot} N=M \ot^{\bf{L}}_{\OO_{\ZZ}(^{\dagger}\infty)}N [-2].$$

Note that the sheaves
$p_1^{-1}\DD^{\dagger}_{\YY}(\infty_{\YY})$ (resp.
$p_2^{-1}\DD^{\dagger}_{\YY^{\vee}}(\infty_{\YY^{\vee}})$),
 are canonically subsheaves of rings of $\DD^{\dagger}_{\ZZ}(\infty)$.
Cohomological operations involve sheaves $\DD^{\dagger}_{\ZZ\rig \YY}(\infty)$, 
respectively
$\DD^{\dagger}_{\YY\lrig \ZZ}(\infty)$, which are left (resp. right) coherent 
$\DD^{\dagger}_{\ZZ}(\infty)$-modules and right (resp. left) 
$p_1^{-1}\DD^{\dagger}_{\YY}(\infty_{\YY})$-modules, 
 which can be explicitely described in our
case. 
The module structures over $\DD^{\dagger}_{\YY\lrig \ZZ}(\infty)$ are obtained 
from these of $\DD^{\dagger}_{\ZZ\rig \YY}(\infty)$ by twisting by the adjoint operator
(1.3 of \cite{Ber2}). In particular underlying abelian groups of both sheaves are the
same.
Because the sheaves $$\OO_{\ZZ}(^{\dagger}\infty)\ot_{\OO_{\ZZ}}\omega_{\ZZ} \textrm{ and }
\OO_{\ZZ}(^{\dagger}\infty)\ot_{p_1^{-1}\OO_{\YY}}p_1^{-1}\omega_{\YY} $$
are free, the twisted actions are easy to describe globally. For example, for 
$P\in \DD^{\dagger}_{\YY\lrig \ZZ}(\infty)$, the right action of $\der_x $ over $P$
is equal to the left action by $-\der_x$ over $P$ seen as an element of
 $\DD^{\dagger}_{\ZZ\rig \YY}(\infty)$.

From 4.2.1 of \cite{NH}, we know that these sheaves admit a free 
resolution, as $\DD^{\dagger}_{\ZZ}(\infty)$-modules
$$\xymatrix@R=5pt{ 0 \ar@{->}[r]& \DD^{\dagger}_{\ZZ}(\infty)\ar@{->}[r]&
\DD^{\dagger}_{\ZZ}(\infty)\ar@{->}[r] & 0\\
  & P\ar@{|->}[r] &  P \der_y.  & }$$
For the sheaf $\DD^{\dagger}_{\YY\lrig \ZZ}(\infty)$, consider the map $P \mapsto \der_y P$.
Actually, if we endow $\DD^{\dagger}_{\ZZ}(\infty)$ with the canonical structure of 
right (resp. left)-$p^{-1}\DD^{\dagger}_{\YY}(\infty_{\YY})$-module this 
complex is a resolution of $\DD^{\dagger}_{\ZZ\rig \YY}(\infty)$ (resp. 
$\DD^{\dagger}_{\YY\lrig \ZZ}(\infty)$) which is $\DD^{\dagger}_{\ZZ}(\infty)\times
\DD^{\dagger}_{\YY}(\infty_{\YY})$ (resp. $\DD^{\dagger}_{\YY}(\infty_{\YY}))\times
\DD^{\dagger}_{\ZZ}(\infty)$-linear).

Moreover, thanks to \ref{equivP1} the sheaves $\DD^{\dagger}_{\ZZ\rig \YY}(\infty)$
and $\DD^{\dagger}_{\YY\lrig \ZZ}(\infty)$ are acyclic for the global section functor.

Moreover, both sheaves $\DD^{\dagger}_{\YY\lrig \ZZ}(\infty)$ and 
$\DD^{\dagger}_{\ZZ\rig \YY}(\infty)$ can be considered as subsheaves of
$\DD^{\dagger}_{\ZZ}(\infty)$. For example, $\DD^{\dagger}_{\ZZ\rig \YY}(\infty)$ is 
a $\DD^{\dagger}_{\ZZ}(\infty)$-coherent module, and its global sections over 
$\ZZ$ are the sections of $A_2(K)^{\dagger}$ with no term in $\der_x$. 

This is 
the same for $\DD^{\dagger}_{\YY\lrig \ZZ}(\infty)$ once we have twisted the 
two actions of $\DD^{\dagger}_{\ZZ}(\infty)$ over itself by the adjoint operator.

For $M\in D^b_{coh}(\DD^{\dagger}_{\YY}(\infty_{\YY})) $ we state as usual 
$$p_1^!(M)=\DD^{\dagger}_{\ZZ\rig
\YY}(\infty)\ot_{p_1^{-1}\DD^{\dagger}_{\YY}(\infty_{\YY})}^{\bL}p_1^{-1}M \,[1]\;
\in D^b_{coh}(\DD^{\dagger}_{\ZZ}(\infty))$$
and for $N\in D^b_{coh}(\DD^{\dagger}_{\ZZ}(\infty))$
$$p_{2+}(N)=\bR p_{2*}\left(\DD^{\dagger}_{\YY^{\vee}\lrig \ZZ}(\infty)
\ot^{\bL}_{\DD^{\dagger}_{\ZZ}(\infty)}N\right),$$
whose cohomology sheaves are not coherent in general. Fourier transform can now be defined.
\begin{defi}\label{Fouriertransdfn}
For $M\in D^b_{coh}(\DD^{\dagger}_{\YY}(\infty_{\YY}))$ 
$$\FF(M)=p_{2+}(p_1^!M \wtilde{\ot}K_{\pi}).$$
\end{defi}
 Let us recall the fundamental results of \cite{NH}.
\begin{thmH}\label{Fourier}
\begin{itemize}
 \item[i.] (4.3.4 of {\normalfont\cite{NH}}) there is a canonical 
 isomorphism 
 $$\DD^{\dagger}_{\YY^{\vee}}(\infty_{\YY^{\vee}})[-1]\rig
 \FF(\DD^{\dagger}_{\YY}(\infty_{\YY})) ,$$
 \item[ii.] (5.3.1 of {\it loc.\ cit.})
 If $M\in D^b_{coh}(\DD^{\dagger}_{\YY}(\infty_{\YY}))$, then $\FF(M)\in
 D^b_{coh}(\DD^{\dagger}_{\YY^{\vee}}(\infty_{\YY^{\vee}}))$.
\end{itemize}

\end{thmH} 
\subsubsection{}
 To define Fourier transform with compact support, we will need 
 to work with 
 $\DD^{\dagger}_{\ZZ}(\infty')$-modules. In particular, we will use 
 the cohomological functor $p'_{2+}\colon D^b_{coh}(\DD^{\dagger}_{\ZZ}(\infty'))\rig
 D^b_{coh}(\DD^{\dagger}_{\YY^{\vee}}(\infty_{\YY^{\vee}}))$, which preserves coherence, since $p_2$ is proper and
 thanks to the fact that $\infty'=p_2^{-1}(\infty)$. 

 We will also use the scalar restriction functor
$\rho\colon D^b(\DD^{\dagger}_{\ZZ}(\infty))\rig D^b(\DD^{\dagger}_{\ZZ}(\infty'))$. 
\subsubsection{}
For $M$ in $D^b_{coh}(\DD^{\dagger}_{\ZZ}(\infty))$, we denote the dual 
$$\mathbf{D}_{\ZZ}(M)=\mathbf{R}\HH om_{\DD^{\dagger}_{\ZZ}(\infty)}(M,\DD^{\dagger}_{\ZZ}(\infty)
\ot_{\OO_{\ZZ}}\omega^{-1}_{\ZZ}[2])\, \in D^b_{coh}(\DD^{\dagger}_{\ZZ}(\infty)),$$
since the sheaf $\DD^{\dagger}_{\ZZ}(\infty)$
has finite cohomological dimension (\cite{NH2}),
and the corresponding dual 
functor $\mathbf{D}'_{\ZZ}(M)$, for $M\in D^b_{coh}(\DD^{\dagger}_{\ZZ}(\infty'))$, 
(resp. $\mathbf{D}_{\YY}$ and $\mathbf{D}_{\YY^{\vee}}$ for
$\DD^{\dagger}_{\YY}(\infty_{\YY})$, 
resp. $\DD^{\dagger}_{\YY^{\vee}}(\infty_{\YY^{\vee}})$-modules).

The following division lemma will be crucial.

\begin{lemH} \label{division_lemma}Let $\UU=D_+(u_1)\times D_+(v_0)$ or $\UU=D_+(u_1)\times D_+(v_1)$ and $x'=1/x$. 
\begin{itemize}
 \item[i.] The elements $(-x'\der_y+\pi)$ and $-\der_y+\pi x$ generate the same left ideal 
 of $\DD^{\dagger}_{\ZZ}(\infty)(\UU)$.
 \item[ii.] For any $P\in \DD^{\dagger}_{\ZZ}(\infty)(\UU)$ there exists 
 $(Q,R)\in  \DD^{\dagger}_{\ZZ}(\infty)(\UU)\times \DD^{\dagger}_{\ZZ}(\infty')(\UU)$ such that 
 $P=Q(-x'\der_y+\pi)+R.$
 \item[iii.] The previous decomposition is not unique. But, if $Q(-\der_y+\pi x)\in
 \DD^{\dagger}_{\ZZ}(\infty')(\UU)$, 
 then $Q\in \DD^{\dagger}_{\ZZ}(\infty')(\UU)$.
 \item[iv.] If $Q\in\DD^{\dagger}_{\ZZ}(\infty)(\UU)$ and $Q(-\der_y+\pi x)=0$, then $Q=0$.
\end{itemize}
\end{lemH} 
\noindent{\bf Remark}. Analogous statements hold for the left multiplication by $-\der_y+\pi x$, 
or $\der_y+\pi x$, considering right ideals generated by these elements.
\begin{proof}  Recall that $x\in\Ga(\ZZ,\OO_{\ZZ}(^{\dagger}\infty))$. The first assertion comes from the equality
$$ x(-x'\der_y+\pi)=-\der_y+\pi x.$$
The open set $D_+(u_1)\times D_+(v_0)$ will be endowed
with coordinates $x'=1/x$ and $t=y$, and $D_+(u_1)\times D_+(v_1)$ with coordinates $x'$ and $t=1/y$. 
Denote $D^{\dagger}=\DD^{\dagger}_{\ZZ}(\infty)(\UU)$, and 
${D'}^{\dagger}=\DD^{\dagger}_{\ZZ}(\infty')(\UU)$, which contains 
the algebras $\what{F}^{(m)}_{\Qr}$ described in \ref{eq_cat_Pn}. 
 
By considering the right $\OO_{\ZZ}$-module structure on
$\DD^{\dagger}_{\ZZ}(\infty)$, we see that an element $Q\in D^{\dagger}$ 
can be written $$Q=\sum_{\uk\in \Ne^2,\ul\in
\Ze^2}a_{\ul,\uk}
\der_{x'}^{[k_1]}\der_t^{[k_2]}{x'}^{l_1}t^{l_2},$$ such that
\begin{align}\label{cond_conv} a_{\ul,\uk}\in K \textrm{ and }\exists C>0, M>0,  \,|\, v_p(a_{\ul,\uk})\geq
\frac{\max\{-l_1,0\}+\max\{-l_2,0\} +|\uk|}{C}-M,\end{align}
(note that on $D_+(u_1)\times D_+(v_0)$ coefficients $a_{(l_1,l_2),\uk}$ are equal to $0$ if 
$l_2<0$). With this notation $Q\in {D'}^{\dagger}$ if and only if 
$\forall \uk\in\Ne^2,\, \forall \ul\,|\,l_1<0,\, a_{\ul,\uk}=0$.
We can also write $Q$ this way 
$$Q=\sum_{l_1\in\Ze}Q_{l_1}{x'}^{l_1},$$
with $$Q_{l_1}=\sum_{\cindex{\uk=(k_1,k_2)\in \Ne\times \Ne}{l_2\in \Ze}}a_{(l_1,l_2),\uk}
\der_{x'}^{[k_1]}\der_t^{[k_2]}t^{l_2}\in \Ga(\UU,\DD^{\dagger}_{\ZZ}(\infty')).$$
Let us observe that the condition(\ref{cond_conv})
is equivalent to the fact that there exist $m>0$ such that 
$$Q_{l_1}\in \Ga(\UU,\what{\DD}^{(m)}_{\ZZ,\Qr}(\infty'))$$ and 
elements $T_{l_1}$ of $\Ga(\UU,\what{\DD}^{(m)}_{\ZZ}(\infty'))$ such that 
\begin{align}\label{formule_Q} Q_{l_1}=u_{l_1}p^{\num(-l_1)}T_{l_1}, \textrm{ for some } u_{l_1}\in K 
                \textrm{ satisfying }v_p(u_{l_1})\rig +\inft \textrm{ if } |l_1| \rig
+\infty.
\end{align}
We define $$R''=\sum_{l_1\in\mathbf{N}}Q_{l_1}{x'}^{l_1}\,\in
\Ga(\UU,\Dcm_{\ZZ,\Qr}(\infty')),$$ so that we can write down 
\begin{align}\label{formule_Q2}Q=\sum_{l_1<0}u_{l_1}p^{\num(|l_1|)}T_{l_1}x^{|l_1|}+R'',
\end{align}
with $T_{l_1}\in \Ga(\UU,\Dcm_{\ZZ,\Qr}(\infty))$ and $v_p(u_{l_1})\rig +\infty$ if
$|l_1|\rig +\infty$.

We will first prove (iii). 
Suppose that $Q(-\der_y+\pi x)\in\Ga(\UU,\DD^{\dagger}_{\ZZ}(\infty')).$ Because 
of (\ref{formule_Q2}), we can suppose that 
$$ Q=\sum_{l_1<0}Q_{l_1}{x'}^{l_1}.$$
Now we are reduced to prove that $Q=0$, that we can do in restriction 
to $D_+(u_1)\times D_+(v_0v_1)$, on which we choose $x'$ and $y$ as coordinates.

Let us compute
$$Q(-\der_y+\pi x)=\sum_{l_1<0}\left(-Q_{l_1}\der_y+\pi Q_{l_1+1}\right){x'}^{l_1}.$$
Therefore, $Q(-\der_y+\pi x)$ is an element of $\Ga(\UU,\DD^{\dagger}_{\ZZ}(\infty'))$ if and only if
\begin{align} Q_{-1}\der_y=0,\textrm{ and } \end{align}
\begin{align}\label{rec_formula}\forall l_1\leq -1,\, -Q_{l_1}\der_y+\pi Q_{l_1+1}=0.\end{align}

Let us decompose
$$Q_{-1}=\sum_{\uk}\beta_{\uk}(y) \der_{x'}^{[k_1]}\der_y^{[k_2]},$$
with $\beta_{\uk}(y)\in \OO_{\ZZ,\Qr}(D_+(v_0u_0u_1))$, then we compute
$$Q_{-1}\der_y=\sum_{\uk|k_2\geq 1}k_2\beta_{(k_1,k_2-1)}(y) \der_{x'}^{[k_1]}\der_y^{[k_2]},$$
which is null if and only if $\forall \uk, \, \beta_{\uk}=0$, {\it i.e.}\
 $Q_{-1}=0$. Thus by descending induction, one sees from (\ref{rec_formula}) that $\forall l_1\leq
-1,\,Q_{l_1}=0,$ and $Q=0$.

Let us prove now (ii), the existence of the decomposition. Recall 
that $\der_y\in\Ga(\ZZ,\DD^{\dagger}_{\ZZ}(\infty'))$ and 
$(-\der_y+\pi x) \in\Ga(\ZZ,\DD^{\dagger}_{\ZZ}(\infty)) $. 
Since these two elements commute we have the following equalities
\begin{align*} p^{\nu_m(|l_1|)}x^{|l_1|} & =  \frac{p^{\nu_m(|l_1|)}}{\pi^{|l_1|}}\left((-\der_y+\pi x)+\der_y\right)^{|l_1|}\\
                           & =
\frac{p^{\nu_m(|l_1|)}}{\pi^{|l_1|}}|l_1|!\der_y^{[|l_1|]}\\
 & \quad +\left[\sum_{r=1}^{|l_1|}\sum_{s=0}^{r-1}
(-1)^{r-1+s}\frac{p^{\nu_m(|l_1|)}}{\pi^{|l_1|-s}}\binom{|l_1|}{r}\binom{r-1}{s}
(|l_1|-1-s)!\der_y^{[|l_1|-1-s]}x^s\right](-\der_y+\pi x)\\
& =  \frac{p^{\nu_m(|l_1|)}}{\pi^{|l_1|}}\frac{|l_1|!}{q_{|l_1|}^{(m+2)}!}\der_y^{\la|l_1|\ra_{(m+2)}}\\
  &   \quad   +\left[\sum_{r=1}^{|l_1|}\sum_{s=0}^{r-1}
		(-1)^{r-1+s}\frac{p^{\nu_m(|l_1|)}}{\pi^{|l_1|-s}}\binom{|l_1|}{r}\binom{r-1}{s}
		\frac{(|l_1|-1-s)!}{q_{|l_1|-1-s}^{(m+2)}!}\der_y^{\la
|l_1|-1-s\ra_{(m+2)}}x^s\right](-\der_y+\pi x).
\end{align*}
Denote $$R_{l_1}=\frac{p^{\nu_m(|l_1|)}}{\pi^{|l_1|}}\frac{|l_1|!}{q_{|l_1|}^{(m+2)}!}\der_y^{\la
|l_1|\ra_{(m+2)}},$$
$$S_{l_1}=\sum_{r=1}^{|l_1|}\sum_{s=0}^{r-1}
(-1)^{r-1+s}\frac{p^{\nu_m(|l_1|)}}{\pi^{|l_1|-s}}\binom{|l_1|}{r}\binom{r-1}{s}
\frac{(|l_1|-1-s)!}{q_{|l_1|-1-s}^{(m+2)}!}\der_y^{\la |l_1|-1-s\ra_{(m+2)}}x^s,$$
$$c_{l_1}(r,s)=(-1)^{r-1+s}\frac{p^{\nu_m(|l_1|)}}{\pi^{|l_1|-s}}\binom{|l_1|}{r}\binom{r-1}{s}
\frac{(|l_1|-1-s)!}{q_{|l_1|-1-s}^{(m+2)}!}.$$ By definition, we have the following relation
\begin{eqnarray} \label{formuleS}p^{\nu_m(|l_1|)}x^{|l_1|}=S_{l_1}(-\der_y+\pi x)+R_{l_1}.\end{eqnarray}
Then, from estimations \ref{estimations1} we see that
$$v_p\left(\frac{p^{\nu_m(|l_1|)}}{\pi^{|l_1|}}\frac{|l_1|!}{q_{|l_1|}^{(m+2)}!}\right)\geq 
\frac{p^2-p-1}{p^{m+2}(p-1)}|l_1|-\log_p(|l_1|+1)-1 \;\rig\infty \textrm{ if }{|l_1|\rig+\infty},$$
which proves that $R_{l_1}\in \what{E}^{(m+2)}$ for $|l_1|$ big enough. We also see that 
$\forall r\leq |l_1|-1,\,s\leq r,$
\begin{align*}
v_p\left(c_{l_1}(r,s)\right)
&\geq \frac{|l_1|}{p^{m+1}}-\frac{|l_1|-s}{p-1}-\log_p(|l_1|+1)+\frac{|l_1|-1-s}{p-1}-1
-\frac{|l_1|-1-s}{p^{m+2}(p-1)} \\
&\geq \left(\frac{p^2-p-1}{p^{m+2}(p-1)}\right)|l_1|-\log_p(|l_1|+1)-2  \rig\infty \textrm{ if
}{|l_1|\rig+\infty},\end{align*}
which proves that $S_{l_1}\in \what{E}^{(m+2)}$ for $|l_1|$ big enough. As a consequence 
$S_{l_1}$ and resp. $R_{l_1}$ are elements of $\Ga(\UU,\what{\DD}^{(m+2)}_{\ZZ}(\infty))$ for $|l_1|$ big enough
(resp. $\Ga(\UU,\what{\DD}^{(m+2)}_{\ZZ}(\infty'))$).

Let $Q\in D^{\dagger}$. We can use the description given in (\ref{formule_Q2}).
Since $|u_{l_1}|\rig 0$ if $|l_1|\rig +\infty$, we observe 
that 
$$Q'=\sum_{l_1<0}u_{l_1}T_{l_1}S_{l_1}\in \Ga(\UU,\what{\DD}^{(m+2)}_{\ZZ,\Qr}(\infty)),$$
and 
$$ R'=\sum_{l_1<0}u_{l_1}T_{l_1}R_{l_1}\in \Ga(\UU,\what{\DD}^{(m+2)}_{\ZZ,\Qr}(\infty')).$$
Moreover, we have the following equalities
\begin{align*}Q&=\sum_{l_1<0}u_{l_1}p^{\num(|l_1|)}T_{l_1}x^{|l_1|}+R''\\
               &=\sum_{l_1<0}u_{l_1}T_{l_1}\left(S_{l_1}(-\der_y+\pi x)+R_{l_1}\right)+R''\\
               &=Q'x(-x'\der_y+\pi)+R'+R'',
\end{align*}
which shows (ii) of the lemma.

Now, let $Q\in D^{\dagger}$ such that $Q(-\der_y+\pi x)=0$. From (iii), we
know that in fact $Q\in {D'}^{\dagger}$. As in the previous case, 
we restrict ourselves to $\UU=D_+(u_0)\times D_+(v_0v_1)$ and we decompose 
$$Q=\sum_{l_1\in\Ne}Q_{l_1}{x'}^{l_1},$$
The recursive formula (\ref{rec_formula}) still holds and we get
\begin{align*} \pi Q_0=0,\textrm{ and } \end{align*}
\begin{align*}\forall l_1\geq 0,\, \pi
Q_{l_1+1}=Q_{l_1}\der_y.\end{align*}
By induction, this proves that $Q_{l_1}=0$ for all $l_1\geq 0$ and thus that $Q=0$.
\end{proof}

The key lemma to define the Fourier transform with compact support is the following: 
\begin{lemH} \label{coh_lemma} Let $M\in D^b_{coh}(\DD^{\dagger}_{\YY}(\infty_{\YY})))$, then 
$$ \rho_*\mathbf{D}_{\ZZ} (p_1^!M\wtilde{\ot}K_{\pi})\in
D^b_{coh}(\DD^{\dagger}_{\ZZ}(\infty')).$$
\end{lemH} 
\begin{proof} It is enough to prove this lemma in the case of a single
$\DD^{\dagger}_{\YY}(\infty_{\YY})$ coherent module 
$M$. Such a module admits a resolution by direct factors of free modules of finite rank
(\ref{equivP0}). It is thus
enough to prove this lemma in the case of $\DD^{\dagger}_{\YY}(\infty_{\YY})$ itself. 

If $\FF$ is an $\OO_{\ZZ}$-coherent module, we denote
$$\wtilde{\FF}=\FF\ot_{\OO_{\ZZ}}\OO_{\ZZ}(^{\dagger}\infty).$$
Then by 2.1.1 of \cite{NH} $\wtilde{\TT}_{\ZZ/S}$ (resp.
$\wtilde{\TT}_{\ZZ/\YY}$) are free
$\OO_{\ZZ}(^{\dagger}\infty)$-modules of basis $\der_x,\, \der_y$ (resp. $\der_y$).

Let us reformulate lemma 4.2.4 of \cite{NH}. Let $K_{\pg}$ the complex 
of induced $\DD^{\dagger}_{\ZZ}(\infty)$-modules
in degrees $-1$ and $0$
 $$ 0 \rig \DD^{\dagger}_{\ZZ}(\infty)\ot_{\OO^{\dagger}_{\ZZ}(\infty)}\wedgeb^1
\wtilde{\TT}_{\ZZ/\YY}
\sta{d}{\rig} \DD^{\dagger}_{\ZZ}(\infty) \rig 0,$$ where 
$$d(P\ot \der_y)=P\cdot(\der_y+\pi x).$$ 
\begin{souslem}\label{complexL} In the derived category 
$D^b_{coh}(\DD^{\dagger}_{\ZZ}(\infty))$ this complex $K_{\pg}$ is equal to the complex  
$$p_1^!\DD^{\dagger}_{\YY}(\infty_{\YY})\wtilde{\ot}_{\OO^{\dagger}_{\ZZ}(\infty)}K_{\pi}[1],$$
which is nothing but the complex  
$p_1^*\DD^{\dagger}_{\YY}(\infty_{\YY})\ot K_{\pi}$ (in degree
$0$).
\end{souslem} 
The augmentation map $\varep'\colon\DD^{\dagger}_{\ZZ}(\infty)\rig
p_1^*\DD^{\dagger}_{\YY}(\infty_{\YY})
\ot K_{\pi}$ is given by
\begin{eqnarray*}\varep'(\der_y)& = & -\pi x \ot 1 \\
                \varep'(\der_x)& = & (\der_x -\pi y) \ot 1. \end{eqnarray*}
In the rest of the proof, we identify the left induced $\DD^{\dagger}_{\ZZ}(\infty)$-module 
$\DD^{\dagger}_{\ZZ}(\infty)\ot_{\OO^{\dagger}_{\ZZ}(\infty)} \Lambda^1
\wtilde{\TT}_{\ZZ/\YY}$ with $\DD^{\dagger}_{\ZZ}(\infty)$. 
Then $ \mathbf{R}\HH om_{\DD^{\dagger}_{\ZZ}(\infty)}(K_{\pg},\DD^{\dagger}_{\ZZ}(\infty))[2]$
is represented by the following complex of right $\DD^{\dagger}_{\ZZ}(\infty)$-modules, 
whose terms are in degrees $-2$ and $-1$
\begin{align}\label{complex-dual}
 0 \rig \DD^{\dagger}_{\ZZ}(\infty)
\sta{d'}{\rig} \DD^{\dagger}_{\ZZ}(\infty) \rig 0,
\end{align}
such that $d'(P)=(\der_y+\pi x)P$.
 Finally, we see that 
$L_{\pg}=\mathbf{D}(K_{\pg})$ is represented in $D^b_{coh}(\DD^{\dagger}_{\ZZ}(\infty))$
by the following complex in degrees $-2$ et $-1$
\begin{align}  \label{compl_dual} 0 \rig \DD^{\dagger}_{\ZZ}(\infty)
\sta{d''}{\rig} \DD^{\dagger}_{\ZZ}(\infty) \rig 0,\end{align}
such that $d''(P)=P(-\der_y+\pi x)$.

Consider now the canonical map  
$$\DD^{\dagger}_{\ZZ}(\infty')\rig \DD^{\dagger}_{\ZZ}(\infty)\left/ \DD^{\dagger}_{\ZZ}(\infty)
(-\der_y+\pi x)\right. .$$
Over $D_+(u_0)$ both sheaves $\DD^{\dagger}_{\ZZ}(\infty')$ and $\DD^{\dagger}_{\ZZ}(\infty)$ 
coincide. Let us study the situation over $D_+(u_1)$. 
From the previous lemma \ref{division_lemma}, we see that this map is surjective 
over 
$D_+(u_1)\times D_+(v_0)$ and $D_+(u_1)\times D_+(v_1)$ and that over these open subsets 
the following complex is exact (using notations of \ref{division_lemma})
$$0 \rig \DD^{\dagger}_{\ZZ}(\infty') \sta{d''}{\rig} \DD^{\dagger}_{\ZZ}(\infty') \rig 
\rho_* \left(\DD^{\dagger}_{\ZZ}(\infty)\left/\DD^{\dagger}_{\ZZ}(\infty) (-\der_y+\pi
x)\right. \right)\rig 0 ,$$
where $d''(P)=P(-x'\der_y+\pi)$, showing that 
$\rho_* \bD_{\ZZ}(K_{\pg})\in D^b_{coh}(\DD^{\dagger}_{\ZZ}(\infty')),$ hence
$\rho_* \mathbf{D}_{\ZZ}(p_1^!M\wtilde{\ot}K_{\pi})\in
D^b_{coh}(\DD^{\dagger}_{\ZZ}(\infty'))$ for any $M\in D^b_{coh}(\DD^{\dagger}_{\ZZ}(\infty))$.
\end{proof}

Finally, this leads us to the following definition.

\begin{defi}\label{defcompsuppfour}
 For $M\in D^b_{coh}(\DD^{\dagger}_{\YY}(\infty_{\YY}))$ 
$$\FF_{!}(M)=p'_{2+} \mathbf{D}'_{\ZZ} 
\rho_*\mathbf{D}_{\ZZ} (p_1^!M\wtilde{\ot}K_{\pi})\in D^b_{coh}(\DD^{\dagger}_{\ZZ}(\infty')).$$
\end{defi}
Note that from the previous lemma \ref{coh_lemma}, we know that $\FF_{!}(M)\in
D^b_{coh}(\DD^{\dagger}_{\YY^{\vee}}(\infty_{\YY^{\vee}}))$.
\usubsection{Comparison theorem}
\begin{propH} Let $M\in D^b_{coh}(\DD^{\dagger}_{\ZZ}(\infty))$, 
there is a canonical map: $\FF_!(M) \rig \FF(M)$.
\end{propH} 
\begin{proof} Let $E$ be a coherent $\DD^{\dagger}_{\ZZ}(\infty)$-module, such that 
$\rho_*E$ is a coherent $\DD^{\dagger}_{\ZZ}(\infty')$-module.   
There are canonical
maps $\mathbf{D}'_{\ZZ}(\rho_* E)\rig \mathbf{D}_{\ZZ}(E)$ and 
also canonical maps
$$p'_{2+}\mathbf{D}'_{\ZZ}(E)\rig p_{2+}\mathbf{D}_{\ZZ}(E).$$
Applying this to $E=\rho_*\bD_{\ZZ}(\infty)(p_1^!M\wtilde{\ot}K_{\pi})$
gives a canonical map
$$\FF_!(M) \rig p_{2+}\mathbf{D}_{\ZZ}\mathbf{D}_{\ZZ}(p_1^!M\wtilde{\ot}K_{\pi}).$$
And we apply the biduality theorem 3.6 of \cite{Vir}, to see that the RHS can be identified 
with $\FF(M)$.
\end{proof}
\begin{thmH} Let $M\in D^b_{coh}(\DD^{\dagger}_{\ZZ}(\infty))$, 
there is a canonical isomorphism: $\FF_!(M) \simeq \FF(M)$.
\end{thmH} 
Since $\FF$ and $\FF_!$ are both way-out functors, we are reduced to the case of a single 
module $M$. Thanks to \ref{equivP0}, we can suppose that
$M=\DD^{\dagger}_{\YY}(\infty_{\YY})$.

\nnsubsubsection{Computation of $\FF_{!}(\DD^{\dagger}_{\YY}(\infty_{\YY}))$.}
Let us consider $\UU$ one of the open sets $D_+(u_1)\times D_+(v_0)$ or $D_+(u_1)\times D_+(v_1)$ of $\ZZ$, denote by 
$x'$ and $t$ coordinates over $\UU$. Then, over $\UU$, we have the following resolution of the 
sheaf $\DD^{\dagger}_{\YY^{\vee}\lrig \ZZ}(\infty')$ by right
$\DD^{\dagger}_{\ZZ}(\infty')$-modules
$$\xymatrix@R=5pt{0 \ar@{->}[r]& \DD^{\dagger}_{\ZZ}(\infty')\ar@{->}[r] & \DD^{\dagger}_{\ZZ}(\infty')
\ar@{->}[r]&   \DD^{\dagger}_{\YY^{\vee}\lrig \ZZ}(\infty')\ar@{->}[r] & 0\\
   & P \ar@{|->}[r]&\der_{x'} P & & }$$
and exactly the same resolution for the sheaf $\DD^{\dagger}_{\ZZ}(\infty)$, replacing 
$\DD^{\dagger}_{\ZZ}(\infty')$ by $\DD^{\dagger}_{\ZZ}(\infty)$. This proves that there is a canonical
isomorphism of right $\DD^{\dagger}_{\ZZ}(\infty)$-modules
$$\DD^{\dagger}_{\YY^{\vee}\lrig \ZZ}(\infty')
\ot_{\DD^{\dagger}_{\ZZ}(\infty')}\DD^{\dagger}_{\ZZ}(\infty)
\simeq \DD^{\dagger}_{\YY^{\vee}\lrig \ZZ}(\infty).$$
In fact, this isomorphism is also right $\DD^{\dagger}_{\YY^{\vee}}(\infty_{\YY^{\vee}})$-linear,
property that we can check over the open set $D_+(u_0)\times D_+(v_0)$ where the previous isomorphism
coincides with identity.

Now we have the following lemma.
\begin{souslem}\label{p2} Let $M$ be a coherent $\DD^{\dagger}_{\ZZ}(\infty)$-module such that 
$\rho_*M$ is a coherent $\DD^{\dagger}_{\ZZ}(\infty')$-module, then there is a canonical 
isomorphism of coherent $\DD^{\dagger}_{\ZZ}(\infty)$-modules
$$\DD^{\dagger}_{\ZZ}(\infty)\ot_{\DD^{\dagger}_{\ZZ}(\infty')}\rho_*M \simeq M.$$
\end{souslem} 
\begin{proof} The canonical map of the statement is a morphism of coherent 
$\DD^{\dagger}_{\ZZ}(\infty)$-modules, which is an isomorphism over $D_+(u_0)\times D_+(v_0)$ and thus
an isomorphism (cf 4.3.7 of \cite{Ber1}). 
\end{proof}
With the hypothesis of the lemma, we get 
\begin{align*}
p_{2+}M &=Rp_{2*}\left(\DD^{\dagger}_{\YY^{\vee}\lrig \ZZ}(\infty)
\ot_{\DD^{\dagger}_{\ZZ}(\infty)}M\right)\\
        &\simeq Rp_{2*}\left(\DD^{\dagger}_{\YY^{\vee}\lrig \ZZ}(\infty')
\ot_{\DD^{\dagger}_{\ZZ}(\infty')}\DD^{\dagger}_{\ZZ}(\infty)\ot_{\DD^{\dagger}_{\ZZ}(\infty)}M\right)\\
       & \simeq Rp_{2*}\left(\DD^{\dagger}_{\YY^{\vee}\lrig \ZZ}(\infty')
\ot_{\DD^{\dagger}_{\ZZ}(\infty')}\rho_*(M)\right)\\
     &\simeq p'_{2+}M.
\end{align*}
Applying this to $M=K_{\pg}[-1]$ of \ref{complexL}, we see that 
\begin{align*}
\FF_!(\DD^{\dagger}_{\YY}(\infty_{\YY}))&= p'_{2+} \bD_{\ZZ}' \rho_*  \bD_{\ZZ}
(K_{\pg}[-1])\\
&\simeq \bD_{\YY}  p'_{2+}\rho_*  \bD_{\ZZ}(K_{\pg}[-1])\qquad\textrm{(see \cite{Vir2})}\\
&\simeq \bD_{\YY}  p_{2+} \bD_{\ZZ} (K_{\pg}[-1])
\end{align*}
Now we need the following statement
\begin{souslem} There is a canonical isomorphism
$$\DD^{\dagger}_{\YY^{\vee}}(\infty_{\YY^{\vee}})[1]\simeq p_{2+} \bD_{\ZZ} (K_{\pg}).$$
\end{souslem} 
\begin{proof} Using (\ref{compl_dual}), we know that 
$\DD^{\dagger}_{\YY^{\vee}\lrig \ZZ}(\infty)\ot^{\bf L}_{\DD^{\dagger}_{\ZZ}(\infty)}\bD_{\ZZ} (K_{\pg})$
  is represented by the following complex with terms in degrees $-2$ and $-1$
$$\xymatrix@R=5pt{0 \ar@{->}[r]& \DD^{\dagger}_{\YY^{\vee}\lrig \ZZ}(\infty)\ar@{->}[r] & 
\DD^{\dagger}_{\YY^{\vee}\lrig \ZZ}(\infty) \ar@{->}[r]   & 0\\
   & P \ar@{|->}[r]&P(-\der_y+\pi x). &  }$$
Denote by $\EE$ the $-1$-cohomology group of this complex. 
Since the module $\DD^{\dagger}_{\YY^{\vee}\lrig \ZZ}(\infty)$ is a right coherent 
$\DD^{\dagger}_{\ZZ}(\infty)$-module, it is acyclic for the functor $\Gamma(\ZZ,.)$. It is
also the case for $\EE$ by the long cohomology exact sequence and we can then compute 
$\Ga(\ZZ,\EE)$ as the cokernel of the map
$$\xymatrix@R=5pt{0 \ar@{->}[r]& \Ga(\ZZ,\DD^{\dagger}_{\YY^{\vee}\lrig \ZZ}(\infty))\ar@{->}[r] & 
\Ga(\ZZ,\DD^{\dagger}_{\YY^{\vee}\lrig \ZZ}(\infty)) \ar@{->}[r]&    0\\
   & P \ar@{|->}[r]&P(-\der_y+\pi x) &  }.$$
In particular, we get an element $1\in \Ga(\ZZ,\EE)$, allowing us to consider 
a morphism $\varphi\colon\DD^{\dagger}_{\YY^{\vee}}(\infty_{\YY^{\vee}})\rig \bR p_{2*}\EE,$ that 
sends $P$ to $P\cdot 1$, where $1\in R^0p_{2*}\EE$. 

On the other hand, from the previous lemma \ref{p2}, we know that 
$p_{2+}\bD_{\ZZ} (K_{\pg}[-1])$, thus $\EE$, is a coherent
$\DD^{\dagger}_{\YY}(\infty_{\YY})$-module. 
By \ref{equivP1}, 
 it is enough to prove that the morphism induced on global sections of both 
sheaves is an isomorphism, to see that $\varphi$ is an isomorphism.

As $\DD^{\dagger}_{\ZZ}(\infty)$-coherent module, $\DD^{\dagger}_{\YY^{\vee}\lrig \ZZ}(\infty)$
is acyclic for the functor $\Gamma$. Using the resolution given in \ref{acyclicity}, we identify 
$\Ga(\ZZ,\DD^{\dagger}_{\YY^{\vee}\lrig \ZZ}(\infty))$ with 
$A_2(K)^{\dagger}/\der_x A_2(K)^{\dagger}$. Finally, we get the following isomorphisms
\begin{align*}\Ga(\ZZ,\DD^{\dagger}_{\YY^{\vee}\lrig \ZZ}(\infty))\left/
         \Ga(\ZZ,\DD^{\dagger}_{\YY^{\vee}\lrig \ZZ}(\infty))(-\der_y+\pi x)\right.&\simeq 
            A_2(K)^{\dagger}\left/\der_x A_2(K)^{\dagger}+A_2(K)^{\dagger}(-\der_y+\pi
x)\right.\\
&\simeq B_2(K)^{\dagger}\left/\der_x B_2(K)^{\dagger}\right.\\
& \simeq A_1(K)^{\dagger}\\
&\simeq \Ga(\YY^{\vee},\DD^{\dagger}_{\YY^{\vee}}(\infty_{\YY^{\vee}})).
\end{align*}
But $\bR\Gamma(\ZZ,\EE)$ is isomorphic to $\Ga(\ZZ,\EE)$ placed in degree $0$, 
and also to $\bR\Gamma(\ZZ,\bR p_{2*}\EE)$. Because the cohomology sheaves of $R p_{2*}\EE$ are acyclic 
for $\Gamma(\YY^{\vee},.)$, the spectral sequence attached to composite functors
$\Gamma(\YY^{\vee},.)$ and $p_{2*}$ degenerates, proving that $R^ip_{2*}\EE=0$ for $i\neq
0$. Finally the previous computation gives that $p_{2*}\EE$, and thus $p_{2+}\bD_{\ZZ}
(K_{\pg})[-1]$ is isomorphic to $\DD^{\dagger}_{\YY^{\vee}}(\infty_{\YY^{\vee}})$ (in degree $0$)
and the lemma.

We finally get 
\begin{align*}
\FF_!(\DD^{\dagger}_{\YY}(\infty_{\YY}))&\simeq \bD_{\YY}p_{2+}\bD_{\ZZ}(K_{\pg}[-1])\\
                       &\simeq
\bD_{\YY}(\DD^{\dagger}_{\YY^{\vee}}(\infty_{\YY^{\vee}})[2])\\
                       &\simeq \DD^{\dagger}_{\YY^{\vee}}(\infty_{\YY^{\vee}})[-1]\\
                       &\simeq
 \FF(\DD^{\dagger}_{\YY}(\infty_{\YY}))\qquad\textrm{(cf.\ \ref{Fourier})}.
\end{align*}
It remains to show that the canonical map $\FF_{!}(\DD^{\dagger}_{\YY}(\infty_{\YY}))\rig 
\FF(\DD^{\dagger}_{\YY}(\infty_{\YY}))$ maps $1$ to $1$. 
For this we observe that 
the canonical map 
$$p'_{2+}\bD_{\ZZ}'\rho_*\bD_{\ZZ}(K_{\pg})[-1]\rig p_{2+}\bD_{\ZZ}\bD_{\ZZ}(K_{\pg})[-1]$$
maps $1$ (considered as an element of the $0$-cohomology group of these complexes) to $1$,
 which is clear from the 
explicit computation of the complex $K_{\pg}$ and the map of functors 
$p'_{2+}\bD_{\ZZ}'\rig p_{2+}\bD_{\ZZ}$. Then identification 
of $\bD_{\ZZ}\bD_{\ZZ}(K_{\pg})$ also maps $1$ to $1$, and this finally gives us the isomorphism
$$\FF_!(\DD^{\dagger}_{\YY}(\infty_{\YY}))\simeq
\FF(\DD^{\dagger}_{\YY}(\infty_{\YY})).$$
\end{proof}

Tomoyuki Abe:\\
Institute for the Physics and Mathematics of the Universe (IPMU)\\
The University of Tokyo\\
5-1-5 Kashiwanoha, Kashiwa, Chiba, 277-8583, Japan \\
e-mail: {\tt tomoyuki.abe@ipmu.jp}

\bigskip\noindent
Christine Noot-Huyghe:\\
Institut de Recherche Math\'ematique Avanc\'ee\\
UMR 7501, Universit\'e de Strasbourg et CNRS\\
7 rue Ren\'e Descartes, 67000 Strasbourg, France\\
e-mail: {\tt huyghe@math.unistra.fr}
\end{document}